%% file: maksymenko.tex
\documentclass[12pt, reqno]{amsart}

\usepackage[utf8]{inputenc}
\usepackage[T2A]{fontenc}
\usepackage[english]{babel}
\usepackage{amssymb, enumitem}
\usepackage{graphicx}
\usepackage{xcolor}
\usepackage[all]{xy}
\usepackage[noadjust]{cite}
\usepackage[margin=2.5cm]{geometry}
\usepackage{bm}
\usepackage{bbm}

\usepackage{hyperref}
\hypersetup{
   hypertexnames=false,
   colorlinks,
   linkcolor={red!50!black},
   citecolor={blue!50!black},
   urlcolor={blue!80!black}
}

\input{macro.tex}

\author{Sergiy Maksymenko}
\email{maks@imath.kiev.ua}
\address{Institute of Mathematics of NAS of Ukraine, Te\-re\-shchenkivska st. 3, Kyiv, 01004 Ukraine}


\title{Deformations of functions on surfaces}

\begin{document}

\begin{abstract}
The paper contains a review on recent progress in the deformational properties of smooth maps from compact surfaces $M$ to a one-dimensional manifold $P$.
It covers description of homotopy types of stabilizers and orbits of a large class of smooth functions on surfaces obtained by the author, E.~Kudryavtseva, B.~Feshchenko, I.~Kuznietsova, Yu.~Soroka, A.~Kravchenko.
We also present here a new direct proof of the fact that for generic Morse maps the connected components their orbits are homotopy equivalent to finite products of circles.
\end{abstract}

\maketitle

\section{Introduction}\label{sect:intro}
Let $\Mman$ be a compact connected surface and $\Pman$ be either the real line $\bR$ or the circle $S^1$.
For a closed subset $\Xman\subset\Mman$ denote by $\DiffMX$ the group of all smooth ($\Cinfty$) diffeomorphisms of $\Mman$ fixed on $\Xman$.
This group acts from the right on the space $\Ci{\Mman}{\Pman}$ by the following rule: if $\dif\in\DiffMX$ and $\func\in \Ci{\Mman}{\Pman}$, then the result of the action of $\dif$ on $\func$ is the composition map $\func\circ\dif:\Mman\to\Pman$.
For $\func\in \Ci{\Mman}{\Pman}$ let $\fSing$ be the set of its critical points, and
\begin{align*}
\Stabilizer{\func,\Xman} &= \{\dif \in \DiffMX \mid \func \circ \dif = \func \}, \\
\OrbfX &= \{\func \circ \dif \mid \dif \in \DiffMX \}
\end{align*}
be respectively the \textit{stabilizer} and the \textit{orbit} of $\func$ under that action.
It will be convenient to say that elements of $\Stabilizer{\func,\Xman}$ \myemph{preserve} $\func$.
Let also
\[
	\StabilizerIsotId{\func,\Xman} = \Stabilizer{\func} \cap \DiffId(\Mman,\Xman)
\]
be the subgroup of $\Stabilizer{\func}$ consisting of isotopic to the identity diffeomorphisms relatively to $\Xman$, though such an isotopy is not required to preserve $\func$.
Endow these spaces with $C^{\infty}$ topologies and denote by $\DiffIdMX$ and $\StabilizerId{\func,\Xman}$ the corresponding path components of $\id_{\Mman}$ in $\DiffMX$ and $\Stabilizer{\func,\Xman}$, and by $\OrbffX$ the path component of $\OrbfX$ containing $\func$.
We will omit $\Xman$ from notation whenever it is empty.

Let $\Cid{\Mman}{\Pman} \subset \Ci{\Mman}{\Pman}$ be the subset consisting of maps $\func:\Mman\to\Pman$ satisfying the following axiom:
\begin{enumerate}[leftmargin=*, topsep=1ex, parsep=1ex, label={\AxBd}]
\item\label{axiom:bd}
\it
The map $\func$ takes a constant value at each connected component of $\partial\Mman$ and has no critical points in $\partial\Mman$.
\end{enumerate}

Let also $\Mrs{\Mman}{\Pman} \subset \Cid{\Mman}{\Pman}$ be the subset consisting of \myemph{Morse maps}, i.e. maps having only non-degenerate critical points, i.e. in some local coordinates $(x,y)$ at such point $\func$ is given by the formula $\pm x^2\pm y^2$ for some choice of signs.
Notice that such a polynomial can be characterized as a non-zero homogeneous polynomial of degree $2$ having no multiple factors.
A Morse map is called
\begin{itemize}[leftmargin=5ex]
\item \myemph{simple} if each connected component of each level set of $\func$ contains at most one critical point;
\item \myemph{generic} if it takes distinct value at distinct critical points.
\end{itemize}
Every generic Morse map is evidently simple.
Denote by $\MrsGen{\Mman}{\Pman}$ and $\MrsSmp{\Mman}{\Pman}$ the sets of all generic and simple Morse maps $\Mman\to\Pman$.
Then
\[
	\MrsGen{\Mman}{\Pman} \ \subset \ \MrsSmp{\Mman}{\Pman} \ \subset \ \Mrs{\Mman}{\Pman}.
\]
It is well known that each of these three spaces is \myemph{open} and \myemph{everywhere dense} in $\Cid{\Mman}{\Pman}$ with respect to $\Cinfty$ topology, e.g.~\cite[Chapter~6]{Hirsch:DiffTop}.

More generally, let $\FSP{\Mman}{\Pman} \subset \Cid{\Mman}{\Pman}$ be the subset consisting of maps satisfying one more axiom:
\begin{enumerate}[leftmargin=*, topsep=1ex, parsep=1ex, label={\AxCrPt}]
\item\label{axiom:sing}
\it
For every critical point $z$ of $\func$, there are local coordinates in which $\func$ is a homogeneous polynomial $\bR^2\to\bR$ of degree $\geq2$ without multiple factors.
\end{enumerate}

The present paper contains a review of recent results about the homotopy types of $\Stabilizer{\func,\Xman}$ and $\OrbfX$ of maps $\func \in \FSP{\Mman}{\Pman}$ obtained by
S.~Maksymenko, B.~Feshchenko, A.~Kravchenko, I.~Kuznietsova, Yu.~Soroka,
\cite{Maksymenko:AGAG:2006, Maksymenko:hamv2, Maksymenko:BSM:2006, Maksymenko:TrMath:2008, Maksymenko:MFAT:2009, Maksymenko:MFAT:2010, Maksymenko:ProcIM:ENG:2010, Maksymenko:UMZ:ENG:2012, Maksymenko:TA:2020, KravchenkoMaksymenko:PIGC:2018, KravchenkoMaksymenko:EJM:2020, KravchenkoMaksymenko:JMFAG:2020, MaksymenkoFeshchenko:UMZ:ENG:2014, MaksymenkoFeshchenko:MS:2015, MaksymenkoFeshchenko:MFAT:2015, Feshchenko:Zb:2015, MaksymenkoKuznietsova:PIGC:2019, KuznietsovaSoroka:UMJ:2021} and
E.~Kudryavtseva~\cite{KudryavtsevaPermyakov:MatSb:2010, Kudryavtseva:MathNotes:2012, Kudryavtseva:SpecMF:VMU:2012, Kudryavtseva:MatSb:2013, Kudryavtseva:ENG:DAN2016},

Notice that we have the following inclusions:
\[
	\Mrs{\Mman}{\Pman} \ \subset \
	\FSP{\Mman}{\Pman}  \ \subset \
	\Cid{\Mman}{\Pman}  \ \subset \
	\Ci{\Mman}{\Pman}.
\]

It is also easy to show (see~\S\ref{sect:isolated_cr_pt}) that every $\func\in\FSP{\Mman}{\Pman}$ has only isolated critical points, so the set of critical points is finite.
Moreover, for every isolated critical point $z$ of a $\Cr{3}$ map $\gfunc:\bR^2\to\bR$ the local \myemph{topological structure} of level-sets of $\gfunc$ near $z$ is realized by level sets of homogeneous polynomial without multiple factors, see~\S\ref{sect:isolated_cr_pt}.

Thus $\FSP{\Mman}{\Pman}$ consists of ``typical'' maps with ``toplogically typical'' critical points, and therefore the presented results thus describe typical deformational properties of smooth maps on surfaces.

\section{Algebraic preliminaries}\label{sect:algebraic}
In this section we will present necessary topological and algebraic definitions and list of preliminary statements.
The reader may skip this section on first reading and use it for the references.
Everywhere in the paper $\monoArrow$ will mean a ``\myemph{monomorphism}'', $\epiArrow$ an ``\myemph{epimorphism}'', and $\isom$ an \myemph{isomorphism}.

\subsection{Commutative diagrams}
Suppose we are given two sequences of homomorphisms of groups:
\begin{align*}
	&\uSeq: \kA_1 \xrightarrow{\alpha_1} \cdots \xrightarrow{\alpha_{k-1}} \kA_{k}, &
	&\vSeq: \kB_1 \xrightarrow{\beta_1}  \cdots \xrightarrow{\beta_{k-1}} \kB_{k}.
\end{align*}
Then by a \myemph{homomorphism} $\gamma=(\gamma_1,\ldots,\gamma_{k}): \uSeq \to \vSeq$ we will mean a collection of homomorphisms $\gamma_i:\kA_i\to \kB_i$, $i=1,\ldots,k$, making commutative the following diagram:
\[
\aligned
\xymatrix@R=1.1em{
\kA_1 \ar[r]^-{\alpha_1}  \ar[d]^-{\gamma_1} &
\kA_2 \ar[r]^-{\alpha_2}  \ar[d]^-{\gamma_2} &
\cdots \ar[r]^-{\alpha_{k-1}} & \kA_{k} \ar[d]^-{\gamma_k} \\
\kB_1 \ar[r]^-{\beta_1}   &
\kB_2 \ar[r]^-{\beta_2}   &
\cdots \ar[r]^-{\beta_{k-1}} & \kB_{k}
}
\endaligned
\]
In this case $\gamma$ is said to be an \myemph{epimorphism} (resp.\ \myemph{monomorphism}, \myemph{isomorphism}) whenever $\gamma_i$ is so.
In particular, one can talk about \myemph{exact sequences} of sequences of homomorphisms.
Also by a \myemph{product} $\uSeq\times\vSeq$ we will mean the following sequence
\[
	\uSeq\times\vSeq:
		\kA_1 \times \kB_1
			\xrightarrow{\alpha_1 \times \beta_1}
		\kA_2 \times \kB_2
			\xrightarrow{\alpha_2 \times \beta_2}
            \cdots
			\xrightarrow{\alpha_{k-1}\times \beta_{k-1}}
		\kA_{k} \times \kB_{k}.
\]
More generally, one can define in an obvious way similar notions for commutative diagrams of arbitrary fixed type not only for chains of homomorphisms.
In particular, by \myemph{exact $(3\times3)$-diagram} we will mean a commutative diagram shown on the left:
\begin{equation}\label{equ:3x3_diagram}
\aligned
\xymatrix@C=1.2em@R=1em{
	\uSeq_0: &
	K \ar@{^(->}[d] \ar@{^(->}[r]  &
	L \ar@{^(->}[d]  \ar@{->>}[r]  &
	M \ar@{^(->}[d] \\
	\uSeq_1: &
	A \ar@{^(->}[r] \ar@{->>}[d] &
	B \ar@{->>}[r] \ar@{->>}[d] &
	C \ar@{->>}[d] \\
	\uSeq_2: &
	P \ar@{^(->}[r] &
	Q \ar@{->>}[r]  &
	R
}
\endaligned
\end{equation}
in which each row and column is a short exact sequence.
It can be regarded as a \myemph{short exact sequence of its rows $\uSeq_0 \monoArrow \uSeq_1 \epiArrow \uSeq_2$ (or columns) being in turn short exact sequences of groups homomorphisms}.
Notice that if all monomorphisms in~\eqref{equ:3x3_diagram} are inclusions of subgroups, then $L$ and $A$ are normal subgroups of $B$ with $K=L\cap A$ and that diagrams is isomorphic with the following one
\begin{equation}\label{equ:iso_3x3_diagrams}
\aligned
	\xymatrix@C=1.5em@R=1em{
	A\cap L \ar@{^(->}[d] \ar@{^(->}[rr]  &&
	L \ar@{^(->}[d]  \ar@{->>}[r]  &
	\frac{L}{A\cap L} \ar@{^(->}[d] \\
	A \ar@{^(->}[rr] \ar@{->>}[d] &&
	B \ar@{->>}[r] \ar@{->>}[d] &
	B/A \ar@{->>}[d] \\
	\frac{A}{A\cap L} \ar@{^(->}[rr] &&
	B/L \ar@{->>}[r]  &
	\frac{B/L}{A/(A\cap L)} \cong \frac{B/A}{M/(A\cap L)}
}
\endaligned
\end{equation}
via isomorphism being identity on $A$, $L$, and $B$.

\subsection{Sections of homomorphisms}
Let $p:G \to Z$ be a homomorphism.
Then another homomorphism $s:Z\to G$ is called a \myemph{section} of $p$, whenever $p\circ s= \id_{Z}$.
In this case $p$ must be surjective, and $s$ isomorphically maps $Z$ onto the subgroup $s(Z)$ of $G$.

An essential point here is that $s$ must be a \myemph{homomorphism}.
For instance $p:\bZ \xepiArrow{b \bmod n} \bZ_n$ for $n\geq2$ has no sections.
Indeed if $s:\bZ_n \to \bZ$ a section, then $s(\bZ_n)$ must be a finite subgroup of $\bZ$, and therefore it is $\{0\}$.
Thus $s$ is the zero homomorphism, and thus $p\circ s = 0 \not = \id_{\bZ_n}$.

\subsection{Direct products}
Let $G$ be a group, $G_1,\ldots,G_k$ a collection of its subgroups, and $\psi:\mprod\limits_{i=1}^{k} G_i \to G$ a \myemph{map} defined by $\psi(g_1,\ldots,g_k) = g_1\cdots g_k$.
Then $\psi$ is an \myemph{isomorphism of groups} if and only if $G_1,\ldots,G_k$ pairwise commute, generate $G$, and $G_i\cap G_j = \{e\}$ for $i\not=j$.
In that case $G$ \myemph{splits into a direct product of its subgroups $G_1,\ldots,G_k$}

\newcommand\lact[2]{{}^{#1}{#2}}
\newcommand\eA{e_A}
\newcommand\eZ{e_Z}

\subsection{Semidirect products}
Let $Z$ and $A$ be two groups with units $\eZ$ and $\eA$ respectively, and $\phi:Z \to \Aut(A)$ a \myemph{homomorphism} into the group of automorphisms of $A$ regarded as a group with respect to the composition of automorphisms.
For $x\in Z$ and $a\in A$ it will be convenient to denote $\phi(x)(a) = \lact{x}{a}$, whence for $x,y\in Z$ and $a,b\in A$ we have that
\begin{align*}
	\lact{(xy)}{a} &= \phi(xy)(a) = \phi(x)\bigl(\phi(y)(a)\bigr) = \lact{x}{(\lact{y}{a})}, \\
	\lact{x}{(ab)}   &= \phi(x)(ab) = \phi(x)(a) \cdot \phi(x)(b) = \lact{x}{a} \cdot   \lact{x}{b}.
\end{align*}
Then there is a group structure on the Cartesian product of \myemph{sets} $A\times Z$, denoted by $A\rtimes_{\phi} Z$ and called \myemph{a semidirect product of $A$ and $Z$ (with respect to $\phi$)}, defined by the following rule:
\[
	(a,x) \cdot (b,y) := (a\cdot  \lact{x}{b}, xy), \qquad (a,x), (b,y)\in A\times Z.
\]
If $\phi$ is assumed from the context, then $A\rtimes_{\phi} Z$ is sometimes denoted simply by $A\rtimes Z$.

One easily check associativity of such multiplication, and that
\[
	(a,x) \cdot (b,y) \cdot (c,z) := (a\cdot  \lact{x}{b} \cdot \lact{xy}{c}, xyz),
\]
the unit is $(\eA, \eZ)$, and $(a,x)^{-1} = (\lact{x^{-1}}{a^{-1}}, x^{-1})$.

It follows that the following maps
\begin{align*}
	& \rho: A \to A\rtimes_{\phi} Z, && \rho(a) = (a,\eZ), \\
	& \sigma: Z \to A\rtimes_{\phi} Z, && \sigma(z) = (\eA,z), \\
	& \pi: A\rtimes_{\phi} Z \to Z,    && \pi(a,z) = z,
\end{align*}
are homomorphisms, $\rho$ isomorphically maps $A$ on $A\times\eZ$, and $\sigma$ is a \myemph{section} of $\pi$, i.e. $\pi\circ \sigma = \id_{Z}$, and its isomorphically maps $Z$ onto $\eA\times Z$.
In other words, we have the following short exact sequence admitting a section:
$
	\xymatrix{
		 A \  \ar@{^{(}->}[r]^-{\rho} &
		\ A  \rtimes Z \ \ar@{->>}[r]^-{\pi} &
		\ Z. \ar@/^1.6ex/[l]^-{\sigma} 
	}
$

The following statement characterizes semidirect products via short exact sequences admitting sections.
\begin{sublemma}\label{lm:short_ex_seq_sections}
Let $p:G \to Z$ be a homomorphism with $A = \ker(p)$.
Suppose there exists a \myemph{section $s:Z\to G$} of $p$.
Then $s(Z)$ acts on $A$ by conjugations, so we get a homomorphism
\[ \phi:Z \to \Aut(A), \qquad \phi(z)(a) = s(z) \cdot a \cdot s(z)^{-1},\]
and can define the semidirect product $A\rtimes_{\phi} Z$.
Moreover, the map
\[ \psi:A\rtimes_{\phi} Z \to G, \qquad \psi(a,z) = r(a)\sigma(z),\]
is an \myemph{isomorphism} which induces isomorphism of the following short exact sequences:
\begin{equation}\label{equ:semidirect_prod}
	\aligned
	\xymatrix@R=3ex{
	\ A            \           \ar@{=}[d] \ar@{^{(}->}[rr]^-{a \mapsto (a,e)} &&
	\  A\rtimes_{\phi} Z \ \ar[d]^-{\psi}_-{\cong} \ar@{->>}[rr]^-{(a,z) \mapsto z \ \ } &&
	\ Z            \ \ar@{=}[d]  \\
	\ A            \ \ar@{^{(}->}[rr] && \ G  \  \ar@{->>}[rr]^-{p} && \ Z \
	}
	\endaligned
\end{equation}
\end{sublemma}

\newcommand\wrGmZ{\nwr{G}{m}{\bZ}}
\newcommand\wrGZm{\ewr{G}{\bZ_m}}
\newcommand\wrGmnZZ{\nwr{G}{m,n}{\bZ^2}}
\newcommand\wrGZmZn{\ewr{G}{(\bZ_{m}\times\bZ_{n})}}

\subsection{Semidirect products with $\bZ$}
Clearly, any epimorphism $\eta:G \to \bZ$ onto the group $\bZ$ of integers, always has a section: just take any $g\in G$ with $\eta(g)=1$, and put $s(n) = g^n$.
Then $G$ is a semidirect product $\ker(\eta) \rtimes \bZ$, i.e. a Cartesian product $\ker(\eta) \times \bZ$ of sets with the following multiplication:
\begin{equation}\label{equ:prod_in_GrtimesZ}
(a,m)(b,n) = \bigl(a g^{m} b g^{-m}, \ m+n\bigr).
\end{equation}

\subsection{Wreath products}
Let ${G}$ be a group, $m \geq 1$, and ${G}^m$ be the $m$-th power of ${G}$, so its elements are $m$-tuples of elements of ${G}$.
Also for $n\geq1$ one can regard the elements of $mn$-power ${G}^{mn}$ of ${G}$ as $(m\times n)$-matrices whose entries belong to ${G}$.
In particular, there are the following natural \myemph{non-effective} actions of $\bZ$ on ${G}^m$ and $\bZ^2$ on ${G}^{mn}$ by cyclic shifts of coordinates:
\begin{align*}
&{G}^m \times \bZ\to {G}^m,         &&  \bigl( \{ g_i\}_{i=0}^{m-1}, a \bigr) \mapsto \{ g_{i+a}\}, \\
&{G}^{mn} \times \bZ^2\to {G}^{mn}, &&  \bigl( \{ g_{i,j}\}_{i=0,\ldots,m-1}^{j=0,\ldots, n-1}, (a,b) \bigr) \mapsto \{ g_{i+a, j+b}\},
\end{align*}
for $a,b\in\bZ$, where $i+a$ is taken modulo $m$, and $j+b$ is taken modulo $n$.
These actions reduce to \myemph{effective} actions of $\bZ_m$ on ${G}^{m}$ and $\bZ_m\times\bZ_n$ on ${G}^{mn}$.
Let
\begin{align*}
&\wrGmZ, &
&\ewr{G}{\bZ_m}, &
&\nwr{G}{mn}{\bZ^2}, &
&\ewr{G}{(\bZ_{m}\times\bZ_{n})}
\end{align*}
be the respective semidirect products induced by the above actions.
Such semidirect products are called \myemph{wreath products}.
For instance, $\wrGmZ$ is a cartesian product of sets $G^m \times \bZ$ with the following operation:
\begin{multline*}
(g_0,\ldots,g_{m-1}, a) \cdot (h_0,\ldots,h_{m-1}, b) = \\
 =  \bigl( g_0 h_a, \, g_{1} h_{a+1}, \, \ldots, \, g_{m-1} h_{a-1}, \, a+b \bigr),
\end{multline*}
where all indices are taken modulo $m$.
The multiplications in other groups $\ewr{G}{\bZ_m}$, $\nwr{G}{mn}{\bZ^2}$, and $\ewr{G}{(\bZ_{m}\times\bZ_{n})}$ are similar.

These groups will play the key role in what follows.
For a group $G$ denote by $Z(G)$ its center and by $G':=[G,G]$ its derived subgroup (commutant).
Let also $\ab:G\to G/G'$ be the natural abelianization epimorphism.

\begin{sublemma}{\rm\cite{KuznietsovaSoroka:UMJ:2021}}\label{lm:center_comm}
Let $G$ be any group and $m,n\geq1$.
Then
\begin{gather*}
Z\bigl( \wrGmZ \bigr) =
\bigl\{ (g,\ldots,g, mk) \mid g\in G, k\in\bZ \bigr\} \ \cong \ Z(G)\times \bZ, \\
Z\bigl( \nwr{G}{m,n}{\bZ^2} \bigr) =
\bigl\{ (g,\ldots,g, mk, nl) \mid g\in G, k,l\in\bZ \bigr\} \ \cong \ Z(G)\times \bZ^2.
\end{gather*}
Moreover, the following maps
\begin{align*}
	&\gamma: \wrGmZ   \to (G/G') \times \bZ,   && \gamma(g_1,\ldots,g_n, k)   = \bigl(\ab(g_1\cdots g_m), k), \\
	&\delta: \wrGmnZZ \to (G/G') \times \bZ^2, && \!\gamma(\{ g_{i,j} \}, k, l) = \Bigl(\ab\bigl( \mprod_{i=1}^{m} \mprod_{j=1}^{n} g_{i,j} \bigr), k, l\Bigr)
\end{align*}
are well-defined surjective homomorphisms with
\begin{align*}
	\bigl(\wrGmZ \bigr)' &= \ker(\gamma) = \bigl\{  (g_1,\ldots,g_n, 0) \mid \mprod_{i=1}^{m} g_i\in G' \bigr\}, \\
	\bigl(\wrGmnZZ \bigr)' &= \ker(\delta) = \bigl\{  (\{ g_{i,j} \}, 0, 0)  \mid \mprod_{i=1}^{m} \mprod_{j=1}^{n} g_{i,j} \in G' \bigr\},
\end{align*}
so we have the following commutative diagrams:
\begin{align*}
&\xymatrix@R=4ex{
	\wrGmZ  \ar[r]^-{\ab} \ar@{=}[d] &
	\frac{\wrGmZ}{\bigl(\wrGmZ\bigr)'}  \ar[d]^-{\eta}_-{\cong} \\
	 \wrGmZ  \ar[r]^-{\gamma}  &
	 G/G' \times \bZ
}&\qquad
\xymatrix@R=4ex{
	\wrGmnZZ  \ar[r]^-{\ab} \ar@{=}[d] &
	\frac{\wrGmnZZ}{\bigl(\wrGmnZZ\bigr)'}  \ar[d]^-{\nu}_-{\cong} \\
	\wrGmnZZ  \ar[r]^-{\delta}  &
	 G/G' \times \bZ^2
}
\end{align*}
for unique isomorphisms $\eta$ and $\nu$.

Finally, if $G$ is \myemph{torsion free}, then so are $\wrGmZ$ and $G\wrm{m,n}\bZ$.
\end{sublemma}
\begin{smallproof}{Notes to the proof}
Statements about centers, derived subgroups, and abelianization are proved in~I.~Kuznietsova and Yu.~Soroka~\cite{KuznietsovaSoroka:UMJ:2021}.
The latter statement about torsion free property of $\wrGmZ$ is established in~\cite[Lemma~2.2]{Maksymenko:TA:2020}, and the same arguments can be used to prove it for $\wrGmnZZ$.
\end{smallproof}

\subsection{Special short exact sequences}
In what follows for $m,n\geq1$ we will use the following short exact sequences:
\begin{align}
\label{equ:ex_seq:z0}  \seqTriv   & : \ \{1\}\monoArrow \{1\} \epiArrow \{1\}, \\
\label{equ:ex_seq:z1}  \seqZ{1}   & : \ \bZ\xmonoArrow{\id} \bZ \epiArrow \{1\}, \\
\label{equ:ex_seq:zm}  \seqZ{m}   & : \ m\bZ\xmonoArrow{~~} \bZ \xepiArrow{~\bmod m~} \bZ_m, \\
\label{equ:ex_seq:zmn} \seqZ{m,n} & : \ m\bZ \times n\bZ \xmonoArrow{~~} \bZ\times \bZ \xepiArrow{~(\mathrm{mod}\, m, \ \mathrm{mod}\, n)~} \bZ_m \times \bZ_n,
\end{align}
Then for a short exact sequence $\aSeq: \kA\xmonoArrow{\alpha} \kB \xepiArrow{~\beta~} \kC$ we have the following two exact $(3\times3)$-diagram:

\newcommand\asp{\hspace{-0.5cm}}%
\begin{equation}\label{equ:wr_ex_seq_qzm}
\aligned
\xymatrix@C=1.8em@R=1.2em{
\ \aSeq^{m}: \ar@{^(->}[d] &\asp
\kA^m\times0 \ar@{^(->}[d] \ar@{^(->}[r]  &
\kB^m\times0 \ar@{^(->}[d]  \ar@{->>}[r]  &
\kC^m\times0 \ar@{^(->}[d] \\
\ \seqWrm{\aSeq}{m}: \ar@{->>}[d] &\asp
\kA^m\times m\bZ \ \ar@{^(->}[r]^-{\alpha'} \ar@{->>}[d] &
\kB\wrm{m}\bZ \ar@{->>}[r]^-{\beta'} \ar@{->>}[d]^-{\,p'} &
\kC\wr\bZ_m \ar@{->>}[d] \\
\ \seqZ{m}: &\asp
m\bZ \ar@{^(->}[r] &
\bZ \ar@{->>}[r]  &
\bZ_m
}
\endaligned
\end{equation}%
\begin{equation}\label{equ:wr_ex_seq_qzmn}
\aligned
\MResize{0.85\textwidth}{
	\xymatrix@C=1.8em@R=1.5em{
	\ \aSeq^{mn}: \ar@{^(->}[d] &\asp
	\ \kA^{mn}\times 0\times 0 \ \ar@{^(->}[d] \ar@{^(->}[r]  &
	\kB^{mn}\times 0\times 0 \ar@{^(->}[d]  \ar@{->>}[r]  &
	\kC^{mn}\times 0\times 0 \ar@{^(->}[d] \\
	\ \seqWrm{\aSeq}{m,n}: \ar@{->>}[d] &\asp
	\kA^{mn}\times m\bZ \times n\bZ \ \ar@{^(->}[r]^-{\alpha''} \ar@{->>}[d] &
	\kB\wrm{m,n}\bZ^2 \ar@{->>}[r]^-{\beta''} \ar@{->>}[d]^-{\,p''} &
	\kC\wr(\bZ_m\times\bZ_n) \ar@{->>}[d] \\
	\ \seqZ{m,n}: &\asp
	\ m\bZ \times n\bZ \ \ar@{^(->}[r] &
	\ \bZ^2 \            \ar@{->>}[r]  &
	\ \bZ_m\times\bZ_n \
	}
}
\endaligned
\end{equation}
where $p'$ and $p''$ are the projection to the last coordinates, and
\begin{align*}
\alpha'\bigl(a_1,\ldots,a_m, mk\bigr)     &= \bigl(\alpha(a_1),\ldots,\alpha(a_m), mk \bigr), \\
\beta'\bigl(b_1,\ldots,b_m, k\bigr)       &= \bigl(\beta(b_1),\ldots,\beta(b_m), \MOD{k}{m} \bigr) \\
\alpha''\bigl(\{ a_{i,j}\},  mk, nl\bigr) &= \bigl( \{ \alpha(a_{i,j}) \}, mk, nl \bigr), \\
\beta''\bigl( \{ b_{i,j} \}, k, l\bigr)    &= \bigl(\{ \beta(b_{i,j}) \}, \MOD{k}{m}, \MOD{l}{n} \bigr),
\end{align*}
$\alpha_i\in \kA$, $\beta_i\in \kB$ for all $i=1,\ldots,m$, $j=1,\ldots,n$, and $k,l\in\bZ$.
The middle horizontal sequences of both diagrams will be called the \myemph{wreath product of $\aSeq$ with $\seqZ{m}$ and $\seqZ{m,n}$} and denoted by $\seqWrm{\aSeq}{m}$ and $\seqWrmn{\aSeq}{m}{n}$ respectively.
Thus~\eqref{equ:wr_ex_seq_qzm} and~\eqref{equ:wr_ex_seq_qzmn} can be written as short exact sequences of their rows:
\begin{align}\label{equ:seq_qzm_qzmn}
	&\aSeq^{m}\monoArrow\seqWrm{\aSeq}{m}\epiArrow\seqZ{m}, &
	&\aSeq^{mn}\monoArrow\seqWrm{\aSeq}{m,n}\epiArrow\seqZ{m,n}
\end{align}
Evidently,
\begin{gather*}\label{equ:W_Z1}
\seqWrmn{\aSeq}{m}{1} \ \cong \ \seqWrmn{\aSeq}{1}{m}, \\
\seqWrm{\aSeq}{1} \ \cong \ \aSeq \times \seqZ{1} : \ \
\kA \times \bZ \ \monoArrow \ \kB \wrm{1} \bZ\equiv \kB\times\bZ \ \epiArrow \ \kC, \\
\label{equ:Zm_Zn}
\seqWrm{\seqZ{k}}{m}: \ \ (k\bZ)^m \times m\bZ \ \xmonoArrow{~} \ \bZ\wrm{m}\bZ \ \xepiArrow{~} \  \bZ_k\wr\bZ_m, \\
\seqWrmn{\seqZ{k}}{m}{n}: \ (k\bZ)^{mn} \times m\bZ\times n\bZ \ \xmonoArrow{~} \ \bZ\wrm{m,n}\bZ^2 \ \xepiArrow{~} \  \bZ_k\wr(\bZ_m\times\bZ_n).
\end{gather*}

\newcommand\dwSeq[1]{{}^{#1}\!/_{\!#1}}
\newcommand\ssSeq[1]{\mathsf{s}(#1)}
\newcommand\ddSeq[1]{\mathsf{d}(#1)}
\newcommand\qwSeq[1]{\natural#1}

\newcommand\grs{\gamma}
\newcommand\grsi[1]{\grs_{#1}}
\subsection{Garside elements}\label{sect:garside_elem}
Let $G$ be any group with unit $e$ and $m\geq1$.
Then the element $\grs=(\underbrace{e,\ldots,e}_{m}, m) \in \wrGmZ$ will be called the \myemph{Garside} element.
By Lemma~\ref{lm:center_comm} $\grs$ belongs to the center of $\wrGmZ$.


\subsubsection{Garside sequences}
Let $\aSeq:\kA \monoArrow \kB \epiArrow \kC$ be a short exact sequence and $m\geq1$.
For simplicity assume that $\kA$ is a subgroup of $\kB$ and let $e$ be the common unit element of $\kA$ and $\kB$.
Consider the wreath product
\begin{equation}\label{equ:q_wr_zm}
	\seqWrm{\aSeq}{m}: \kA^m\times m\bZ  \monoArrow \kB \wrm{m} \bZ \epiArrow \kC \wr \bZ.
\end{equation}
Then the Garside element $\grs = (e,\ldots,e, m)$ of $\kB \wrm{m} \bZ$ belongs to $\kA^m\times m\bZ$, and $\grs^k = (e,\ldots,e, km)$ for all $k\in\bZ$.
It follows that there exists the following exact $(3\times3)$-diagram:
\begin{equation}\label{equ:seqWrZm_finite}
\aligned
\xymatrix@C=1.8em@R=1.4em{
	\seqZ{1} \ar@{^(->}[d]
	& \hspace{-1cm} :  &
	e^m\times m\bZ  \ar@{^(->}[d]^-{\,k \,\mapsto\, \grs^k} \ar@{=}[r]  &
	e^m\times m\bZ  \ar@{^(->}[d] \ar@{->>}[r]  &
	1 \ar@{^(->}[d] \\
	\seqWrm{\aSeq}{m} \ar@{->>}[d]
	& \hspace{-1cm}:   &
	\kA^m\times m\bZ  \ar@{^(->}[r] \ar@{->>}[d] &
	\kB\wrm{m}\bZ \ar@{->>}[r] \ar@{->>}[d] &
	\kC\wr\bZ_m \ar@{=}[d] \\
	\seqWrm{\aSeq}{m}
	& \hspace{-1cm} : &
	\kA^m  \ar@{^(->}[r]  &
	\kB\wr\bZ_m \ar@{->>}[r]  &
	\kC\wr\bZ_m
	}
\endaligned
\end{equation}
in which the bottom row sequence will be denoted by $\seqWrm{\aSeq}{m}$. 
The total sequence $\seqZ{1} \monoArrow \seqWrm{\aSeq}{m} \epiArrow \seqWrm{\aSeq}{m}$ will be called the \myemph{Garside sequence of $\seqWrm{\aSeq}{m}$}.

Similarly, for $m,n\geq1$ one has analogous diagram
\begin{equation}\label{equ:seqWrZmn_finite}
\MResize{0.85\textwidth}{
\aligned
\xymatrix@C=1.8em@R=1.4em{
	\seqZ{1}^2 \ar@{^(->}[d]
	& \hspace{-1cm} : \hspace{-0.6cm} &
	e^{mn}\times m\bZ\times n\bZ  \ar@{^(->}[d]^-{\,k \,\mapsto\, \grs^k} \ar@{=}[r]  &
	e^{mn}\times m\bZ\times n\bZ  \ar@{^(->}[d] \ar@{->>}[r]  &
	1 \ar@{^(->}[d] \\
	\seqWrmn{\aSeq}{m}{n} \ar@{->>}[d]
	& \hspace{-1cm}: \hspace{-0.6cm}  &
	\kA^{mn}\times m\bZ \times n\bZ  \ar@{^(->}[r] \ar@{->>}[d] &
	\kB\wrm{m,n}\bZ^2 \ar@{->>}[r] \ar@{->>}[d] &
	\kC\wr(\bZ_m\times\bZ_n) \ar@{=}[d] \\
	\seqWrmn{\aSeq}{m}{n}
	& \hspace{-1cm} : \hspace{-0.6cm} &
	\kA^{mn}  \ar@{^(->}[r]  &
	\kB\wr(\bZ_m\times\bZ_n) \ar@{->>}[r]  &
	\kC\wr(\bZ_m\times\bZ_n)
	}
\endaligned
}
\end{equation}
The total sequence $\seqZ{1}^2 \monoArrow \seqWrmn{\aSeq}{m}{n} \epiArrow \seqWrmn{\aSeq}{m}{n}$ will be called the \myemph{Garside sequence of $\seqWrmn{\aSeq}{m}{n}$}.

\subsubsection{Diagonal Garside sequences}
More generally, let $n\geq1$, and for each $i=1,\ldots,n$ let $\aSeq_i:\kA_i \monoArrow \kB_i \epiArrow \kC_i$, be a short exact sequence, $m_i\geq1$, $\grsi{i} \in \kB_i \wrm{m_i} \bZ$ be the corresponding Garside element, and
\[ \widehat{\grs} = (\grsi{1}, \ldots,\grsi{n}) \in \mprod\limits_{i=1}^{n} (\kA_i^{m_i}\times m_i\bZ)\] be the \myemph{``diagonal'' element} in the product of subgroups $\kA_i^{m_i}\times m_i\bZ$ generated by the corresponding Garside elements.
Then we have the following commutative diagram:
\begin{equation}\label{equ:diag_Garside_seq}
\aligned
\MResize{0.89\textwidth}{\xymatrix@C=1em@R=1.4em{
	\seqZ{1} \ar@{^(->}[d]
	& \hspace{-1cm} :  &
	\bZ  \ar@{^(->}[d]^-{k \,\mapsto\, \langle \widehat{\grs} \rangle^k = (\grsi{1}^k, \ldots, \grsi{n}^k)} \ar@{=}[r]  &
	\bZ  \ar@{^(->}[d] \ar@{->>}[r]  &
	1 \ar@{^(->}[d] \\
	\mprod\limits_{i=1}^{n}\seqWrm{\aSeq}{m} \ar@{->>}[d]
	& \hspace{-1cm}:   &
	\mprod\limits_{i=1}^{n} (\kA_i^{m_i}\times m_i\bZ)  \ar@{^(->}[r] \ar@{->>}[d] &
	\mprod\limits_{i=1}^{n} (\kB_i\wrm{m_i}\bZ)  \ar@{->>}[r] \ar@{->>}[d] &
	\mprod\limits_{i=1}^{n} (\kC_i \wr\bZ_{m_i}) \ar@{=}[d] \\
	(\mprod\limits_{i=1}^{n}\seqWrm{\aSeq}{m}) / \seqZ{1}
	& \hspace{-1cm} :  &
	\Bigl( \mprod\limits_{i=1}^{n} (\kA_i^{m_i}\times m_i\bZ) \Bigr) / \langle \widehat{\grs} \rangle   \ar@{^(->}[r]  &
	\mprod\limits_{i=1}^{n} (\kB_i\wrm{m_i}\bZ) / \langle \widehat{\grs} \rangle  \ar@{->>}[r]  &
	\mprod\limits_{i=1}^{n} (\kC_i \wr\bZ_{m_i})
	}
}
\endaligned
\end{equation}
which will be called \myemph{diagonal Garside sequence of $\mprod\limits_{i=1}^{n}\seqWrm{\aSeq_i}{m_i}$}.

\subsubsection{Several constructions}
For every group $A$ one can associate two short exact sequences $\ssSeq{A}: A = A \epiArrow 1$ and $\ddSeq{A}: 1 \monoArrow A = A$.

Then for a pair $\kSeq:A \monoArrow B \epiArrow C$ and $\lSeq:P \monoArrow Q \epiArrow R$ of two exact sequences one can also define the following \myemph{split sequence of $\kSeq$ and $\lSeq$:}
\begin{gather*}\label{equ:split_seq}
 \ssSeq{\kSeq} \times \ddSeq{\lSeq}:  \kSeq \monoArrow \kSeq\times \lSeq \epiArrow \lSeq \\
    \xymatrix@C=4ex@R=3ex{
        \kSeq \ar@{^(->}[d]             & \hspace{-1.2cm} : & A \ar@{^(->}[r] \ar@{^(->}[d]         & B \ar@{^(->}[d] \ar@{->>}[r]         & C \ar@{^(->}[d] \\
        \kSeq\times \lSeq \ar@{->>}[d]  & \hspace{-1.2cm} : & A\times P \ar@{^(->}[r] \ar@{->>}[d]  & B\times Q \ar@{->>}[r] \ar@{->>}[d]  & C\times R \ar@{->>}[d]   \\
        \lSeq                           & \hspace{-1.2cm} : & P \ar@{^(->}[r]                       & Q \ar@{=}[r]                         & R
    }
\end{gather*}
Also for each short exact sequence $\uSeq: A \monoArrow B \epiArrow C$ one can define the following exact $(3\times3)$-diagram

\begin{equation}\label{equ:u_div_u}
\aligned
    \xymatrix@C=4ex@R=2ex{
      \ssSeq{A} \ar@{^(->}[d] & \hspace{-1.8cm} : \hspace{-1cm} & A \ar@{=}[r]    \ar@{=}[d]    & A \ar@{->>}[r] \ar@{^(->}[d] & 1 \ar@{^(->}[d] \\
      \uSeq     \ar@{->>}[d]  & \hspace{-1.8cm} : \hspace{-1cm} & A \ar@{^(->}[r] \ar@{->>}[d]  & B \ar@{->>}[r] \ar@{->>}[d]  & C \ar@{->>}[d]   \\
      \ddSeq{A}               & \hspace{-1.8cm} : \hspace{-1cm} & 1 \ar@{^(->}[r]               & C \ar@{=}[r]                 & C
    }
\endaligned
\end{equation}
which can be viewed as a short exact sequence $\qwSeq{\uSeq}: \ssSeq{A} \monoArrow \uSeq \epiArrow \ddSeq{C}$ of short exact sequences which will be denoted by $\qwSeq{\uSeq}$.

\subsection{Characterization of $\aSeq^{m}\monoArrow\seqWrm{\aSeq}{m}\epiArrow\seqZ{m}$}
Suppose we have a \myemph{short exact sequence} of short exact sequences $\kSeq \monoArrow \lSeq \epiArrow \seqZ{m}$:
\begin{equation}\label{equ:3x3_general:m}
\aligned
\xymatrix@C=0.8em@R=1.1em{
**[l] \kSeq: &
\ \kA \ \ar@{^(->}[d] \ar@{^(->}[rrr]         &&&
\ \kB \ \ar@{^(->}[d]  \ar@{->>}[rrr]  &&&
\ \kC \ \ar@{^(->}[d] \\
**[l] \lSeq: &
\ \xA \  \ar@{^(->}[rrr] \ar@{->>}[d]         &&&
\ \xB \   \ar@{->>}[rrr]^-{\fc}  \ar@{->>}[d]^-{\hb} &&&
\ \xC \  \ar@{->>}[d] \\
**[l] \seqZ{m}: &
 m\bZ \ar@{^(->}[rrr]                    &&&
\bZ    \ar@{->>}[rrr]^-{\mathrm{mod}\,m} &&&
\bZ_m
}
\endaligned
\end{equation}
in which $\kA, \xA, \kB$ are normal subgroups of $\xB$ and
\begin{align}
\ \ \xA &= \ker(\fc) = \hb^{-1}(m\bZ), &
\ \ \kB &= \ker(\hb), &
\ \ \kA &= \xA \cap \kB.
\end{align}

Fix an element $g\in \xB$ and let $\zB{0} \subset \kB$ be a subgroup.
Denote $\zA{0} = \kA \cap \zB{0}$ and  $\zC{0} = \zB{0}/\zA{0}$, so we get a short exact sequence $\uSeq: \zA{0} \monoArrow \zB{0} \epiArrow \zC{0}$.

Also let $\zB{i} := g^{-i} \zB{0} g^{i}$ and $\zA{i} := g^{-i} \zA{0} g^{i}$ for $i=0,\ldots,m-1$.
Then $\zB{i}\subset\kB$ and $\zA{i}\subset\kA$ since $\kB$ and $\kA$ are normal.

\def\bsp{\hspace{-7mm}}
\begin{sublemma}\label{lm:charact_seq_wrm}{\rm\cite[Lemma~2.3]{Maksymenko:TA:2020}, cf.\cite{KuznietsovaMaksymenko:Mob:2020}.}
Suppose that
\begin{enumerate}[leftmargin=*, label={\rm(\alph*)}]
\item\label{enum:3x3:g:m}
$\hb(g)=1$ and $g^m$ commutes with $\kB$;

\item\label{enum:3x3:prod:m}
$\kB$ splits into the product subgroups $\zB{0},\ldots,\zB{m-1}$, i.e. those subgroups generate $\kB$, pairwise commute, and $\zB{i}\cap \zB{j}=\{e\}$ for all $i\not=j$;

\item\label{enum:3x3:A_i:m}
$\zA{0},\ldots,\zA{m-1}$ generate $\kA$.
\end{enumerate}
Then the map $\beta:\zB{0}\wrm{m}\bZ \to \xB$ defined by
\begin{multline*}
\beta(\xv{0},\xv{1},\ldots,\xv{m-1}, k) = \\
 = \xv{0} \, (g^{-1} \xv{1} g^{1}) \, (g^{-2} \xv{2} g^{2}) \, \cdots \, (g^{-m+1} \xv{m-1} g^{m-1}) g^{k} = \\
 = \xv{0} \, g^{-1} \, \xv{1}\,  \cdots \, g^{-1} \, \xv{{}m-1} \, g^{-1+m+k},
\end{multline*}
for $\xv{i}\in \zB{0}$, $i=0,\ldots,m-1$, and $k\in\bZ$, is an isomorphism of groups inducing an isomorphism of exact $(3\times3)$-diagrams:
\begin{equation*}
\MResize{\textwidth}{
\xymatrix@C=1.7em@R=1.4em{
\zA{0}^m\times0 \ar@{^(->}[d] \ar@{^(->}[r]  &
\zB{0}^m\times0 \ar@{^(->}[d]  \ar@{->>}[r]  &
\zC{0}^m\times0 \ar@{^(->}[d]                &
\bsp\bsp\bsp &&
\ \kA \ \ar@{^(->}[d] \ar@{^(->}[r] &
\ \kB \ \ar@{^(->}[d]  \ar@{->>}[r] &
\ \kC \ \ar@{^(->}[d]               \\
\zA{0}^m\times m\bZ \ \ar@{^(->}[r]^-{\alpha'} \ar@{->>}[d] &
\zB{0}\wrm{m}\bZ \ar@{->>}[r]^-{\beta'} \ar@{->>}[d]^{p}    &
\zC{0}\wr\bZ_m \ar@{->>}[d]                                 &
\bsp\bsp\bsp\ar@{=>}[r]^-{~~\beta~~}_-{\cong} & &
\ \xA \  \ar@{^(->}[r] \ar@{->>}[d]                &
\ \xB \   \ar@{->>}[r]^-{\fc}  \ar@{->>}[d]^-{\hb} &
\ \xC \  \ar@{->>}[d]                              \\
 m\bZ \ar@{^(->}[r] &
\bZ    \ar@{->>}[r] &
\bZ_m               &
\bsp\bsp\bsp\ar@{=}[r]^-{~~\id~~} & &
m\bZ \ar@{^(->}[r] &
\bZ \ar@{->>}[r]   &
\bZ_m
}}
\end{equation*}
being identity on the lower sequence.
In other words, we get an \myemph{``isomorphism over $\seqZ{m}$''} of short exact sequences
\ $\uSeq^m \monoArrow \seqWrm{\uSeq}{m} \epiArrow \seqZ{m}$ \ and \ $\kSeq \monoArrow \lSeq \epiArrow \seqZ{m}$.
\end{sublemma}

\subsection{Characterization of $\aSeq^{mn}\monoArrow\seqWrm{\aSeq}{m,n}\epiArrow\seqZ{m,n}$}
Suppose we have a \myemph{short exact sequence} of short exact sequences $\kSeq \monoArrow \lSeq \epiArrow \seqZ{m,n}$:
\begin{equation}\label{equ:3x3_general:mn}
\aligned
\xymatrix@C=0.8em@R=1.5em{
**[l] \kSeq: &
\ \kA \ \ar@{^(->}[d] \ar@{^(->}[rrr]         &&&
\ \kB \ \ar@{^(->}[d]  \ar@{->>}[rrrr]  &&&&
\ \kC \ \ar@{^(->}[d] \\
**[l] \lSeq: &
\ \xA \  \ar@{^(->}[rrr] \ar@{->>}[d]         &&&
\ \xB \   \ar@{->>}[rrrr]^-{\fc}  \ar@{->>}[d]^-{\hb} &&&&
\ \xC \  \ar@{->>}[d] \\
**[l] \seqZ{m}: &
 m\bZ \times n\bZ \ar@{^(->}[rrr]                    &&&
\bZ\times\bZ    \ar@{->>}[rrrr]^-{\mathrm{mod}\,m,\, \mathrm{mod}\,n} &&&&
\bZ_m\times\bZ_n
}
\endaligned
\end{equation}
in which $\kA, \xA, \kB$ are normal subgroups of $\xB$ and
\begin{align}
\xA &= \ker(\fc) = \hb^{-1}(m\bZ\times n\bZ), &
\kB &= \ker(\hb), &
\kA &= \xA \cap \kB.
\end{align}

Let $\zB{0} \subset \kB$ be a subgroup, $\zA{0} = \kA \cap \zB{0}$ and  $\zC{0} = \zB{0}/\zA{0}$, so we get a short exact sequence $\uSeq: \zA{0} \monoArrow \zB{0} \epiArrow \zC{0}$.
Fix also two elements $g,h\in \xB$ and for $i=0,\ldots,m-1$, $j=0,\ldots,n-1$ put
\begin{align*}
	\zA{i,j} &:= g^{i} \, h^{j} \, \zA{0} \, h^{-j} \, g^{-i}, &
	\zB{i,j} &:= g^{i} \, h^{j} \, \zB{0} \, h^{-j} \, g^{-i}.
\end{align*}
Then $\zA{i,j}\subset\kA$ and $\zB{i,j}\subset\kB$ since $\kA$ and $\kB$ are normal.
The following lemma can be proved similarly to Lemma~\ref{lm:charact_seq_wrm}.
\begin{sublemma}\label{lm:charact_seq_wrmn}
Suppose that
\begin{enumerate}[leftmargin=*, label={\rm(\alph*)}]
\item\label{enum:3x3:g:mn}
$gh=hg$, $\hb(g)=(1,0)$, $\hb(h)=(0,1)$, and both $g^m$ and $h^n$ commute with $\kB$;

\item\label{enum:3x3:prod:mn}
$\kB$ splits into the product subgroups $\{\zB{i,j}\}$, i.e. those subgroups generate $\kB$, pairwise commute, and $\zB{i,j}\cap \zB{i',j'}=\{e\}$ for all $(i,j)\not=(i',j')$;

\item\label{enum:3x3:A_i:mn}
$\{ \zA{i,j} \}$ generate $\kA$.
\end{enumerate}
Then the map $\beta:\zB{0}\wrm{m,n}\bZ^2 \to \xB$ defined by
\[
\beta\bigl(\{ \xv{i,j} \}, k, l\bigr) =
\Bigl( \mprod\limits_{i,j} g^{i} \, h^{j} \, \xv{i,j} \, h^{-j} \, g^{-i} \Bigr) \cdot g^k h^l,
\]
for $\xv{i,j}\in \zB{0}$, $i=0,\ldots,m-1$, $j=0,\ldots,n-1$, and $k,l\in\bZ$, is an isomorphism of groups inducing an isomorphism of the following exact $(3\times3)$-diagram:
\begin{equation*}
\xymatrix@C=1.7em@R=1.4em{
	\uSeq^m \ar@{^(->}[d]: &
	\zA{0}^m\times 0 \times 0 \ar@{^(->}[d] \ar@{^(->}[r]  &
	\zB{0}^m\times 0 \times 0 \ar@{^(->}[d]  \ar@{->>}[r]  &
	\zC{0}^m\times 0 \times 0 \ar@{^(->}[d]                \\
	\seqWrm{\uSeq}{m,n} \ar@{->>}[d]: &
	\zA{0}^m\times m\bZ \times n\bZ \ \ar@{^(->}[r]^-{\alpha'} \ar@{->>}[d] &
	\zB{0}\wrm{m,n}\bZ^2 \ar@{->>}[r]^-{\beta'} \ar@{->>}[d]^{p}    &
	\zC{0}\wr(\bZ_m\times\bZ_n) \ar@{->>}[d]                    \\
	\seqZ{m,n} &
	 m\bZ\times n\bZ \ar@{^(->}[r] &
	\bZ\times \bZ    \ar@{->>}[r] &
	\bZ_m \times \bZ_n
	}
\end{equation*}
to the diagram~\eqref{equ:3x3_general:mn} and that isomorphism is the identity on the lower rows.
In other words, we get an \myemph{``isomorphism over $\seqZ{m,n}$''} of short exact sequences
\[
	\uSeq^m \monoArrow \seqWrm{\uSeq}{m,n} \epiArrow \seqZ{m,n}
	\qquad \text{and} \qquad
	\kSeq \monoArrow \lSeq \epiArrow \seqZ{m,n}.
\]
\end{sublemma}

\section{Homogeneous polynomials without multiple factors}\label{sect:homo_poly}
The fundamental theorem of algebra implies that every real homogeneous polynomial
$\gfunc:\bR^2\to\bR$ is a product of finitely many linear $L_i = a_i x + b_i y$ and irreducible over $\bR$ quadratic factors $Q_j(x,y) = c_j x^2 + 2d_j xy + e_j y^2$:
\[
	\gfunc(x,y) = \prod_{i=1}^{p} L_i(x, y) \cdot \prod_{j=1}^{q}Q_j(x, y).
\]

Evidently, $\gfunc$ has critical points if and only if $\deg\gfunc\geq2$.
Moreover, if $\gfunc$ has two proportional linear factors, that is $L_i = s L_j$ for some $i\not=j$ and $s\not=0$, then all the line $L_i=0$ consists of critical points of $\gfunc$.
This implies that \myemph{the origin $0\in\bR^2$ is a unique critical point of $\gfunc$ iff $\deg\gfunc\geq2$ and $\gfunc$ has no multiple linear factors}.

In the latter case the number $p$ of linear factors can be seen from the topological structure of level sets of $\gfunc$, see Figure~\ref{fig:isol_crit_pt}.
\begin{figure}[htbp!]
	\centering
	\footnotesize
	\begin{tabular}{ccccccc}
		\includegraphics[height=1.2cm]{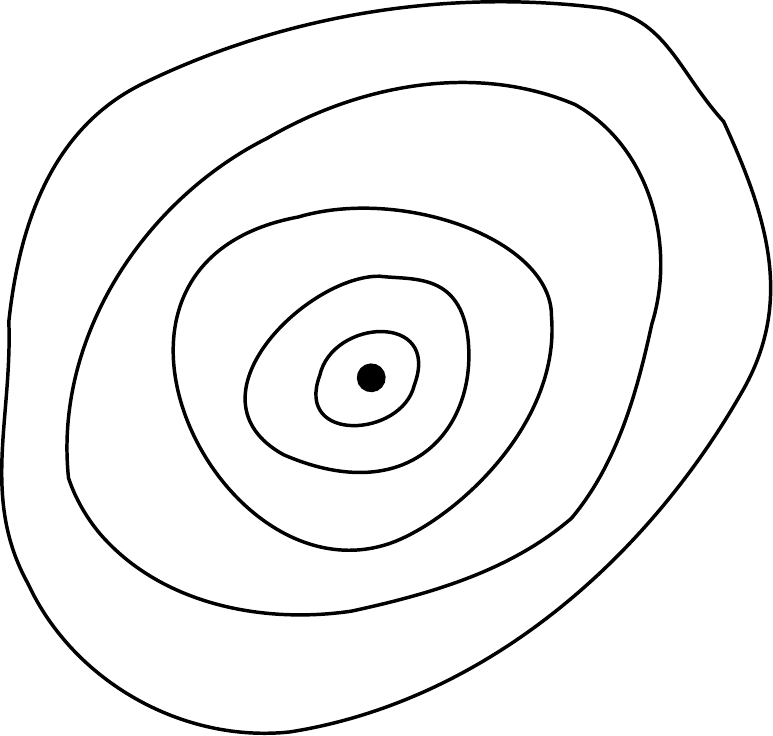}  & \ &
		\includegraphics[height=1.2cm]{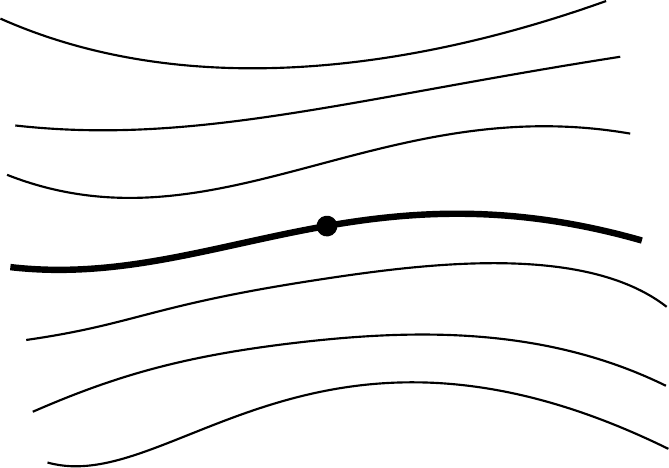}    & \ &
		\includegraphics[height=1.2cm]{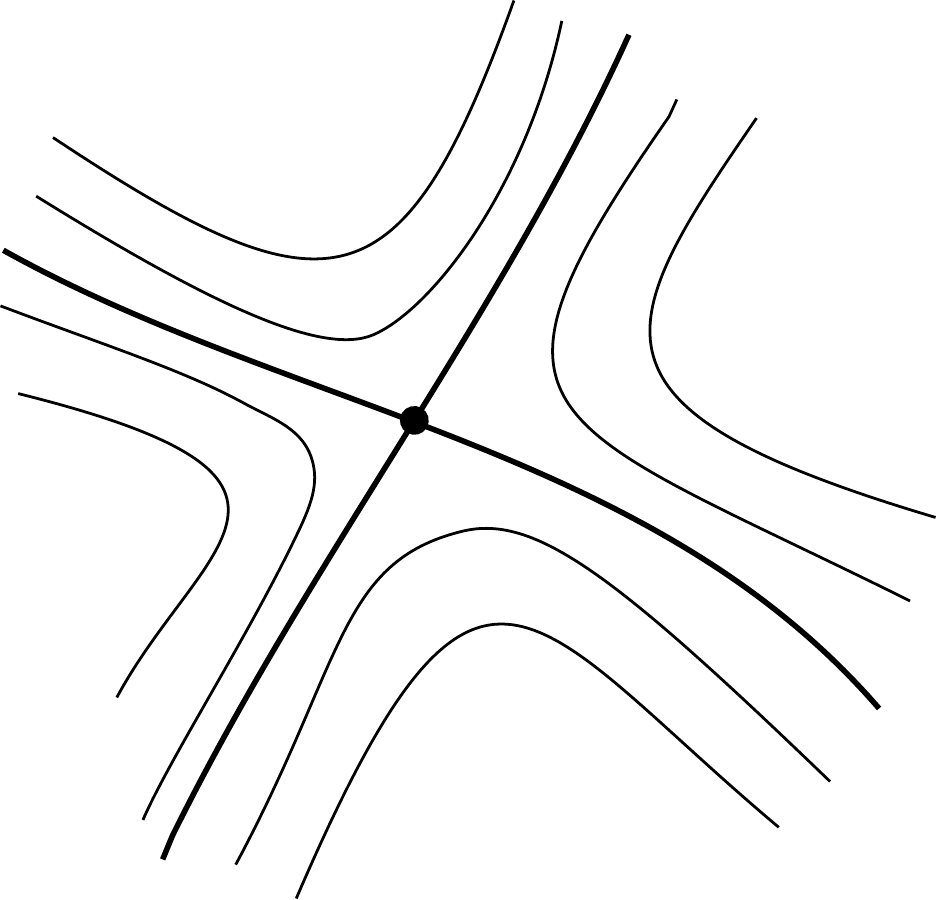}    & \ &
		\includegraphics[height=1.2cm]{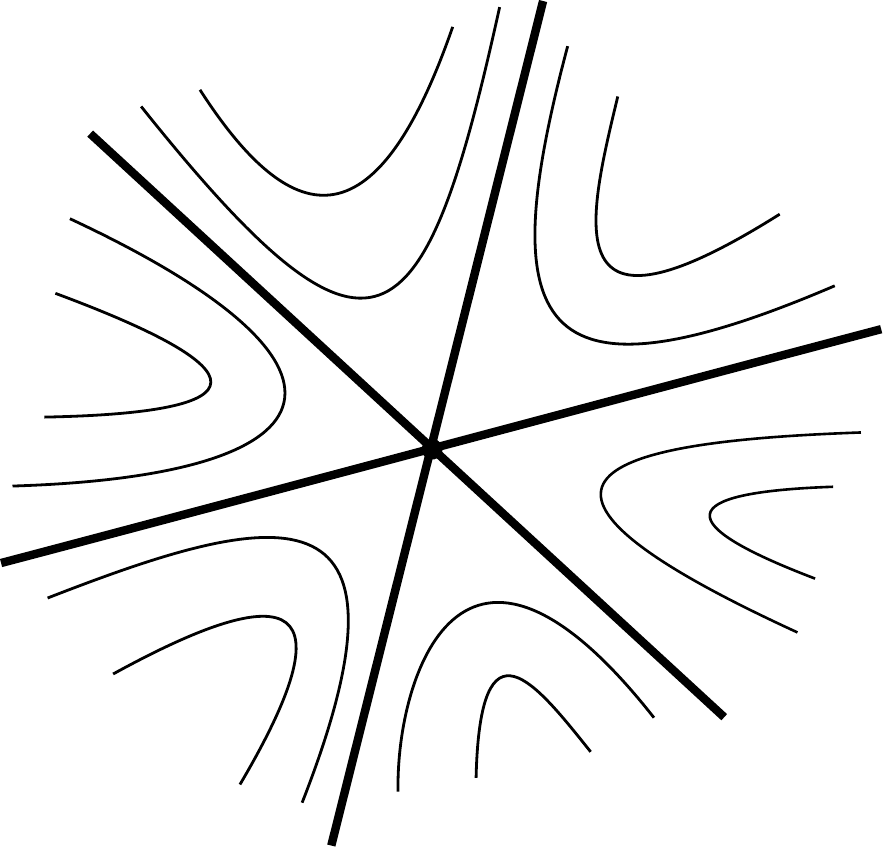}  \\
		$p=0$ & & $p=1$ & & $p=2$ & & $p=3$
	\end{tabular}
	\caption{Topological structure of level-sets of homogeneous polynomials without multiple factors}
	\label{fig:isol_crit_pt}
\end{figure}

We will say that $0\in\bR^2$ is
\begin{enumerate}[label={$\bullet$}, itemsep=0.8ex]
\item a \myemph{non-degenerate extreme} if $p=0$ and $q=1$, so $\gfunc = Q_1$;
\item a \myemph{degenerate extreme} if $p=0$ and $q\geq2$, so $\gfunc = Q_1 Q_2 \cdots Q_q$;
\item a \myemph{quasi saddle} if $p=1$, so $\gfunc = L_1 \cdot Q_1\cdots Q_q$;
\item a \myemph{non-degenerate saddle} if $p=2$ and $q=0$, so $\gfunc = L_1 L_2$;
\item a \myemph{saddle} if $p\geq2$ and $\deg\gfunc = p+2q\geq3$.
\end{enumerate}

\subsection{Symmetries}
Let $\gfunc:\bR^2\to\bR$ be a homogeneous polynomial without multiple factors.
Let also $\Stabilizer{\gfunc}$ be the group of \myemph{germs} at $0\in\bR^2$ of \myemph{diffeomorphisms} $\dif:(\bR^2,0)\to(\bR^2,0)$ such that $\gfunc\circ \dif = \gfunc$, and $\LStab(\gfunc)\subset \GL(\bR,2)$ be the group of \myemph{linear isomorphisms} $A:\bR^2\to\bR^2$ which also preserve $\gfunc$, that is $\gfunc(Ax) = \gfunc(x)$ for all $x\in\bR^2$.
Then $\LStab(\gfunc)$ can be regarded as a subgroup of $\Stabilizer{\gfunc}$.

Let $\dif\in\Stabilizer{\gfunc}$, and $A = J(\dif)$ be its Jacobi matrix at $0$.
Since $\gfunc$ is homogeneous, say of degree $k$, the identity $\gfunc\circ \dif = \gfunc$ easily implies that $\gfunc(Ax) = \gfunc(x)$ for all $x\in\bR^2$, \cite[Lemma~36]{Maksymenko:TA:2003}:
\[
\gfunc(x) = \frac{\gfunc(tx)}{t^k} = \frac{\gfunc\bigl(\dif(tx)\bigr)}{t^k} =
\gfunc\left( \frac{\dif(tx)}{t}\right) \xrightarrow[t\to 0]{} \gfunc(Ax).
\]
In other words, \myemph{if a diffeomorphism preserves a homogeneous polynomial $\gfunc$, then its Jacobi matrix at $0$ also preserves $\gfunc$}.
Hence we get a natural homomorphism:
\[J:\Stabilizer{\gfunc}\to \LStab(\gfunc), \qquad \dif\mapsto J(\dif). \]
Since for a linear map given by a matrix $A$ its Jacobi matrix is $A$, we see that $J$ is the identity on $\LStab(\gfunc)$.
In other words, $J$ is a \myemph{retraction} of $\Stabilizer{\gfunc}$ onto $\LStab(\gfunc)$, and in particular it is \myemph{surjective}.

Let also $\LStabPl(\gfunc) = \LStab(\gfunc)\cap \GL^{+}(\bR,2)$ and $\StabilizerPlus{\gfunc}$ be the subgroup of $\Stabilizer{\gfunc}$ consisting of orientation preserving germs.
Then \[ J(\StabilizerPlus{\gfunc}) = \LStabPl(\gfunc).\]

\begin{sublemma}\label{lm:LStabf}{\rm\cite[Lemma~6.2]{Maksymenko:MFAT:2009}.}
After a proper linear change of coordinates in $\bR^2$ and replacing (if necessary) $\gfunc$ with $-\gfunc$, one can assume that the following properties hold.
\begin{enumerate}[label=$\mathrm{(\alph*)}$, leftmargin=*, topsep=0.5ex, parsep=0ex, itemsep=0ex]
\item\label{enum:lm:LStabf:nondeg_loc_extr}
Suppose $0\in\bR^2$ is a \myemph{non-degenerate local extreme}, i.e.\ $\deg\gfunc=2$, and $\gfunc$ is an irreducible quadratic form.
Then
\begin{align*}
	\gfunc(x,y) &=x^2+y^2, &
	\LStab(\gfunc) &= O(2), &
	\LStabPl(\gfunc) &= SO(2).
\end{align*}

\item\label{enum:lm:LStabf:nondeg_saddle}
If $0\in\bR^2$ is a \myemph{non-degenerate saddle}, so $\deg\gfunc=2$, and $\gfunc$ is a product of two independent linear factors, then $\gfunc(x,y)=xy$,
\begin{align*}
\LStab(\gfunc) &=   \bigl\{ \pm\left(\begin{smallmatrix} t^{-1} & 0 \\ 0 & t\end{smallmatrix} \right) \mid t\not=0 \bigr\},
	& \
\LStabPl(\gfunc) &= \bigl\{ \left(\begin{smallmatrix} t^{-1} & 0 \\ 0 & t\end{smallmatrix} \right) \mid t\not=0 \bigr\}.
\end{align*}

\item\label{enum:lm:LStabf:degen}
Finally, assume that $0\in\bR^2$ is a \myemph{degenerate} critical point of $\gfunc$, so $\deg\gfunc\geq3$.
Then
\begin{enumerate}[label=$\mathrm{(\alph*)}$]
\item
$\LStabPl(\gfunc)$ is a finite cyclic subgroup of $SO(2)$ of some order $m\geq1$ generated by the rotation by $\frac{2\pi}{m}$;
\item
$\LStab(\gfunc)$ either coincides with $\LStabPl(\gfunc)$ or it is a dihedral group $\mathbb{D}_{m}$ of order $2m$ generated by $\LStabPl(\gfunc)$ and the reflection $(x,y)\mapsto(-x,y)$;
\item\label{enum:lm:LStab:c}
if $\deg\gfunc$ is even, e.g. $z$ is a \myemph{degenerate local extreme}, then $m$ is always even, since $\LStabPl(\gfunc)$ contains the map $q(x,y) = (-x,-y)$.
\end{enumerate}
\end{enumerate}
\end{sublemma}

\begin{subremark}\label{rem:comparing_deg_nondeg}
Statements~\ref{enum:lm:LStabf:nondeg_loc_extr} and~\ref{enum:lm:LStabf:degen} of Lemma~\ref{lm:LStabf} indicate an essential difference between diffeomorphisms preserving degenerate and non-degenerate local extremes.
Suppose we have a continuous family
\[\dif_t = (p_t,q_t):(\bR^2,0)\to(\bR^2,0), \quad t\in[0,1],\]
of germs of diffeomorphisms preserving a homogeneous polynomial $g(x,y)$, such that this family defines a continuous path in $\StabilizerPlus{\gfunc}$ ``with respect to at least $\Cr{1}$ topology'', which mean that the Jacobi matrices $J(\dif_t)$ are continuous in $t$.

If, for example, $g(x,y)=x^2+y^2$, so $p_t(x,y)^2+q_t(x,y)^2 \equiv x^2+y^2$, then \myemph{the matrices $J(\dif_t)$ can be arbitrary rotations}.
On the other hand, if $g(x,y)=x^4+y^4$, i.e. $p_t(x,y)^4+q_t(x,y)^4 \equiv x^4+y^4$, then there are only finitely many possibilities for $J(\dif_t)$.
Actually, in this case $J(\dif_t)$ can only be a rotation by $\frac{k\pi}{2}$ for $k=0,1,2,3$.
Therefore continuity of $J(\dif_t)$ in $t$ implies that now \myemph{$J(\dif_t)$ must be the same for all $t\in[0,1]$}.
\end{subremark}

\subsection{Symmetry index of a degenerate local extreme}
Suppose that $0\in\bR^2$ is a degenerate local extreme of $\gfunc$, as in the in the case 3\ref{enum:lm:LStab:c} of Lemma~\ref{lm:LStabf}.
The order $m$ of the cyclic group $\LStab(\gfunc)$ will be called the \myemph{symmetry index} of $0\in\bR^2$.

\begin{subexample}
Let \[ \gfunc(x,y)=x^4+y^4 = (x^2+\sqrt{2}xy + y^2)(x^2-\sqrt{2}xy + y^2).\]
Then $\LStabPl(\gfunc)\cong\bZ_4$ is generated by rotation by $\pi/2$: $r(x,y)=(-y,x)$, and $\LStab(\gfunc)$ is isomorphic to the dihedral group $\mathbb{D}_{4}$ generated by $r$ and the reflection $s(x,y)=(-x,y)$.
Thus here $m=4$.
\end{subexample}

\subsection{Framings at a degenerate local extreme}\label{sect:framing}
Let $v\in T_{0}\bR^2 = \bR^2$ be a non-zero tangent vector at $0\in\bR^2$.
Then its orbit
\[ \zfrm{0, v} = \{ T_{0}\dif(v) \mid \dif\in\Stabilizer{\gfunc} \} \]
with respect to $\LStab(\gfunc)$, and thus with respect to $\Stabilizer{\gfunc}$, consists of either $2m$ or $m$ vectors that are cyclically order, see Figure~\ref{fig:framings}.

This set $\zfrm{0, v}$ will be called a \myemph{framing} at $0\in\bR^2$, whenever either of the following equivalent conditions hold:
\begin{enumerate}[label={\rm(\alph*)}]
\item
the action of $\LStab(\gfunc)$ to $\zfrm{0, v}$ is \myemph{effective}, i.e. there is a non-unit element $A\in\LStab(\gfunc)$ fixed on $\zfrm{0, v}$;
\item if $\dif\in\Stabilizer{\gfunc}$ is such that $T_{0}\dif(w) = w$ for all $w\in \zfrm{0, v}$, then $T_{0}\dif=\id_{\bR^2}$.
\end{enumerate}
Roughly speaking the elements of $\LStab(\gfunc)$ can be distinguished by their action on the finite set $\zfrm{0, v}$.

\begin{subexample}
Let $\gfunc(x,y)=(x^2+y^2)(x^2+2y^2)$.
Then $\LStabPl(\gfunc) \cong \bZ_2$ is generated by $-\id_{\bR^2}$, while $\LStab(\gfunc)$ is the dihedral group $\mathbb{D}_{2} \cong \bZ_2 \times \bZ_2$ generated by $-\id_{\bR^2}$ and the reflection $s(x,y)=(-x,y)$.
Here $m=2$.
Also, if $v = (0,1)\in\bR^2$, then $\zfrm{0, v}=\{\pm v\}$ is not a framing, since $s$ trivially acts on $\zfrm{0, v}$.
For any other vector $w\in\bR^2$ which is not collinear to $v$, its orbit $\zfrm{0, w}$ is a framing.
\end{subexample}

\begin{subexample}
Let $\gfunc(x,y)=(x^2+y^2)(3x^2+2y^2)(x^2+ xy + y^2)$.
Then $\LStab(\gfunc)= \LStabPl(\gfunc) \cong \bZ_2$ is generated by $-\id_{\bR^2}$.
Here $m=2$ as well and for any non-zero vector $w\in\bR^2$ its orbit $\zfrm{0, w}$ is a framing.
\end{subexample}

The following easy lemma shows that one can always choose $v$ so that $\zfrm{0, v}$ is a framing.
\begin{sublemma}
\it
The set $\zfrm{0, v}$ is \myemph{not a framing} iff the following two conditions hold:
\begin{enumerate}[label={\rm(\alph*)}]
\item
$\LStab(\gfunc)$ is the dihedral group $\mathbb{D}_{2} \cong \bZ_2 \times \bZ_2$, so $m=2$;
\item
$v$ is fixed under the reflection from $\LStab(\gfunc)$.
\end{enumerate}
\end{sublemma}

\section{Space of maps $\FSP{\Mman}{\Pman}$}\label{sect:FSP}
In this section we will describe several properties of maps belonging to $\FSP{\Mman}{\Pman}$.
Recall that by definition a map $\func:\Mman\to\Pman$ belongs to $\FSP{\Mman}{\Pman}$ if it satisfies axioms~\ref{axiom:bd} and~\ref{axiom:sing}.

In the case when $\Pman=\Circle$ one can also say about local extremes of $\func$, and even about local minimums or maximums if we fix an orientation of $\Circle$.
Moreover, since $\Circle$ is a group $\bR/\bZ$, one can say about neighborhoods of a point $c\in\Circle$ of the form $(c-\eps; c+\eps)$ for small $\eps>0$.

First we will study structure of maps from $\FSP{\Mman}{\Pman}$ at their critical points.

\subsection{Smooth functions on the plane with isolated critical points}\label{sect:isolated_cr_pt}
Let $\gfunc:\bR^2 \equiv \bC \to \bR$ be a $\Cinfty$ function such that $0\in\bC$ is an isolated critical point and $\gfunc(0)=0$.
Then there are germs of homeomorphisms $\dif:(\bC,0) \to (\bC,0)$ and $\phi:(\bR,0) \to (\bR,0)$ such that
\begin{equation}\label{equ:isol_sing}
\phi \circ \gfunc \circ \dif (z) =
\begin{cases}
|z|^2, & \text{if $0\in\bC$ is a \myemph{local extreme} of $\func$, \cite{Dancer:2:JRAM:1987}}, \\
Re(z^m), & \text{otherwise, \cite{ChurchTimourian:PJM:1973, Prishlyak:TA:2002}}.
\end{cases}
\end{equation}
In particular, critical points of homogeneous polynomials without multiple factors cover all possible topological types of isolated critical points of maps $\bR^{2}\to\bR$.

\subsection{$\func$-adapted subsurfaces}\label{sect:func_adapted}
Now let $\func\in\FF(\Mman,\Pman)$.
As mentioned above condition~\ref{axiom:sing} implies that each critical point of $\pz\in\fSing$ of $\func$ is isolated.

A connected component $\Kman$ of a level-set $f^{-1}(c)$, $c\in\Pman$, will be called a \myemph{leaf} (of $\func$).
We also call $\Kman$ \myemph{regular} if it contains no critical points, and \myemph{critical} otherwise.

Evidently, a regular leaf of $\func$ is a submanifold of $\Mman$ diffeomorphic to the circle.
On the other hand, it follows from Axiom~\ref{axiom:sing}, (see also Figure~\ref{fig:isol_crit_pt}) that a critical leaf $\Kman$ has a structure of a $1$-dimensional CW-complex whose $0$-cells are critical points of $\func$ belonging to $\Kman$.
Notice that if $\Kman$ contains only quasi-saddles of $\func$, see Figure~\ref{fig:isol_crit_pt}\,b), then it is a smooth submanifold of $\Mman$ diffeomorphic to the circle, however it is still \myemph{critical} as a leaf of $\func$.

Let $\Kman$ be a (regular or critical) leaf of $\func$.
For $\eps>0$ let $\Nman_{\eps}$ be the connected component of $f^{-1}[c-\varepsilon, c+\varepsilon]$ containing $\Kman$.
Then $\Nman_{\eps}$ will be called an \myemph{$\func$-regular neighborhood of $\Kman$} if $\eps$ is so small that $\Nman_{\eps}\setminus\Kman$ contains no critical points of $\func$   and no boundary components of $\partial\Mman$.

A submanifold $\XSman\subset\Mman$ will be called \myemph{$\func$-adapted} if $\XSman = \mathop{\cup}\limits_{i=1}^{a} \Aman_i$, where each $\Aman_i$ is either a critical point of $\func$ or a regular leaf of $\func$ or an $\func$-regular neighborhood of some (regular or critical) leaf of $\func$.
We will denote by $\skl{\XSman}{i}$, $i=0,1,2$, the union of connected components of $\XSman$ of dimension $i$.

Notice that if $\XSman$ is an $\func$-adapted subsurface, then $\func|_{\XSman}:\XSman\to\Pman$ satisfies Axioms~\ref{axiom:bd} and~\ref{axiom:sing}, that is $\func|_{\XSman} \in \FF(\XSman,\Pman)$.

\subsection{Graph of $\func\in\FSP{\Mman}{\Pman}$}
Consider the partition $\KRGraphf$ of $\Mman$ into the leaves of $\func$, and let $\prj:\Mman\to\KRGraphf$ be the natural map associating to each $x\in\Mman$ the corresponding element of $\KRGraphf$ containing $x$.
Endow $\KRGraphf$ with the quotient topology, so a subset (a collection of leaves) $A\subset \KRGraphf$ is open iff $p^{-1}(A)$ (that is their union) is open in $\Mman$.
It follows from axioms~\ref{axiom:bd} and~\ref{axiom:sing} that $\KRGraphf$ has a natural structure of $1$-dimensional CW-complex, whose $0$-cells correspond to boundary components of $\Mman$ and critical leaves of $\func$.
We will call $\KRGraphf$ the \myemph{graph} of $\func$.

Since by definition $\func$ takes constant values on elements of $\KRGraphf$, it induces a function $\PF{\func}:\KRGraphf\to\Pman$ such that $\func = \PF{\func}\circ p$, see Figure~\ref{fig:graph}.
\begin{figure}[htbp!]
\centering
\includegraphics[height=1.8cm]{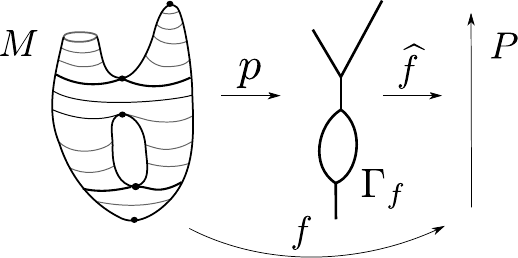}
\caption{}\label{fig:graph}
\end{figure}

\begin{subremark}
This graph was independently introduced in the papers by G.~Adelson-Welsky and A.~Kronrod~\cite{AdelsonWelskyKronrod:DANSSSR:1945}, and G.~Reeb~\cite{Reeb:CR:1946}, and often called \myemph{Kronrod-Reeb} or \myemph{Reeb} graph of $\func$.
It is a useful tool for understanding the topological structure of smooth functions on surfaces, e.g.~\cite{Kronrod:UMN:1950}, \cite{BolsinovFomenko:ENG:1997}.
It plays as well an important role in a theory of dynamical systems on manifolds and called \myemph{Lyapunov graph} of $\func$ following J.~Frank~\cite{Franks:Top:1985, RezendeFranzosa:TrAMS:1993, Yu:TrAMS:2013, RezendeLedesmaManzoli-NetoVago:TA:2018}.
The reason is that for generic Morse maps $\KRGraphf$ can be embedded into $\Mman$, so that $\func$ will be monotone on its edges.
\end{subremark}

\begin{subremark}
Notice that if $\Mman$ is a torus or a Klein bottle, then there exists a locally trivial fibrations $\func:\Mman\to\Circle$, see~maps~\ref{enum:specfunc:torus} and~\ref{enum:specfunc:klein} in Theorem~\ref{th:sdo:except_cases}.
Such a map has no critical points and so $\func\in\FSP{\Mman}{\Circle}$.
Evidently, the graph $\KRGraphf$ of $\func$ is a circle and it has no ``vertices'' that correspond to critical leaves.
In this case we assume that $\KRGraphf$ consists of one edge (homeomorphic to the circle) and has no vertices.
\end{subremark}

\begin{sublemma}\label{lm:p_H1M_H1G_surj}
The induced maps $\prj_1: H_1(\Mman,\bZ) \to H_1(\KRGraphf,\bZ)$ of homology groups is surjective and if $H_1(\Mman,\bZ) \not=0$, always have a non-trivial kernel.
Therefore if $\Mman=\Sphere$, $\Disk$, $\Circle\times[0,1]$, $\PrjPlane$, M\"obius band, then $\KRGraphf$ is a tree.
If $\Mman=\Torus$ or Klein bottle, then $\KRGraphf$ is either a tree or has a unique cycle.
\end{sublemma}

Denote by $\Homeo(\KRGraphf)$ the group of homeomorphisms of $\KRGraphf$.
Then for each $\dif\in\Stabilizer{\func}$ the identity $\func\circ\dif=\func$ implies that $\dif(\func^{-1}(c))=\func^{-1}(c)$ for all $c\in\Pman$.
Thus $\dif$ leaves invariant every level set of $\func$, and in particular induces a certain permutation $\rho(\dif)$ of connected components of $\func^{-1}(c)$, i.e. leaves of $\func$ being in turn points of $\KRGraphf$.
On other words, we get a map $\rho(\dif):\KRGraphf\to\KRGraphf$.
One can easily check that $\rho(\dif)$ is a \myemph{homeomorphism} of $\KRGraphf$ making commutative the following diagram:
\begin{equation}\label{equ:2x2_M_Graph}
\aligned
\xymatrix@R=2ex{
\Mman \ar[rr]^-{p} \ar[d]_-{\dif} &&
\KRGraphf \ar[rr]^-{\widehat{\func}} \ar[d]^-{\rho(\dif)} &&
\Pman \ar@{=}[d]  \\
\Mman \ar[rr]^-{p} &&
\KRGraphf \ar[rr]^-{\widehat{\func}} &&
\Pman
}
\endaligned
\end{equation}

Moreover, the correspondence $\dif\mapsto \rho(\dif)$ is a \myemph{homomorphism} of groups
\begin{equation}\label{equ:Stab_2_KRGraphf_homo}
	\rho :\Stabilizer{\func} \to \Homeo(\KRGraphf).
\end{equation}

\subsection{Enhanced graph of $\func\in\FF(\Mman,\Pman)$}\label{sect:enhanced_graph}
In order to encode an information coming from \myemph{degenerate local extremes} of $\func$, see Remark~\ref{rem:comparing_deg_nondeg}, we will add to $\KRGraphf$ new edges corresponding to framings at such points.

Let $z$ be a \myemph{degenerate} local extreme of $\func$ and $v\in T_{z}\Mman$, and
\[ \zfrm{z, v} = \{ T_{z}\dif(v) \mid \dif\in\Stabilizer{\func,z} \}  \ \subset \ T_{z}\Mman \]
be its orbit with respect to the action of $\Stabilizer{\func,z}$.
Similarly, to~\S\ref{sect:framing}, we will say that $\zfrm{z, v}$ is a \myemph{framing at $z$} if for every $\dif\in \Stabilizer{\func,z}$ such that $\dif$ fixes each element from $\zfrm{z, v}$ the tangent map $T_{z}\dif:T_{z}\Mman\to T_{z}\Mman$ is the identity.

Let $z_i$, $(i=1,\ldots,l)$, be all the \myemph{degenerate} local extremes of $\func$, and $\lfrm_i=\zfrm{z_i,v^i_0}$ be some framing at $z_i$ containing $k_i$ edges.
We will say that such a collection of framings $\lfrm = \{\lfrm_i\}_{i=1,\ldots,l}$ is \myemph{$\func$-adapted}, if it is invariant with respect to $\Stabilizer{\func}$, that is if $\dif\in\Stabilizer{\func}$ and $\dif(z_i) = z_j$ for some $i,j$, then $T_{z_i}\dif(\lfrm_i)=\lfrm_j$.
One easily checks, that $\func$-adapted framings always exist, \cite[Corolary 1]{Maksymenko:ProcIM:ENG:2010}.

\begin{figure}[ht]
\centering
\begin{tabular}{ccc}
\includegraphics[height=1.8cm]{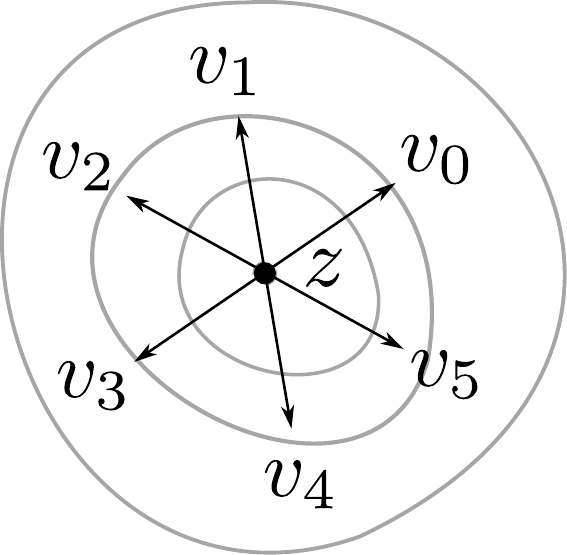}  &  \qquad  &
\includegraphics[height=1.8cm]{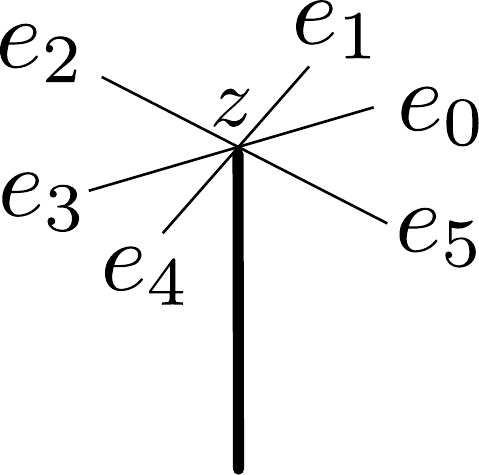}   \\
a) Framing at a degenerate local & & b) Enhanced graph  \\
extreme with symmetry index $m=6$
\end{tabular}
\caption{}\label{fig:framings}
\end{figure}

Thus $\Stabilizer{\func}$ naturally acts on $\KRGraphf$ as well as on each $\func$-adapted framing $\lfrm$ via the corresponding differentials of maps.
We want to ``join'' these actions.

Regard every point $z_i$, $i=1,\ldots,l$, as a vertex of $\KRGraphf$ of degree $1$, and glue to this vertex $k_i$ edges $e^i_{0},\ldots,e^i_{k_i-1}$ corresponding to the vectors in the corresponding framing $\zfrm{z_i,v^i_0}$, see Figure~\ref{fig:framings}b).
The obtained graph will be denoted by $\EKRGraphf$ and called the \myemph{enhanced} graph of $\func$.

One can assume that each new edge has length $1$.
Then the action of $\Stabilizer{\func}$ on $\KRGraphf$ extends to a unique action of $\Stabilizer{\func}$ on $\EKRGraphf$ so that each $\dif\in\Stabilizer{\func}$ interchanges edges  (via length preserving maps) of $\EKRGraphf\setminus\KRGraphf$ in the same way as its tangent map $T\dif$ interchanges vectors in the $\func$-adapted framing $\lfrm$.

\subsection{Action of $\Stabilizer{\func}$ on $\EKRGraphf$}
The previous paragraphs means that the homomorphism~\eqref{equ:Stab_2_KRGraphf_homo} extends to a homomorphism
\begin{equation}\label{equ:Stab_2_EKRGraphf_homo}
	\rho :\Stabilizer{\func} \to \Homeo(\EKRGraphf).
\end{equation}
associating to each $\dif\in\Stabilizer{\func}$ the induced homeomorphism of $\EKRGraphf$.
Let
\begin{align*}
	\FolStabilizer{\func} &= \ker(\rho), &
	\GrpKR{\func} &= \rho(\Stabilizer{\func}), &
	\GrpKR{\func,\XSman} &= \rho(\Stabilizer{\func,\XSman}),
\end{align*}
where $\XSman$ is any $\func$-adapted submanifold.
Then, we obtain the following short exact sequence:
\begin{equation}\label{equ:DSG_sequence_fX}
 \seqStab{\func,\XSman}: \quad \FolStabilizer{\func,\XSman} \monoArrow \Stabilizer{\func,\XSman} \epiArrow \GrpKR{\func,\XSman}.
\end{equation}
Evidently, $\dif\in\Stabilizer{\func}$ belongs to $\FolStabilizer{\func}$ if and only if
\begin{itemize}[leftmargin=*]
\item
$\dif$ preserves every leaf of $\func$;
\item
and for every \myemph{degenerate local extreme} $z$ (being also a critical leaf) of $\func$ the corresponding tangent map $T_{z}\dif: T_z\Mman\to T_z\Mman$ of $\dif$ at $z$ is the identity.
\end{itemize}
For an $\func$-adapted submanifold $\XSman$ let
\begin{align}
\label{equ:DeltafX}
\FolStabilizerIsotId{\func,\XSman} & = \FolStabilizer{\func} \cap \DiffId(\Mman,\XSman), \\
\label{equ:StabfX}
\StabilizerIsotId{\func,\XSman}    & = \Stabilizer{\func} \cap \DiffId(\Mman,\XSman), \\
\label{equ:GKRfX}
\GrpKRIsotId{\func,\XSman}         & = \rho( \StabilizerIsotId{\func,\XSman}).
\end{align}

\begin{sublemma}\label{lm:GfX}
{\rm{e.g.}~\cite[Lemma 2.2]{KravchenkoMaksymenko:EJM:2020}}
The group $\GrpKR{\func,\XSman}$ is finite.
Moreover, if $\func$ is a generic Morse map, then $\GrpKRIsotId{\func,\XSman}$ is trivial.
\end{sublemma}

\begin{sublemma}[{\cite[Lemma 4.1]{Maksymenko:TA:2020}}]\label{lm:DeltaIdfX}
$\StabilizerId{\func,\XSman}$ is the identity path component of $\FolStabilizerIsotId{\func,\XSman}$.
Therefore $\FolStabilizerIsotId{\func,\XSman}$ is a disjoint union of some path components of $\StabilizerIsotId{\func,\XSman}$, and the inclusion $\FolStabilizerIsotId{\func,\XSman} \subset \StabilizerIsotId{\func,\XSman}$ induces a monomorphism
\[
	\pi_0\FolStabilizerIsotId{\func,\XSman} \monoArrow \pi_0\StabilizerIsotId{\func,\XSman},
\]
whence
\begin{equation}\label{equ:GfX_quotient}
	\GrpKRIsotId{\func,\XSman} \cong
	\frac{\StabilizerIsotId{\func,\XSman}}{\FolStabilizerIsotId{\func,\XSman}} \cong
	\frac{\pi_0\StabilizerIsotId{\func,\XSman}}{\pi_0\FolStabilizerIsotId{\func,\XSman}}.
\end{equation}
\end{sublemma}
It follows from~\eqref{equ:GfX_quotient} that we have another short exact sequence
\begin{equation}\label{equ:bieberbach_sequence}
	\seqStabIsotId{\func,\XSman}:
	\quad
	\pi_0\FolStabilizerIsotId{\func,\XSman} \monoArrow
	\pi_0\StabilizerIsotId{\func,\XSman} \epiArrow
	\GrpKRIsotId{\func,\XSman}.
\end{equation}
This sequence will be called the \myemph{Biebarbach sequence of the pair $(\func,\XSman)$}.
It plays the main role in our considerations.

\subsection{Simplification results}
The following Lemmas~\ref{lm:reduct:reg_nbh} and~\ref{lm:reduction_Mconn_V_dM} reduce computation of $\seqStabIsotId{\func,\XSman}$ to the case when $\Mman$ is connected and $\XSman$ consists of critical points of $\func$ and boundary components of $\Mman$.
Let $\regN{\XSman}$ be a $\func$-regular neighborhood of $\XSman$.
\begin{sublemma}\label{lm:reduct:reg_nbh}{\rm(\cite[Corollary~7.2]{Maksymenko:TA:2020})}
The natural inclusions of pairs
\begin{gather}
    \label{equ:inclusios:stab}
     \bigl( \Stabilizer{\func,\regN{\XSman}}, \FolStabilizer{\func,\regN{\XSman}} \bigr)  \ \subset \ \bigl( \Stabilizer{\func,\XSman}, \FolStabilizer{\func,\XSman} \bigr), \\
    \label{equ:inclusios:stab:isot_id}
     \bigl( \StabilizerIsotId{\func,\regN{\XSman}}, \FolStabilizerIsotId{\func,\regN{\XSman}} \bigr)
    \ \subset \
    \bigl( \StabilizerIsotId{\func,\XSman}, \FolStabilizerIsotId{\func,\XSman} \bigr),
\end{gather}
are homotopy equivalences, and therefore they induce isomorphisms of the corresponding sequences:
\begin{align*}
    \seqStab{\func,\regN{\XSman}} &\cong \seqStab{\func,\XSman}, &
    \seqStabIsotId{\func,\regN{\XSman}} &\cong \seqStabIsotId{\func,\XSman}.
\end{align*}
\end{sublemma}

\begin{sublemma}[{cf.~\cite[Lemma~5.2]{Maksymenko:TA:2020}}]\label{lm:reduction_Mconn_V_dM}
Let $\Xman_1,\ldots,\Xman_{\cnt}$ all the connected components of $\overline{\Mman \setminus \regN{\XSman}}$, and $\hXman_i := \Xman_i \cap \regN{\XSman}$, $i=1,\ldots,\cnt$, so $\hXman_i$ consists of some common boundary components of $\Xman_i$ and $\regN{\XSman}$.
Then the natural homomorphism
\[
   \alpha: \DiffId(\Mman,\XSman) \to \prod_{i=1}^{\cnt} \DiffId(\Xman_i, \hXman_i), \qquad
   \alpha(\dif) = \bigl( \dif|_{\Xman_1}, \ldots, \dif|_{\Xman_{\cnt}}\bigr),
\]
induces an isomorphism $\seqStabIsotId{\func,\XSman} \cong \prod\limits_{i=1}^{\cnt} \seqStabIsotId{\func|_{\Xman_i},\hXman_i}$.

If $\Mman$ contains no connected components $\Mman_1$ and $\Mman_2$ that are diffeomorphic each other and do not intersect $\XSman$, then we also have an isomorphism $\seqStab{\func,\XSman} \cong \prod\limits_{i=1}^{\cnt} \seqStab{\func|_{\Xman_i},\hXman_i}$.
\end{sublemma}
\begin{smallproof}{Notes to the proof}
In fact, Lemma~5.2 in~\cite{Maksymenko:TA:2020} \myemph{wrongly} states that we always have an isomorphism of sequences $\seqStab{\func,\XSman} \cong \prod_{i=1}^{\cnt} \seqStab{\func|_{\Xman_i},\hXman_i}$.
The gap is in the claim that the map $\alpha: \Diff(\Mman,\Uman) \to \prod_{i=1}^{\cnt} \Diff(\Xman_i, \Uman_i)$ given by the same formula $\alpha(\dif) = \bigl( \dif|_{\Xman_1}, \ldots, \dif|_{\Xman_{\cnt}}\bigr)$ is \myemph{well-defined}, where $\Uman_i$ is a $\func$-regular neighborhood of $\hXman_i$ in $\Xman_i$, and $\Uman = \cup_{i=1}^{\cnt}\Uman_i$.

Indeed, the definition of $\alpha$ assumes that every $\dif\in\Diff(\Mman,\Uman)$ leaves invariant each connected component of $\Mman$.
However, it might happen that $\Mman$ has two diffeomorphic connected components $\Mman_1$ and $\Mman_2$ which do not intersect $\Uman$.
Then $\alpha$ is not-defined for a diffeomorphism of $\Mman$ interchanging $\Mman_1$ and $\Mman_2$ and fixed on $\Mman\setminus(\Mman_1\cup\Mman_2)$.

Nevertheless, the statement about isomorphism of $\DSG$-sequences remains true, since every isotopic to the identity diffeomorphism preserves every connected component.
\end{smallproof}

The following lemma allows to cut an $\func$-adapted collar of some boundary component of $\Mman$.
\begin{sublemma}\label{lm:cut_the_collar}{\rm(\cite[Theorem~5.5(3)]{Maksymenko:TA:2020})}
Let $\Mman$ be a not necessarily orientable connected compact surface, $\func\in\FF(\Mman,\Pman)$, and $\XSman$ be a connected component of $\partial\Mman$.
Suppose there exists a regular component $\Wman$ of some level set of $\func$ separating $\Mman$ and let $\Bman$ and $\Cylinder$ be the connected components of $\overline{\Mman\setminus\Wman}$.
Assume that $\XSman\subset\Cylinder$ and let $\Xman\subset \Bman\setminus \Wman$ be an $\func$-adapted submanifold, see Figure~\ref{fig:cylinder_split}.
\begin{figure}[htbp!]
\centering
\includegraphics[height=1.8cm]{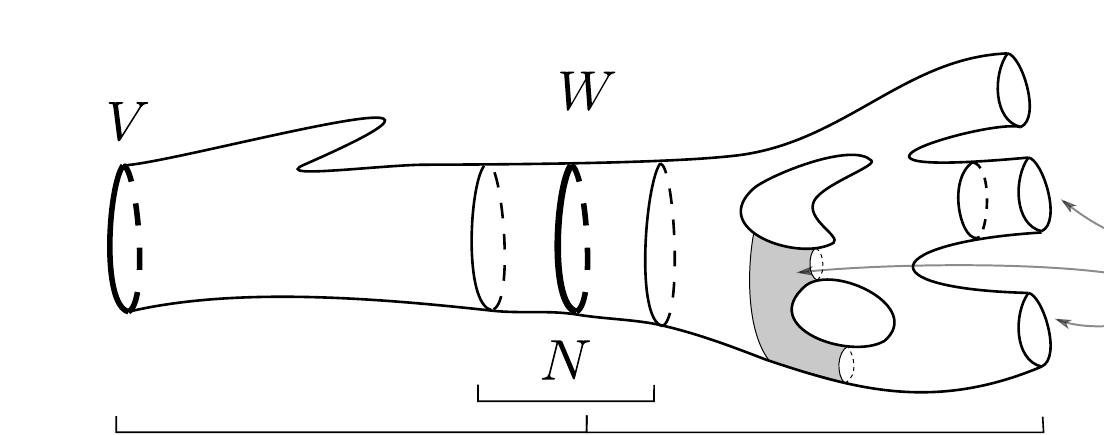}
\caption{}\label{fig:cylinder_split}
\end{figure}
Suppose $\Cylinder$ is an \myemph{annulus} and $\dif(\Cylinder)=\Cylinder$ for all $\dif\in\Stabilizer{\func,\Xman\cup\XSman}$.
Then we have the following homotopy equivalence:
\[
\Stabilizer{\func, \Xman \cup \XSman} \cong \Stabilizer{\func|_{\Bman}, \Xman \cup \Wman} \times \StabilizerIsotId{\func|_{\Cylinder}, \partial\Cylinder}
\]
which induces the homotopy equivalence
\[
	\StabilizerIsotId{\func, \Xman \cup \XSman} \cong \StabilizerIsotId{\func|_{\Bman}, \Xman \cup \Wman} \times \StabilizerIsotId{\func|_{\Cylinder}, \partial\Cylinder}
\]
and homotopy equivalences between the corresponding $\Delta$- and $\Delta'$-groups.
In particular, we get isomorphisms of short exact sequences:
\begin{equation}\label{equ:cyl:split}
\begin{aligned}
\seqStab{\func,\Xman\cup\XSman}&\cong \seqStab{\func|_{\Bman}, \Xman \cup \Wman} \times \seqStabIsotId{\func|_{\Cylinder}, \partial\Cylinder}, \\
\seqStabIsotId{\func,\Xman\cup\XSman} &\cong \seqStabIsotId{\func|_{\Bman}, \Xman \cup \Wman} \times \seqStabIsotId{\func|_{\Cylinder}, \partial\Cylinder}.
\end{aligned}
\end{equation}
Moreover, if $\Bman$ is either a $2$-disk or an annulus and $\Xman=\varnothing$, then
\begin{equation}\label{equ:bseq_two_cyl}
\seqStabIsotId{\func,\partial\Mman} \cong
\seqStabIsotId{\func|_{\Bman}, \partial\Bman} \times
\seqStabIsotId{\func|_{\Cylinder}, \partial\Cylinder}.
\end{equation}
\end{sublemma}

\subsection{Structure of $\pi_0\FolStabilizer{\func,\XSman}$}
Fix a possibly empty collection $\XSman$ of boundary components of $\Mman$.

Let $e$ be an (open) edge of $\KRGraphf$, and $x\in e$ be a point, so it corresponds to a regular leaf $\gamma$ of $\func$ in $\Int{\Mman}$.
Then there exists a Dehn twist $\dtw{\gamma} \in \FolStabilizer{\func}$ along $\gamma$ supported in $\func$-regular neighborhood of $\gamma$ which does not intersect $\partial\Mman$.
We will call a \myemph{Dehn twist around the edge $e$} and also denote by $\dtw{e}$.

Notice that by definition $\dtw{e}$ is fixed on $\partial\Mman \supset \XSman$.
Then its isotopy class in $\pi_0\FolStabilizer{\func,\XSman}$ will be denoted by $\idtw{e}{}$.
Since for distinct edges $e,e'$ their Dehn twists $\dtw{e}$ and $\dtw{e'}$ have disjoint supports, we obtain that $\dtw{e}$ and $\dtw{e'}$ commute.
Hence their isotopy classes $\idtw{e}{}$ and $\idtw{e'}{}$ in $\pi_0\FolStabilizer{\func,\XSman}$ also commute.

Let also $v_0, v_1$ be the vertices of $e$, and $\Kman_i = \prj^{-1}(v_i)$, $i=0,1$, be the corresponding leaves of $\func$.
Then each $\Kman_i$ is either a critical leaf of $\func$ or a boundary components of $\Mman$.
We will say that $e$ (as well as the corresponding Dehn twist $\dtw{e}$ is \myemph{internal for $\XSman$} if each $\Kman_i$ is one of following leaves:
\begin{itemize}
	\item either a degenerate local extreme,
	\item or contains a saddle point of $\func$,
	\item or is contained in $\XSman$.
\end{itemize}
Otherwise $e$ is \myemph{external with respect to $\XSman$}.

\begin{sublemma}\label{lm:pi0_Delta}{\rm(e.g.~\cite[Theorem~6.2]{Maksymenko:AGAG:2006})}
Let $\XSman$ be a possibly empty collection of boundary components of $\Mman$.
\begin{enumerate}[leftmargin=*, label={\rm(\arabic*)}, itemsep=0.8ex]
\item\label{enum:lm:pi0_Delta:1}
An edge $e$ of $\KRGraphf$ is \myemph{external}, if and only if $\dtw{e} \in \StabilizerId{\func}$.

\item\label{enum:lm:pi0_Delta:3}
Suppose $\KleinBottle$ is a Klein bottle, and $\func:\KleinBottle\to\Circle$ is a map without critical points, see map~\ref{enum:specfunc:klein} in Theorem~\ref{th:sdo:except_cases}.
Then there exists an isotopy $H:\KleinBottle\times[0,1]\to\KleinBottle$ such that $H_0 = \id_{\KleinBottle}$ and $H_1 \in \FolStabilizerIsotId{\func}$ preserves each leaf of $\func$ and reverses its orientation.
In this case $\pi_0\FolStabilizerIsotId{\func,\XSman}\cong 0$ and $\pi_0\FolStabilizer{\func,\XSman} \cong\bZ_2$.

\item\label{enum:lm:pi0_Delta:4}
In all other cases $\pi_0\FolStabilizer{\func,\XSman}$ is a free abelian group freely generated by internal for $\XSman$ Dehn twists, so $\pi_0\FolStabilizer{\func,\XSman} \cong \bZ^k$.
Hence $\pi_0\FolStabilizerIsotId{\func,\XSman}$ is also a free abelian group generated by \myemph{relations between internal Dehn twists in $\pi_0\Diff(\Mman,\XSman)$.}
\end{enumerate}
\end{sublemma}

\section{Homotopy types of stabilizers and orbits}\label{sect:hom_types_stab_orb}
In this section we assume that $\Mman$ is a compact surface, $\func\in\FSP{\Mman}{\Pman}$, and $\XSman$ is $\func$-adapted submanifold.
We will expose the results about homotopy types of $\Stabilizer{\func,\XSman}$ and orbits of $\Orbit{\func,\XSman}$.
Everywhere the sign $\homeq$ means ``\myemph{homotopy equivalent}''.
For a set $\XSman$ we will denote by $\ptnum{\XSman}$ the number of points in $\XSman$ if it is finite, and $\infty$ otherwise.

\begin{theorem}[\cite{Sergeraert:ASENS:1972, Maksymenko:AGAG:2006}]\label{th:SDO_lt_fibr}
\begin{enumerate}[leftmargin=*, label={\rm(\arabic*)}, itemsep=0.8ex, topsep=0.8ex]
\item\label{enum:th:SDO_lt_fibr:1}
The natural map into the orbit
\[
    \pj:\Diff(\Mman,\XSman)\to\Orbit{\func,\XSman}, \qquad
    \pj(\dif) = \func\circ\dif
\]
is a locally trivial principal $\Stabilizer{\func,\XSman}$-fibration.

\item\label{enum:th:SDO_lt_fibr:2}
The restriction $\pj:\DiffId(\Mman,\XSman)\to\OrbitPathComp{\func,\XSman}{\func}$ is a locally trivial principal $\StabilizerIsotId{\func,\XSman}$-fibration.

\item\label{enum:th:SDO_lt_fibr:2.1}
$\OrbitPathComp{\func,\XSman}{\func} = \OrbitPathComp{\func,\XSman \cup \Xman}{\func}$ for any collection $\Xman$ of boundary components of $\Mman$.

\item\label{enum:th:SDO_lt_fibr:3}
$\OrbitPathComp{\func}{\func}$ is a Fr\`echet manifold, whence {\em(by~\cite{Palais:Top:1966})} it has a homotopy type of a CW-complex.
\end{enumerate}
\end{theorem}

\begin{smallproof}{Notices to the proof}
Statements~\ref{enum:th:SDO_lt_fibr:2} and~\ref{enum:th:SDO_lt_fibr:2.1} are simple consequences of~\ref{enum:th:SDO_lt_fibr:2.1} (see~\cite[Corollary~2.1]{Maksymenko:UMZ:ENG:2012}).
Statements~\ref{enum:th:SDO_lt_fibr:1} and~\ref{enum:th:SDO_lt_fibr:3} were proved by F.~Sergeraert~\cite{Sergeraert:ASENS:1972} for functions $\func:\Mman\to\bR$ of the so-called \myemph{finite codimension} on a closed manifolds $\Mman$ and $\XSman=\varnothing$.
Further it was extended to a more general context of tame actions of \myemph{tame Lie groups on tame Frechet spaces} in~\cite[Appendix~11]{Maksymenko:AGAG:2006}.
The latter includes Theorem~\ref{th:SDO_lt_fibr}, see~\cite[Theorem~1.1]{Maksymenko:TA:2020} for details.
\end{smallproof}

\begin{subdefinition}[Weak homotopy equivalences of sequences of maps]
Suppose we are given two sequences of continuous maps of topological spaces:
\begin{align*}
	&\uSeq: \kA_1 \xrightarrow{\alpha_1} \cdots \xrightarrow{\alpha_{k-1}} \kA_{k}, &
	&\vSeq: \kB_1 \xrightarrow{\beta_1}  \cdots \xrightarrow{\beta_{k-1}} \kB_{k}.
\end{align*}
Then by a \myemph{morphism} $\gamma=(\gamma_1,\ldots,\gamma_{k}): \uSeq \to \vSeq$ of those sequences we will mean the following commutative diagram

\[
\aligned
\xymatrix{
\kA_1 \ar[r]^-{\alpha_1}  \ar[d]^-{\gamma_1} &
\kA_2 \ar[r]^-{\alpha_2}  \ar[d]^-{\gamma_2} &
\cdots \ar[r]^-{\alpha_{k-1}} & \kA_{k} \ar[d]^-{\gamma_k} \\
\kB_1 \ar[r]^-{\beta_1}   &
\kB_2 \ar[r]^-{\beta_2}   &
\cdots \ar[r]^-{\beta_{k-1}} & \kB_{k}
}
\endaligned
\]
Such a morphism will be called a \myemph{weak homotopy equivalence}, whenever each $\gamma_i$ is a weak homotopy equivalence.
In this case we will write $\gamma: \uSeq \whomeq \vSeq$.
\end{subdefinition}

The following sequence of maps will be called an $\sdo$-sequence of $(\func,\XSman)$:
\begin{equation}\label{equ:fibr_StabIsotId}
    \sdoseq{\func,\XSman}: \quad
    \FolStabilizerIsotId{\func,\XSman} \monoArrow{} \StabilizerIsotId{\func,\XSman} \monoArrow{} \DiffId(\Mman,\XSman) \xepiArrow{\pj} \OrbitPathComp{\func,\XSman}{\func},
\end{equation}
where the first two arrows are natural inclusions of normal subgroups, and the last map is a projection to cosets (we denoted it by $\epiArrow$ meaning that it is surjective).

We will study such sequences up to a weak homotopy equivalence and one of the main tools are long exact sequence of fibration $\pj$:
\begin{multline*}
    \cdots
    \to \pi_{k+1}\StabilizerId{\func,\XSman} \to \pi_{k+1}\DiffId(\Mman,\XSman) \xrightarrow{\pj_{k+1}} \pi_{k+1} \OrbitPathComp{\func,\XSman}{\func} \xrightarrow{\partial_{k+1}} \\
    \to \pi_{k}\StabilizerId{\func,\XSman}   \to \pi_{k}\DiffId(\Mman,\XSman)   \xrightarrow{\pj_{k}} \pi_{k} \OrbitPathComp{\func,\XSman}{\func}   \xrightarrow{\partial_{k}} \cdots \\    \cdots
                                            \to \pi_1\DiffId(\Mman,\XSman)     \xrightarrow{\pj_{1}} \pi_1\OrbitPathComp{\func,\XSman}{\func}      \xepiArrow{\partial_1} \pi_{0}\StabilizerIsotId{\func,\XSman},
\end{multline*}
where the corresponding base points are $\id_{\Mman}$ and $\func$.

The homotopy types of groups $\DiffId(\Mman,\XSman)$ are computed in~\cite[Theorem~A]{Smale:ProcAMS:1959}, \cite[Theorems~1,2]{EarleEells:BAMS:1967}, \cite[Theorems 1B,1C,1D]{EarleSchatz:DG:1970}, \cite[Th\'eor\`eme~1]{Gramain:ASENS:1973}, and are presented in the following Table~\ref{tbl:hom_type_DidMX}, in which $|\XSman|$ is the number of points in $\XSman$, $Mo$ is a M\"obius band, and $K$ is a Klein bottle.
Moreover, Theorem~\ref{th:stabilizer} below describes the homotopy type of $\Stabilizer{\func,\XSman}$.
This will give a complete information about higher homotopy groups of spaces in~\eqref{equ:fibr_StabIsotId} and a lot of information about fundamental groups.
\begin{table}[htbp!]
\centering
\caption{}\label{tbl:hom_type_DidMX}
\small
\begin{tabular}{|c|c|c|c|c|c|}\hline
$\Mman$  & $\ptnum{\XSman}$ &  \multicolumn{4}{|c|}{$\DiffIdM$} \\ \cline{3-6}
         &                 & homotopy    & \multicolumn{3}{|c|}{homotopy groups} \\ \cline{4-6}
         &                 & type        & $\pi_1$ & $\pi_2$ & $\pi_k$, $k\geq3$  \\ \hline \hline
$S^2$, $\PrjPlane$ & $0$ & $SO(3) \equiv \bR{P}^3$ & $\bZ_2$ & $0$ & $\pi_k S^3 = \pi_k S^2 = \pi_k \PrjPlane$ \\ \hline
$T^2$ & $0$ & $T^2$ & $\bZ^2$ & $0$ & $0$ \\ \hline
$\Disk$, $\Circle\times\UInt$, $Mo$, $\KleinBottle$ & $0$ & & & & \\  \cline{1-2}
$\Disk$, $S^2$, $\PrjPlane$           & $1$ & $\Circle$ & $\bZ$ & $0$ & $0$ \\ \cline{1-2}
$S^2$                              & $2$ & & & &\\ \hline
\multicolumn{2}{|c|}{other cases} & point  & $0$ & $0$ & $0$ \\ \hline
\end{tabular}
\end{table}

\begin{subremark}[Explicit construction of $\partial_1$]\label{rem:construction_partial_1}
Let $\gamma:[0,1]\to \OrbitPathComp{\func,\XSman}{\func}$ be a loop with $\gamma(0)=\gamma(1)=\func$, and $[\gamma]$ its homotopy class in $\pi_1\OrbitPathComp{\func,\XSman}{\func}$.
Since $\pj$ satisfies path lifting axiom, there exists a path $\widehat{\gamma}:[0,1]\to\DiffId(\Mman,\XSman)$ such that $\widehat{\gamma}(0)=\id_{\Mman}$, and $\widehat{\gamma}(1) \in \StabilizerIsotId{\func,\XSman}$.
Then \[ \pj_1([\gamma]) = [\widehat{\gamma}(1)] \in \pi_0 \StabilizerIsotId{\func,\XSman} \]
is the isotopy class of $\widehat{\gamma}(1)$ in $\StabilizerIsotId{\func,\XSman}$.
\end{subremark}

\begin{sublemma}\label{lm:long_ex_seq_pj}
\begin{enumerate}[leftmargin=*, label={\rm(\arabic*)}]
\item
$\DiffId(\Mman,\XSman)$ is contractible if and only if $\chi(\Mman) < |\XSman|$, which in particular holds if $\XSman$ is infinite or when $\chi(\Mman)<0$.

\item
If $\XSman=\varnothing$, then $\pi_2\DiffIdM =0$ and for $k\geq 3$ we have that
\[
    \pi_k\DiffIdM = \pi_k \Mman =
    \begin{cases}
        \pi_k S^3 & \text{if $\Mman = S^2$ or $\PrjPlane$}, \\
        0         & \text{otherwise}.
    \end{cases}
\]

\item
The image of $\pj_{1}\bigl( \pi_1\DiffId(\Mman,\XSman) \bigr)$ is contained in the center of $\pi_1\OrbitPathComp{\func,\XSman}{\func}$, (\cite[Lemma~2.2]{Maksymenko:UMZ:ENG:2012}).

\item
The set $\pi_{0}\StabilizerIsotId{\func,\XSman}$ is a group, and $\partial_1:\pi_1\OrbitPathComp{\func,\XSman}{\func} \to \pi_{0}\StabilizerIsotId{\func,\XSman}$ is a homomorphism.
\end{enumerate}

\end{sublemma}


\subsection{Structure of $\StabilizerId{\func,\XSman}$}
The following theorem is proved in~\cite{Maksymenko:AGAG:2006} for Morse maps and in~\cite{Maksymenko:ProcIM:ENG:2010} for all $\func\in\FSP{\Mman}{\Pman}$.
In fact, it is a consequence of a series of papers~\cite{Maksymenko:TA:2003, Maksymenko:hamv2, Maksymenko:CEJM:2009, Maksymenko:IUMJ:2010, Maksymenko:OsakaJM:2011} about diffeomorphisms preserving orbits of flows.
The technique developed in those paper is also extensively used for the proof of almost all presented here results, e.g. Lemmas~\ref{lm:pi0_Delta},~\ref{lm:reduct:reg_nbh}, \ref{th:stab:chi_neg}, Theorem~\ref{th:sdo:except_cases}.

\begin{subtheorem}\label{th:stabilizer}{\rm(\cite{Maksymenko:AGAG:2006, Maksymenko:ProcIM:ENG:2010}).}
Suppose $\Mman$ is connected.
Then $\StabilizerId{\func,\XSman}$ is \myemph{contractible} if and only if either of the following conditions is satisfied:
\begin{enumerate}[leftmargin=*, label={\rm(\roman*)}]
    \item $\func$ is has at least one saddle critical point or degenerate local extreme;
    \item $\Mman$ is non-orientable;
    \item $\XSman$ contains more than $\chi(\Mman)$ points (which holds e.g.~when $\dim\XSman\geq1$ or $\chi(\Mman)<0$).
\end{enumerate}
In all other cases $\StabilizerId{\func,\XSman}$ is \myemph{homotopy equivalent to the circle $S^1$}.
In fact the latter holds if and only if $\func$ is one of maps of types~\ref{enum:specfunc:sphere}-\ref{enum:specfunc:torus} of Theorem~\ref{th:sdo:except_cases} below.
\end{subtheorem}

The following theorem characterizes Morse maps having only local extremes and describes their $\sdo$-sequences.

\begin{subtheorem}\label{th:sdo:except_cases}{\rm(\cite[Theorem~1.9]{Maksymenko:AGAG:2006})}
Let $\func\in\FSP{\Mman}{\Pman}$.
Then every critical point of $\func$ is a \myemph{non-degenerate local extreme} if and only if $\func$ is one of the maps of the following types~\ref{enum:specfunc:sphere}-\ref{enum:specfunc:klein}.
\begin{enumerate}[leftmargin=*, label={\rm(\Alph*)}, itemsep=0.8ex]
\item\label{enum:specfunc:sphere}
$\func:S^2\to\Pman$ is a Morse map having precisely two critical points $\{a,b\}$, and those points are non-degenerate local extremes.
In this case
\begin{gather*}
    \sdoseq{\func} \ \whomeq \ [ \Circle \equiv \Circle \monoArrow{} SO(3) \epiArrow{} S^2 ], \\
    \sdoseq{\func, \{a\}} \ \whomeq \ \sdoseq{\func, \{a, b\}} \ \whomeq \ [\Circle \equiv \Circle \equiv \Circle \epiArrow{} \pnt].
\end{gather*}

\item\label{enum:specfunc:disk}
$\func\in\FSP{\Disk}{\Pman}$ is a map having precisely one critical point $a$, and that point is a non-degenerate local extreme.
In this case
\begin{gather*}
    \sdoseq{\func} \ \whomeq \ \sdoseq{\func, \{a\}} \ \whomeq \ [\Circle \equiv \Circle \equiv \Circle \epiArrow{} \pnt].
\end{gather*}

\item\label{enum:specfunc:cylinder}
$\func\in\FSP{\Circle\times[0,1]}{\Pman}$ is a map having no critical points and
\[ \sdoseq{\func} \ \whomeq \ [\Circle \equiv \Circle \equiv \Circle \epiArrow{} \pnt].\]

\item\label{enum:specfunc:torus}
$\func\in\FSP{\Torus}{\Circle}$ is a map without critical points.
Let $m\bZ \subset \bZ \cong \pi_1 \Circle$ be the image of the induced homomorphism $\func_{1}:\pi_1 \Torus \to \pi_1 \Circle$ for some $m$.
Then $m\geq1$, $\func:\Torus\to\Circle$ is a locally trivial fibration whose fiber is a disjoint union of $m$ circles, and
\[
\sdoseq{\func}  \ \whomeq \ [\Circle\times 0 \monoArrow \Circle\times\bZ_{m} \monoArrow\Circle\times\Circle \xepiArrow{(w,z) \mapsto z^m} \Circle]
\]

\item\label{enum:specfunc:klein}
$\func\in\FSP{\KleinBottle}{\Circle}$ is be a map from Klein bottle to the circle having no critical points.
Let $m\bZ \subset \bZ \cong \pi_1 \Circle$ be the image of the induced homomorphism $\func_{1}:\pi_1 \KleinBottle \to \pi_1 \Circle \equiv \bZ$ for some $m$.
Then $m\geq1$, $\func:\KleinBottle\to\Circle$ is a locally trivial fibration whose fiber is a disjoint union of $m$ circles, and
\[
\sdoseq{\func}  \ \whomeq \ [0 \monoArrow \bZ_{m} \monoArrow \Circle \xepiArrow{z \mapsto z^m} \Circle]
\]
\end{enumerate}
All the above weak homotopy equivalences are homotopy equivalences.
\end{subtheorem}

\begin{subremark}
To clarify the statement notice that the first weak homotopy equivalence $\sdoseq{\func} \ \whomeq \ [ \Circle \equiv  \Circle \monoArrow{} SO(3) \epiArrow{} S^2 ]$ in~\ref{enum:specfunc:sphere} means that $\OrbitPathComp{\func}{\func}$ is homotopy equivalent to $2$-sphere $S^2$, and $\StabilizerIsotId{\func}$ is homotopy equivalent to the circle $\Circle$.
Since it is path connected, we also have that $\StabilizerIsotId{\func} = \FolStabilizerIsotId{\func} = \StabilizerId{\func}$.

For example, regard $S^2=\{x^2+y^2+z^2=1\}$ as a unit sphere in $\bR^3$, and let $\func:S^2\to\bR$ be given by $\func(x,y,z) = z$.
Then $\func$ is a function of type~\ref{enum:specfunc:sphere}.
Let $H:S^2\times[0,1] \to S^2$ be the linear rotations given by the formula:
\[
    H(x,y,z,t) =
    \begin{pmatrix}
        \cos 2\pi t & \sin 2\pi t & 0 \\
        -\sin 2\pi t & \cos 2\pi t & 0 \\
        0 & 0 & 1
    \end{pmatrix}
    \begin{pmatrix}
        x \\ y \\ z
    \end{pmatrix}
\]
Evidently, each $H_t$ preserves coordinate $z$, i.e.~$\func\circ H_t = \func$.
Moreover, $H_0 = H_1 =\id_{S^2}$.
In other words, $H_t$ is a loop in $\Stabilizer{\func}$ at $\id_{S^2}$.
Then the inclusion of this loop as a map $\nu:\Circle\to\Stabilizer{\func}$, $\nu(e^{2\pi i t}) = H_t$, which turns out to be a homotopy equivalence.

Similar consideration can be applied to other cases and we refer the reader to~\cite[Proof of Theorem~1.9]{Maksymenko:AGAG:2006}.
\end{subremark}

\begin{subcorollary}\label{cor:sed}
Suppose $\StabilizerId{\func,\XSman}$ is contractible.
Then the following statements hold.
\begin{enumerate}[leftmargin=*, label={\rm(\alph*)}]
\item\label{enum:cor:sed:pinOf}
$\pi_{\vn}\Diff(\Mman,\XSman) \cong \pi_{\vn}\OrbitPathComp{\func,\XSman}{\func}$ for $\vn\geq2$ and we get the following short exact sequence
\begin{equation}\label{equ:exact_seq_for_pi1OfX}
    \pi_1\DiffId(\Mman,\XSman)  \xmonoArrow{~\pj_{1}~}
    \pi_1\OrbitPathComp{\func,\XSman}{\func} \xepiArrow{~\partial_1~}
    \pi_{0}\StabilizerIsotId{\func,\XSman}.
\end{equation}

\item\label{enum:cor:sed:not_s2}
If $\Mman\not= S^2$ and $\PrjPlane$, then $\OrbitPathComp{\func,\XSman}{\func}$ is aspherical%
\footnote{A path connected topological space $Q$ is \myemph{aspherical}, if $\pi_{\vn} Q = 0$ for $\vn\geq2$.
In this case $Q$ is also called Eilenberg–MacLane space $K(\pi_1Q, 1)$.}, so its homotopy type is determined by the fundamental group $\pi_1\OrbitPathComp{\func,\XSman}{\func}$.
In particular, $\pi_1\OrbitPathComp{\func,\XSman}{\func}$ is torsion free.

\item\label{enum:cor:sed:chi_leq_X}
If $\chi(\Mman) < |\XSman|$, then $\DiffId(\Mman,\XSman)$ is contractible as well, and and Eq.~\eqref{equ:exact_seq_for_pi1OfX} yields an isomorphism
\begin{equation}\label{equ:pi1OfX_pi0SfX}
\partial_1: \pi_1\OrbitPathComp{\func,\XSman}{\func} \cong \pi_0\StabilizerIsotId{\func,\XSman}.
\end{equation}
Then $\pi_0\StabilizerIsotId{\func,\XSman}$ is also torsion free and this group completely determine the homotopy type of $\OrbitPathComp{\func,\XSman}{\func}$.
\end{enumerate}
\end{subcorollary}

\subsection{Structure of $\pi_{1}\OrbitPathComp{\func,\XSman}{\func}$}
We will describe here the general structure of fundamental groups of orbits, see Theorem~\ref{th:cases:Bib:prop} and Remark~\ref{rem:prop:bib}.

Suppose that $\StabilizerId{\func,\XSman}$ is \myemph{contractible}.
Then we have the following commutative diagram:
{\small%
\begin{equation}\label{equ:diagram_3x3_sdo_b}
\aligned
\xymatrix@R=3ex{
 \ \pi_1\Diff(\Mman,\XSman) \ar@{^(->}[r]^-{\pj_1} \ \ar@{=}[d] &
 \ \partial_1^{-1}\bigl(\pi_0\FolStabilizerIsotId{\func,\XSman} \bigr) \ \ar@{^(->}[d]^{} \ar@{->>}[r]^-{\partial_1} &
 \ \pi_0\FolStabilizerIsotId{\func,\XSman} \  \ar@{^(->}[d]^{} \\
 \ \pi_1\Diff(\Mman,\XSman) \ar@{^(->}[r]^-{\pj_1} \  &
 \ \pi_1\OrbitPathComp{\func,\XSman}{\func} \ \ar@{->>}[r]^-{\partial_1} \ar@{->>}[d]_-{\rho\circ\partial_1} &
 \ \pi_0\StabilizerIsotId{\func,\XSman} \ \ar@{->>}[d]^{\rho} \\
 &
 \ \GrpKRIsotId{\func,\XSman} \ \ar@{=}[r] &
 \ \GrpKRIsotId{\func,\XSman} \
}
\endaligned
\end{equation}}%
in which all horizontal and vertical sequences are exact.
In fact, the middle horizontal sequence is the last part of the above long exact sequence of $\pj$, while the right vertical is $\DSG$-sequence for $(\func,\XSman)$.

Denote $\Aman = \partial_1^{-1}\bigl(\pi_0\FolStabilizerIsotId{\func,\XSman} \bigr)$ and consider the upper horizontal sequence:
\begin{equation}\label{equ:seq:DJDelta}
    \pi_1\Diff(\Mman,\XSman) \ \xmonoArrow{~\pj_1~} \ \Aman \ \xepiArrow{~\partial_1~} \ \pi_0\FolStabilizerIsotId{\func,\XSman}.
\end{equation}
We know that $\pi_0\FolStabilizerIsotId{\func,\XSman}$ is a free abelian, $\pi_1\Diff(\Mman,\XSman)$ is abelian ($0$, $\bZ_2$, $\bZ$, or $\bZ^2$) and its image is contained in the center of $\pi_1\OrbitPathComp{\func,\XSman}{\func}$, and therefore in the center of $\Aman$.
One might expect that this sequence splits, i.e.~$\partial_1$ in~\eqref{equ:seq:DJDelta} admits a section $q:\pi_0\FolStabilizerIsotId{\func,\XSman} \to \Aman$, which gave an isomorphism
$\phi:\pi_1\DiffId(\Mman,\XSman)\oplus \pi_0\FolStabilizerIsotId{\func,\XSman} \cong \Aman$, $\phi(a,b) = \pj_1(a)\cdot q(b)$.
This was initially stated in~\cite[Theorem~1.5]{Maksymenko:AGAG:2006} for $\XSman=\varnothing$, however the arguments contained a gap and were completed in~\cite[Theorem~6]{Maksymenko:ProcIM:ENG:2010}.

The following theorem describes a more general effect.
Let $\torus{k}$ be the $k$-dimensional torus, i.e. a product of $k$ circles.

\begin{subtheorem}\label{th:sect_partial1}{\rm(cf.~\cite[Theorem~6]{Maksymenko:ProcIM:ENG:2010})}
Suppose $\StabilizerId{\func,\XSman}$ is contractible and that 
\[ \pi_0\FolStabilizerIsotId{\func,\XSman} \cong \bZ^k\]
for some $k\geq0$.
Let $\omega:\pi_1\torus{k}\to \pi_0\FolStabilizerIsotId{\func,\XSman}$ be any isomorphism.
Then there exists a continuous map $\eta: \torus{k} \to \OrbitPathComp{\func,\XSman}{\func}$ such that we have the following commutative diagram:
\begin{equation}\label{equ:map_of_torus}
\aligned
\xymatrix@R=3ex{
\pi_1\torus{k} \ar[r]^-{\omega}_-{\cong} \ar[d]_-{\eta} &  \pi_0\FolStabilizerIsotId{\func,\XSman} \ar@{^(->}[d] \\
\pi_1 \OrbitPathComp{\func,\XSman}{\func}  \ar[r]^-{\partial_1} &  \pi_0\StabilizerIsotId{\func,\XSman}
}
\endaligned
\end{equation}
Hence for the map
\[
\psi: \DiffId(\Mman,\XSman) \times \torus{k} \to \OrbitPathComp{\func,\XSman}{\func},
\qquad
\psi(\dif, \ps) = \eta(\ps)\circ \dif
\]
the induced homomorphisms $\psi_{\vn}: \pi_1\bigl(\DiffId(\Mman,\XSman) \times \torus{k}\bigr) \to \pi_1\OrbitPathComp{\func,\XSman}{\func}$ are \myemph{isomorphisms} for $\vn\geq2$ and a monomorphism for $\vn=1$, so we will have the following commutative diagram:
\begin{equation}\label{equ:diagram_3x3_sdo_b:reduced}
    \aligned
    \xymatrix@R=3ex{
        \ \pi_1\Diff(\Mman,\XSman)  \oplus \pi_0\FolStabilizerIsotId{\func,\XSman} \ \ar@{->>}[rr]^-{\prj_2} \ar@{^(->}[rd]^-{\prj_1 \oplus \psi_1}  \ar@{=}[d] & &
        \ \pi_0\FolStabilizerIsotId{\func,\XSman} \ \ar@{^(->}[d]^{} \\
        \ \pi_1\Diff(\Mman,\XSman) \ar@{^(->}[r]^-{\pj_1} \  &
        \ \pi_1\OrbitPathComp{\func,\XSman}{\func} \ \ar@{->>}[r]^-{\partial_1} \ar@{^(->}[rd]_-{\rho\circ\partial_1} &
        \ \pi_0\StabilizerIsotId{\func,\XSman} \ \ar@{->>}[d]^{\rho} \\
        &
        &
        \ \GrpKRIsotId{\func,\XSman} \
    }
    \endaligned
\end{equation}
in which main diagonal, middle row and right column are exact.

Therefore, if $\sigma: \tilde{\Orb} \to \OrbitPathComp{\func,\XSman}{\func}$ is a covering map corresponding to the subgroup 
\[ \psi_1\bigl( \pi_1\bigl(\DiffId(\Mman,\XSman) \times \torus{k}\bigr) \bigr) \]
of $\pi_1 \OrbitPathComp{\func,\XSman}{\func}$, then $\psi$ lifts to the map 
\[ \hat{\psi}:\DiffId(\Mman,\XSman) \times \torus{k} \to \tilde{\Orb}\]
such that $\psi = \sigma\circ\hat{\psi}$ and $\hat{\psi}$ induces an isomorphism of all homotopy groups, i.e.~$\hat{\psi}$ is a weak homotopy equivalence.

In particular, if $\GrpKRIsotId{\func,\XSman}$ is trivial (which holds e.g.~when $\func$ is a generic Morse map), then $\psi_1$ is an isomorphism as well, whence $\psi$ is a weak homotopy equivalence (cf.~\cite[Theorem~1.5(3)]{Maksymenko:AGAG:2006}).
\end{subtheorem}
\begin{proof}
Let $I^k = [0,1]^k$ be a $k$-dimensional cube.
Then there is a natural quotient map
\[
\nu: I^k \to \torus{k},
\qquad
\nu(\pt_1,\ldots,\pt_k) = \bigl( e^{2\pi i t_1}, \ldots,  e^{2\pi i t_k}\bigr).
\]
We will construct a certain continuous map $\zeta:I^k \to \DiffId(\Mman,\XSman\cup\Qman)$ such that for all $i=1,\ldots,k$, and $t_1,\ldots,t_{i-1},t_{i+1}, \ldots, t_k \in [0,1]$ we will have
\[
\func\circ \zeta(t_1,\ldots,t_{i-1}, 0, t_{i}, \ldots, t_k) = \func\circ \zeta(t_1,\ldots,t_{i-1}, 1, t_{i+1}, \ldots, t_k).
\]

This will imply existence of a unique continuous map $\eta:\torus{k} \to \OrbitPathComp{\func,\XSman}{\func}$ with $\pj\circ \zeta = \eta\circ\nu$.
Moreover, $\zeta$ will be constructed so that $\eta$ will make commutative the diagram~\eqref{equ:map_of_torus}.

{\em Notation.}
Notice that $[0,1]$ can be regarded as a CW-complex with two $0$-cells $\{0\}$ and $\{1\}$ and a $1$-cell $(0,1)$.
Then $I^k = [0,1]^k$ inherits a CW-structure, whose cells are product of cells of copies of $[0,1]$.
It will be convenient to introduce uniform notation for such faces.

\newcommand\eface[3]{E_{#1,\, #2,\, #3}}
Let $\Kman = \{1,\ldots,k\}$.
Then each $i$-th face $E$ of $I^k$ is determined by partition of $\Kman$ into three mutually disjoint subsets $\Aman$, $\Bman$, $\Cman$ such that
\begin{itemize}[leftmargin=5ex]
\item for each $s\in \Aman$ there is $(\pt_1,\ldots,\pt_k)\in E$ with $0 < \pt_s < 1$.
\item for each $s\in\Bman$ and $(\pt_1,\ldots,\pt_k)\in E$ we always have $\pt_s=0$;
\item for each $s\in\Cman$ and $(\pt_1,\ldots,\pt_k)\in E$ we always have $\pt_s=1$;
\end{itemize}
Denote such face by $\eface{\Aman}{\Bman}{\Cman}$.
Evidently, the dimension of $\eface{\Aman}{\Bman}{\Cman}$ equals to the number of elements of $\Aman$.
Moreover, when writing points of $\eface{\Aman}{\Bman}{\Cman}$ we will sometimes indicate only coordinates belonging to $\Aman$.

\smallskip

{\em Construction of $\zeta$.}
One can easily find
\begin{enumerate}[label={\rm(\alph*)}]
\item\label{enum:ex_sect:Qinf} a subset $\Qman\subset\Mman$ such that $|\Qman|>\chi(\Mman)$;
\item\label{enum:ex_sect:taui} pairwise commuting diffeomorphisms $\tau_1,\ldots,\tau_k \in \FolStabilizerIsotId{\func,\XSman}$ whose isotopy classes $[\tau_i]$ constitute a basis for $\pi_0\FolStabilizerIsotId{\func} \cong \bZ^k$;
\item\label{enum:ex_sect:isot} isotopies $H^i:\Mman\times[0,1]\to\Mman$ relatively $\Qman$ with $H^i=\id_{\Mman}$ and $H^i_1 = \tau_i$ for all $i=1,\ldots,k$.
\end{enumerate}
Indeed, for each edge $e$ of $\KRGraphf$ choose a Dehn twist $\tau_e$ along $e$.
Then one can assume that each $\tau_i$ is a product of some $\tau_e$.
This will guarantee condition~\ref{enum:ex_sect:taui}.
Existence of $\Qman$ and isotopies $H^i$ as in~\ref{enum:ex_sect:Qinf} and~\ref{enum:ex_sect:isot} is proved in~\cite[Theorem~6]{Maksymenko:ProcIM:ENG:2010} for $\XSman=\varnothing$.
That sets also suffice if $\XSman\not=\varnothing$ is finite but $|\XSman| < \chi(\Mman)$.
E.g. if $\Mman=\Disk$, then one can take $\Qman=\partial\Disk$.
If $|\XSman|>\chi(\Mman)$, then one can put $\Qman=\XSman$.

\smallskip

Now $\zeta$ will be constructed by induction on $d$-dimensional faces of $I^k$.

\underline{Let $d=0$}.
Put
\[ \zeta( \eps_1, \ldots, \eps_{k}) = \tau_1^{\eps_1} \circ \cdots \circ \tau_k^{\eps_k} = \prod_{i=1}^{k}\tau_i, \qquad \eps_i \in \{0,1\}. \]
Since $\tau_i$ pairwise commute, one can write the product of $\tau_i$ in any order, and therefore the notation $\prod_{i=1}^{k}\tau_i$ is well-defined.
Then
\[
\func \circ \zeta( \eps_1, \ldots, \eps_{k}) = \func \circ \tau_1^{\eps_1} \circ \cdots \circ \tau_k^{\eps_k} = \func.
\]

\underline{Let $d=1$}.
Every $1$-dimensional cell has the form $\eface{\{i\}}{\Bman}{\Cman}$, and so it is ``parametrized'' by its $i$-th coordinate $t_i$.
Then we define $\zeta$ on $\eface{\{i\}}{\Bman}{\Cman}$ by
\[
    \zeta(t_i) = \bigl( \prod_{j \in \Cman}\tau_j \bigr) \circ H^i_{t_i}.
\]
In other words, we first apply an isotopy $H^i$, and if $j$-th coordinate of the face is $1$ (for $j\not=i$), then we additionally make a composition with $\tau_j$.
Since all $\tau_i$ pairwise commute, $\zeta$ is correctly defined on $1$-skeleton of $I^k$.

It follows that
\[
\func \circ \zeta(t_i) = \func\circ H^i_{t_i}, \qquad t_i \in \eface{\Aman}{\Bman}{\Cman}.
\]

For example, if $k=2$, then
\begin{align*}
    \zeta(t_1,0) &= H^1_{t}, &
    \zeta(t_1,1) &= \tau_2 \circ H^1_{t}, &
    \zeta(0,t_2) &= H^2_{t}, &
    \zeta(1,t_2) &= \tau_1\circ H^2_{t}.
\end{align*}

\underline{Let $d=2$}.
Let $i_1<i_2$ be two distinct indices and
\[
    E = \eface{\{i_1,i_2\}}{\Kman \setminus \{i_1,i_2\}}{\varnothing}
\]
be a $2$-face of $I^k$ in which all coordinates except for $i_1$ and $i_2$ are zeros.

By the construction $\zeta(\partial E) \in \DiffId(\Mman,\XSman\cup\Qman)$, and the latter space is contractible.
Hence $\zeta$ extends to some continuous map
\[ H^{i_1,i_2}: E \to \DiffId(\Mman,\XSman\cup\Qman).\]

Extend $\zeta$ to each other $2$-face $\eface{\{i_1,i_2\}}{\Bman}{\Cman}$ spanned by coordinates $i_1,i_2$ in a similar way:
\[
    \zeta(t_{i_1}, t_{i_2}) =  \bigl( \prod_{j \in \Cman}\tau_j \bigr) \circ H^{i_1,i_2}(t_{i_1}, t_{i_2}).
\]
Then again
\[
\func\circ \zeta(t_{i_1}, t_{i_2}) = \func\circ H^{i_1,i_2}(t_{i_1}, t_{i_2})
\]
for any $2$-face $\eface{\{i_1,i_2\}}{\Bman}{\Cman}$.

\underline{Let $d=3$}.
The construction is similar to the case $d=2$.
Let $i_1<i_2<i_3$ be three distinct indices and
\[
    E = \eface{\{i_1,i_2,i_3\}}{\Kman \setminus \{i_1,i_2,i_3\}}{\varnothing}
\]
be a $3$-face of $I^k$ in which all coordinates except for $i_1$, $i_2$, and $i_3$ are zeros.

By the construction $\zeta(\partial E) \in \DiffId(\Mman,\XSman\cup\Qman)$, and the latter space is contractible.
Hence $\zeta$ extends to some continuous map
\[ H^{i_1,i_2,i_3}: E \to \DiffId(\Mman,\XSman\cup\Qman)\]
and we extend $\zeta$ to each other $3$-face $\eface{\{i_1,i_2,i_3\}}{\Bman}{\Cman}$ spanned by coordinates $i_1,i_2,i_3$ in a similar way:
\[
    \zeta(t_{i_1}, t_{i_2}, t_{i_3}) =
    \bigl( \prod_{j \in \Cman}\tau_j \bigr) \circ H^{i_1,i_2,i_3}(t_{i_1}, t_{i_2}, t_{i_3}).
\]

Applying the same arguments so on, we extend $\zeta$ to all of $I^k$.

{\em Verification of properties of $\eta$.}
By construction the restriction of $\zeta$ to each $1$-face spanned by coordinate $t_i$, is given by the loop $\gamma_i :=\func\circ H^i$.
But due to description of $\partial_1$ (see Remark~\ref{rem:construction_partial_1}), $\partial_1([\gamma_i]) = H^i_1 = \tau_i$.
This means commutativity of~\eqref{equ:map_of_torus}.
\end{proof}

It is convenient to formulate the general observation about the homotopy types of $\Orbit{\func,\XSman}$ in the form of the following property~\ref{prop:bib}.

\begin{definition}
Let $\Mman$ be a compact surface, $\func\in\FSP{\Mman}{\Pman}$, $\XSman$ be an $\func$-adapted submanifold, and $\GrpKRIsotId{\func,\XSman}$ be the group of automorphisms of the graph of $\func$ induced by elements of $\StabilizerIsotId{\func,\XSman}$, see~\eqref{equ:GKRfX}.
Say that a pair $(\func,\XSman)$ has property~\ref{prop:bib} whenever
\begin{enumerate}[leftmargin=*, label={\rm(Bib)}]
\item\label{prop:bib}
there exists a \myemph{free} action of the (finite) group $\GrpKRIsotId{\func,\XSman}$ on a $k$-torus $\torus{k}$ for some $k\geq0$ such that we have the following homotopy equivalence:
\[
    \OrbitPathComp{\func,\XSman}{\func} \,\homeq\, \DiffId(\Mman,\XSman) \times \bigl( \torus{k}/\GrpKRIsotId{\func,\XSman} \bigr),
\]
where $\torus{k}/\GrpKRIsotId{\func,\XSman}$ is the corresponding quotient space.
\end{enumerate}
\end{definition}

\begin{subremark}\label{rem:prop:bib}
Evidently, Property~\ref{prop:bib} generalizes the effect described in Theorem~\ref{th:sect_partial1} for the case when $\GrpKRIsotId{\func,\XSman}$ is trivial.
That property was discovered by E.~Kudryavtseva in a series of papers~\cite{Kudryavtseva:SpecMF:VMU:2012, Kudryavtseva:MathNotes:2012, Kudryavtseva:MatSb:ENG:2013}, in which she described the homotopy structure of the spaces $\mathcal{F}_{p,q,r}$ of Morse functions on orientable compact surfaces equipped with a numeration of their critical points (such spaces are finite covering spaces of the spaces of all Morse functions).
She presented a ``finite-dimensional'' model for the stratifications of those spaces into orbits of Morse functions, and demonstrated Property~\ref{prop:bib} for them.
\end{subremark}

\begin{subremark}
Notice that not all pairs $(\func,\XSman)$ have property~\ref{prop:bib}.
For instance, if $\func:S^2\to\bR$ is a Morse function with two only critical points as in Theorem~\ref{th:stabilizer}\ref{enum:specfunc:sphere}, then $\pi_1\OrbitPathComp{\func}{\func} = \pi_1 S^2 = 0$.
On the other hand, the graph of $\func$ is a closed segment, whence $\GrpKRIsotId{\func}$ is trivial, and therefore
\[ \pi_1\bigl( \DiffId(S^2) \times (\torus{k}/\GrpKRIsotId{\func} ) \bigr) \cong \bZ_2 \times \bZ^k \not = 0 \ \text{for any $k\geq0$}. \]
\end{subremark}

Nevertheless this property seems to be very typical.
The following theorem describes several cases when~\ref{prop:bib} holds.
\begin{theorem}\label{th:cases:Bib:prop}
Let $\Mman$ be a compact connected surface, $\func\in\FSP{\Mman}{\Pman}$, and $\XSman$ be an $\func$-adapted submanifold.
Then the pair $(\func,\XSman)$ has property~\ref{prop:bib} in the following cases.
\begin{enumerate}[leftmargin=*, label={\rm(\arabic*)}, itemsep=0.8ex]
\item\label{enum:th:prop:TG:generic}
$\StabilizerId{\func,\XSman}$ is contractible, and $\GrpKRIsotId{\func,\XSman}$ is trivial {\rm(Theorem~\ref{th:sect_partial1})}.

\item\label{enum:th:prop:TG:Kudr}
$\Mman$ is orientable, $\XSman=\varnothing$, and $\func\in\Mrs{\Mman}{\bR}$ is such that the Euler characteristic $\chi(\Mman)$ of $\Mman$ is \myemph{less than} the total number of fixed critical points of $\StabilizerIsotId{\func}$ {\rm(E.~Kudryavtseva~\cite{Kudryavtseva:SpecMF:VMU:2012, Kudryavtseva:MathNotes:2012, Kudryavtseva:MatSb:ENG:2013})}.

\end{enumerate}
\end{theorem}

\begin{smallproof}{Warning}
In~\cite[Corollary~1.3 \& Theorem~5.10]{Maksymenko:TA:2020} it was claimed that Property~\ref{prop:bib} holds for pairs $(\func,\XSman)$ admitting a diagonal short exact sequence from~\eqref{equ:diagram_3x3_sdo_b:reduced}.
However, such a statement is not in fact proved.
The reason is that the above claims were just direct applications of \cite[Theorem~2.5]{Maksymenko:TA:2020} being in turn an \myemph{incorrect} formulation of a realization theorem for crystallographic groups.
That \cite[Theorem~2.5]{Maksymenko:TA:2020} states that for \myemph{every} short exact sequence $\bZ^k \monoArrow \Gamma \epiArrow G$ with finite $G$ and torsion free $\Gamma$, there exists a free action of $G$ on $\torus{k}$ such that $\pi_1(\torus{k}/G) \cong \Gamma$.
In fact, in general one should additionally require that $\bZ^k$ is a \myemph{maximal abelian subgroup of $\Gamma$} (this is a theorem by H.~Zassenhaus, see e.g.~\cite[Theorem~2.2]{Szczepanski:ADM:2012}).

Nevertheless this does not mean that \cite[Corollary~1.3 \& Theorem~5.10]{Maksymenko:TA:2020} are wrong.
Indeed, consider the following short exact sequence $\seqZ{n}: n\bZ \monoArrow \bZ \epiArrow \bZ_n$.
Then $n\bZ$ is not a maximal abelian subgroup of $\bZ$, however there exists an action of $\bZ_n$ on $\torus{1} = \Circle$ such that $\pi_1(\torus{1}/\bZ_3) \cong \bZ$.

Therefore to complete the proofs of \cite[Corollary~1.3 \& Theorem~5.10]{Maksymenko:TA:2020} one needs to thoroughly check conditions on $(\func,\XSman)$ when such actions of $\GrpKRIsotId{\func,\XSman}$ on some tori exist.
This will be done in another paper.
\end{smallproof}

In next sections we will describe algebraic structure of $\pi_1\OrbitPathComp{\func,\XSman}{\func}$ for the case when $\Mman$ is orientable and distinct from $S^2$.

\section{Maps on surfaces with $\chi(\Mman)<0$}\label{sect:chi_neg}
Let $\Mman$ be compact surface, $\func\in\FSP{\Mman}{\Pman}$, and $\XSman$ an $\func$-adapted submanifold.
In this section we give a complete description of algebraic structure of Bieberbach sequence~\eqref{equ:DSG_sequence_fX}:
\[
    \seqStab{\func,\XSman}: \quad \pi_0\FolStabilizer{\func,\XSman} \monoArrow \pi_0\Stabilizer{\func,\XSman} \epiArrow \pi_0\GrpKR{\func,\XSman}.
\]
for the case when $\Mman$ is orientable and differs from $S^2$.

\subsection{Reduction to the case $\chi(\Mman)\geq0$}
Theorem~\ref{th:stab:chi_neg} below shows that computations of $\seqStabIsotId{\func,\XSman}$ completely reduces to the case when $\chi(\Mman)\geq0$ (i.e. is one of the following surfaces: $\Sphere$, $\Torus$, $\Disk$, $\Circle\times[0,1]$, $\PrjPlane$, M\"obius band and Klein bottle) and $\XSman$ consists of critical points of $\func$ and boundary components of $\Mman$.

Let $\regN{\XSman}$ be a $\func$-regular neighborhood of $\XSman$ and $D_1,\ldots,D_q$ all the connected components of the closure $\overline{\Mman\setminus\regN{\XSman}}$ which are diffeomorphic with a $2$-disk.
Then the union
\[ \canN{\XSman}: = \regN{\XSman} \cup D_1\cup\ldots\cup D_q \]
is called a \myemph{canonical neighborhood} $\canN{\XSman}$ of $\XSman$ (corresponding to $\regN{\XSman}$), see~\cite{JacoShalen:Topology:1977}.

\begin{subtheorem}\label{th:stab:chi_neg}{\rm(\cite{Maksymenko:MFAT:2010}, \cite[Theorem~5.4]{Maksymenko:TA:2020})}
Suppose $\Mman$ is connected with $\chi(\Mman)<0$.
Let also $\crLev$ be the union of all non-extremal critical leaves of $\func$ \myemph{whose canonical neighborhoods have negative Euler characteristic}, $\regN{\crLev}$ an $\func$-regular neighborhood of $\crLev$, $\Bman_1,\ldots,\Bman_{\cnt}$ all the connected components of $\overline{\Mman\setminus\regN{\crLev}}$, and $\hXman_i := \Bman_i \cap (\XSman\cup\regN{\crLev})$, $i=1,\ldots,\cnt$.
Then
\begin{enumerate}[leftmargin=5ex, topsep=0.8ex, itemsep=0.8ex, label=$(\arabic*)$]
\item\label{enum:pi0SfX:decomp:1}
each $\Bman_i$ is diffeomorphic either with a $2$-disk or a cylinder or a M\"obius band;
\item\label{enum:pi0SfX:decomp:2}
the inclusion $\bigl( \StabilizerIsotId{\func,\XSman\cup\regN{\crLev}}, \FolStabilizerIsotId{\func,\XSman\cup\regN{\crLev}} \bigr) \subset \bigr(\StabilizerIsotId{\func, \XSman}, \FolStabilizerIsotId{\func, \XSman}\bigr)$ is a homotopy equivalence, inducing therefore an isomorphism
\[
    \seqStabIsotId{\func,\XSman} \cong \seqStabIsotId{\func,\XSman\cup\regN{\crLev}} \cong \myprod\limits_{i=1}^{\cnt}\seqStabIsotId{\func|_{\Bman_i},\hXman_i}.
\]
\end{enumerate}
\end{subtheorem}

\subsection{Relation between sequences $\seqStabIsotId{\func,\XSman}$ and $\seqStab{\func}$}
Denote by
\[ \jInclZ:\pi_0\StabilizerIsotId{\func,\XSman} \to \pi_0\StabilizerIsotId{\func}\]
the homomorphism induced by the inclusion $\jIncl:\StabilizerIsotId{\func,\XSman} \subset \StabilizerIsotId{\func}$.

\begin{sublemma}[\!\!{\cite[Lemma~5.1]{Maksymenko:TA:2020}}]\label{lm:incl_SprfX_Sprf}
Suppose $\Mman$ is connected and $\XSman$ is a \myemph{non-empty} union of several boundary components of $\Mman$.
Suppose also that $\StabilizerId{\func}$ is contractible.
\begin{enumerate}[wide, label={\rm(\roman*)}, itemsep=0.8ex]
\item\label{enum:xx:DOS__jSdS}
Then we have the following commutative diagram:
{\small\begin{equation}\label{equ:DOS__jSdS}
\begin{aligned}
\xymatrix{
    \pi_1\DiffId(\Mman) \ar@{^(->}[r]^-{p_1} \ar@/_2.8pc/[ddrrr]_-{\alpha}  &
    \pi_1\OrbitPathComp{\func}{\func}  \ar@{->>}[rr]^-{\partial} \ar[drr]^-{\beta} &&
    \pi_0\StabilizerIsotId{\func}   \\
    &
    \pi_1\OrbitPathComp{\func,\XSman}{\func}
    \ar[u]_-{\cong}^{\text{\rm Theorem~\ref{th:SDO_lt_fibr}\ref{enum:th:SDO_lt_fibr:2.1}}}
    \ar@{->>}[rr]^-{\cong}_-{\eqref{equ:pi1OfX_pi0SfX}} &&
    \pi_0\StabilizerIsotId{\func,\XSman} \ar@{->>}[u]_-{\jInclZ} \\
    & && \ker(\jInclZ) \ar@{^(->}[u]
}
\end{aligned}
\end{equation}}%
In particular, it induces an isomorphism between the right column and the upper row which coincides with the sequence~\eqref{equ:exact_seq_for_pi1OfX}.

\item\label{enum:xx:3x3_diagram}
We also have the following exact $(3\times3)$-diagram:
{\small\begin{equation}\label{equ:3x3_diagram:ker_j}
\begin{gathered}
\xymatrix@R=3ex{
& & \ker(\jInclZ) \ \ar@{^{(}->}[d]^(.25){} \ar@{=}[r]     &
\ \ker(\jInclZ) \ \ar@{^{(}->}[d] \ar[r]                   &
\ \{1\} \ \ar@{^{(}->}[d] \\
 \seqStabIsotId{\func,\XSman}& \hspace{-1cm}: &
\ \pi_0\FolStabilizerIsotId{\func,\XSman} \ \ar@{->>}[d] \ar@{^{(}->}[r] &
\ \pi_0\StabilizerIsotId{\func,\XSman} \ \ar@{->>}[d]^{\jInclZ} \ar@{->>}[r]^-{\fc_\XSman} &
\ \GrpKRIsotId{\func,\XSman} \ar[d]_-{\cong}^-{\GHomIsotId{\XSman,\varnothing}} \\
 \seqStabIsotId{\func}& \hspace{-1cm}: &
\ \pi_0\FolStabilizerIsotId{\func} \ \ar@{^{(}->}[r] \ar@/^5pt/[u]^{\xi}            &
\ \pi_0\StabilizerIsotId{\func} \ \ar@{->>}[r]^-{\fc} &
\ \GrpKRIsotId{\func}
}
\end{gathered}
\end{equation}}%
Since the groups $\pi_0\FolStabilizerIsotId{\func,\XSman}$ and $\pi_0\FolStabilizerIsotId{\func}$ are free abelian, the left column splits, so there is a section
$\xi:\pi_0\FolStabilizerIsotId{\func} \to \pi_0\FolStabilizerIsotId{\func,\XSman}$.

\item\label{enum:xx:chiM_neg}
If $\chi(\Mman)<0$, and thus $\DiffId(\Mman)$ is contractible, then by~\eqref{equ:DOS__jSdS}, $\jInclZ$ is an isomorphism, and~\eqref{equ:3x3_diagram:ker_j} yields an isomorphism  $\seqStabIsotId{\func,\XSman}\cong \seqStabIsotId{\func}$.

\item\label{enum:xx:chiM_nonneg}
If $\chi(\Mman)\geq0$, so $\Mman$ is either a $2$-disk, or a cylinder, or a M\"obius band (since $\partial\Mman\not=\varnothing$), then $\ker(\jInclZ) \cong \pi_1\DiffId(\Mman)\cong\bZ$ due to~\eqref{equ:DOS__jSdS}, and the diagram~\eqref{equ:3x3_diagram:ker_j} is a short exact sequence $\seqZ{1} \monoArrow \seqStabIsotId{\func,\XSman} \epiArrow \seqStabIsotId{\func}$.

In fact, for $\Mman$ being a $2$-disk of a M\"obius band, then $\ker(\jInclZ)$ is generated by a Dehn twist along boundary component of $\Mman$, and for $\Mman=\Circle\times[0,1]$ the kernel $\ker(\jInclZ)$ is generated by a pair of Dehn twists in different directions along boundary components of $\Mman$.
\end{enumerate}
\end{sublemma}

In the next two sections we will describe the algebraic structure of $\seqStab{\func,\XSman}$ for the cases when $\Mman$ is a $2$-disk, cylinder and torus.
Together with Theorem~\ref{th:stab:chi_neg} this will give complete information on $\seqStab{\func,\XSman}$ for all orientable surfaces distinct from $\Sphere$.

\section{Maps on $2$-disk and cylinder}\label{sect:maps_disk_cyl}
Suppose $(\Mman,\XSman)$ is one of the pairs $(\Disk,\partial\Disk)$ or $(\Circle\times\UInt, \Circle\times 0)$. 
Then the group $\Diff(\Mman, \XSman)$ is connected (in fact even contractible, see~Table~\ref{tbl:hom_type_DidMX}).
Hence $\Diff(\Mman, \XSman) = \DiffId(\Mman, \XSman)$ and
\begin{equation}\label{equ:StabIsotIdD2}
 \StabilizerIsotId{\func,\XSman} = \Stabilizer{\func,\XSman} \cap\DiffId(\Mman, \XSman)= \Stabilizer{\func,\XSman} \cap\Diff(\Mman, \XSman) = \Stabilizer{\func,\XSman}.
\end{equation}
Similarly, $\FolStabilizerIsotId{\func,\XSman} = \FolStabilizer{\func,\XSman}$.
This implies that $\seqStab{\func,\XSman} = \seqStabIsotId{\func,\XSman}$.

Recall that if $\StabilizerId{\func}$ is contractible, then by Lemma~\ref{lm:incl_SprfX_Sprf}\ref{enum:xx:chiM_nonneg} we also have the following short exact sequence:
\begin{equation}\label{equ:z1_bfdM_bf}
    \seqZ{1} \monoArrow \seqStabIsotId{\func,\XSman} \epiArrow \seqStabIsotId{\func}
\end{equation}
which is the same as the diagram~\eqref{equ:3x3_diagram:ker_j}.

\subsection{Maps with minimal number of critical points}

\begin{subtheorem}\label{th:stab:cylinder}{\rm(\cite[Theorem~5.5]{Maksymenko:TA:2020})}
Let $(\Mman,\XSman) = (\Circle\times\UInt, \Circle\times0)$.
\begin{enumerate}[leftmargin=*, label={\rm(\arabic*)}, itemsep=0.8ex]
\item\label{eqnu:th:cyl:func:incl}
For each $\func\in\FF(\Mman,\Pman)$ the inclusion of pairs
\begin{multline*}
\bigl(
\StabilizerIsotId{\func, \partial\Mman},\,  \FolStabilizerIsotId{\func, \partial\Mman}
\bigr)
 \subset
\bigl( \StabilizerIsotId{\func, \XSman},\, \FolStabilizerIsotId{\func, \XSman} \bigr)
 \stackrel{\eqref{equ:StabIsotIdD2}}{\equiv}
\bigl( \Stabilizer{\func, \XSman},\, \FolStabilizer{\func, \XSman} \bigr)
\end{multline*}
is a homotopy equivalence, whence $\seqStabIsotId{\func,\partial\Mman} \cong \seqStabIsotId{\func,\XSman} \equiv \seqStab{\func,\XSman}$.

\item\label{eqnu:th:cyl:func:no_cr_pt}
If $\func\in\FF(\Mman,\Pman)$ has \myemph{no critical points}, then
\begin{align*}
    \pi_0\StabilizerIsotId{\func,\XSman}  &= 0, &
    \pi_0\Stabilizer{\func,\partial\Mman} &=\bZ.
\end{align*}
In particular, $\sDSG{\func,\XSman} \cong \seqTriv: \{1\} \monoArrow \{1\} \epiArrow \{1\}$.
\end{enumerate}
\end{subtheorem}

\begin{subtheorem}\label{th:stab:disk:one_crpt}{\rm(\cite[Theorem~5.6]{Maksymenko:TA:2020})}
Suppose $\func\in\FF(\Disk,\Pman)$ has a \myemph{unique critical point} $z$ being therefore a \myemph{local extreme}.
\begin{enumerate}[label={\rm(\arabic*)}, itemsep=0.8ex, topsep=0.8ex]
\item\label{enum:2disk:1crpt:nondeg}
If $z$ is \myemph{non-degenerate}, then
$\sDSG{\func,\partial\Disk} \cong \seqTriv: \{1\} \monoArrow \{1\} \epiArrow \{1\}$.
\item\label{enum:2disk:1crpt:deg}
If $z$ is \myemph{degenerate} of symmetry index $m$, then $\StabilizerId{\func}$ is contractible, and the sequence~\eqref{equ:z1_bfdM_bf} is isomorphic with $(3\times3)$-diagram $\qwSeq{\seqZ{m}}$, see~\eqref{equ:u_div_u}:
\[
 \xymatrix@C=2.2ex@R=2ex{
 \seqZ{1}\ar@{^(->}[d]                                          & \hspace{-1cm}: & m\bZ \ar@{^(->}[d] \ar@{=}[r]    & m\bZ \ar@{^(->}[d] \ar@{->>}[r] & 0 \ar@{^(->}[d] \\
 \seqStabIsotId{\func,\partial\Disk}\cong \seqZ{m} \ar@{->>}[d] & \hspace{-1cm}: & m\bZ \ar@{->>}[d]  \ar@{^(->}[r] & \bZ  \ar@{->>}[d]  \ar@{->>}[r] & \bZ_m \ar@{->>}[d] \\
 \seqStabIsotId{\func}                                          & \hspace{-1cm}: & 0                  \ar@{^(->}[r] & \bZ_m              \ar@{=}[r]   & \bZ_m
 }
\]
\end{enumerate}
\end{subtheorem}

\subsection{General case}\label{sect:disk_cyl:gen_case}
It follows from these theorems that
\begin{equation}\label{equ:bfMdM}
\seqStab{\func,\XSman}
\stackrel{\eqref{equ:StabIsotIdD2}}{\cong}
\seqStabIsotId{\func,\XSman}
\cong
\seqStabIsotId{\func,\partial\Mman},
\end{equation}
where the second and third sequences coincide if $\Mman=\Disk$ and are isomorphic for $\Mman=\Circle\times[0,1]$ by Theorem~\ref{th:stab:cylinder}\ref{eqnu:th:cyl:func:incl}.
Our aim is to compute the sequence~\eqref{equ:bfMdM}.

Let $\crLev$ be a unique critical leaf of $\func$ such that the connected component of $\Mman\setminus\crLev$ containing $\XSman$ includes no critical points of $\func$. 
Equivalently, let $v$ be the vertex of $\KRGraphf$ of $\func$ corresponding to $\XSman$.
Then $v$ belongs to a unique edge $e$ of $\KRGraphf$, and $\crLev$ is the leaf of $\func$ corresponding to another vertex of $e$.

It follows from uniqueness of $\crLev$ that
\begin{equation}\label{equ:hK_K}
\dif(\crLev) =\crLev, \qquad \forall\ \dif\in\Stabilizer{\func,\XSman}.
\end{equation}

Let $c=\func(\crLev) \in \Pman$ and $\eps>0$.
Denote by $\regNK$ the connected component of $\func^{-1}[c-\eps,c+\eps]$ containing $\crLev$.
Decreasing $\eps$ we can assume that $\regNK \setminus \crLev$ contains no critical points of $\func$ and $\regNK \cap \partial\Mman \subset\partial\regNK$ (this intersection may be empty).
In particular, $\regNK$ is an $\func$-regular neighborhood of $\crLev$.

Let $\bZman$ be the collection of all connected components of $\overline{\Mman\setminus\regNK}$.
Since by~\eqref{equ:hK_K} $\crLev$ is $\Stabilizer{\func,\XSman}$-invariant, we have that so are $\regNK$ and $\overline{\Mman\setminus\regNK}$.
In other words, we get a natural action of $\Stabilizer{\func,\XSman}$ on $\bZman$.
Let
\begin{equation}\label{equ:StabZ}
 \Stabilizer{\bZman} = \{ \dif\in\Stabilizer{\func,\XSman} \mid \dif(\Zman)=\Zman \ \text{for each} \ \Zman\in\bZman\}
\end{equation}
be its kernel of non-effectiveness.
Then the quotient $\Stabilizer{\func,\XSman}/\Stabilizer{\bZman}$ \myemph{effectively} acts on $\bZman$.

\begin{itemize}[label=$\bullet$, wide, topsep=0.8ex, itemsep=0.8ex]
\item
Let $\bZmanX=\{ \XFixA, \Xman_1,\ldots,\Xman_a\}$ be all the elements of $\bZman$ invariant under all diffeomorphisms from $\Stabilizer{\func,\XSman}$, i.e. fixed points of $\Stabilizer{\func,\XSman}/\Stabilizer{\bZman}$. 
Enumerate them so that $\XSman\subset \XFixA$, and in particular, $\XFixA$ is always a cylinder.
Evidently, if $\Mman=\Disk$, them all $\Xman_i$, $i\geq1$, are $2$-disks.
On the other hand, if $\Mman=\Circle\times\UInt$, then the element of $\bZman$ containing another boundary component $\Circle\times1$ is invariant under $\Stabilizer{\func,\XSman}$ and we will always denote it by $\XFixB$.
In this case $\XFixB$ is a cylinder, and all others $\Xman_i$, $i\geq2$, must be $2$-disks.

\item
Let also $\bZmanY:=\bZman\setminus\bZmanX = \{ \Yman_1,\ldots,\Yman_b \}$ be all other $2$-disks of $\bZman$.
Thus each $\Yman_i$ is not invariant under some element of $\Stabilizer{\func,\XSman}$.
\end{itemize}

\smallskip

The following theorem expresses $\seqStabIsotId{\func,\partial\Mman}$ via the corresponding sequences $\{ \seqStabIsotId{\func|_{\Zman},\partial\Zman} \}_{\Zman\in\bZman}$.
It is essentially based in Lemma~\ref{lm:charact_seq_wrm},
\begin{subtheorem}\label{th:stab:disk_ann:gen_case}
\begin{enumerate}[wide, label={\rm(\Alph*)}, topsep=0.8ex, itemsep=0.8ex]
\item\label{enum:SS:A}
Suppose that all elements of $\bZman$ are invariant under $\Stabilizer{\func,\XSman}$, which is equivalent to either of the following three conditions:
\begin{align*}
    \bZmanY&=\varnothing, &
    \bZman&=\bZmanX=\{\XFixA,\Xman_1,\ldots,\Xman_a\}, &
    \Stabilizer{\func,\XSman}&=\Stabilizer{\bZman}.
\end{align*}
Then $\seqStabIsotId{\func} \cong \mprod\limits_{i=1}^{a}\seqStabIsotId{\func|_{\Xman_i},\partial\Xman_i}:$
\begin{equation}\label{equ:DOS__jSdS_case_A:no_bd}
\mprod\limits_{i=1}^{a}\FolStabilizerIsotId{\func|_{\Xman_i},\partial\Xman_i}
    \ \monoArrow \
    \mprod\limits_{i=1}^{a}\StabilizerIsotId{\func|_{\Xman_i},\partial\Xman_i}
    \ \epiArrow \
    \mprod\limits_{i=1}^{a}\GrpKRIsotId{\func|_{\Xman_i},\partial\Xman_i},
\end{equation}
and the sequence~\eqref{equ:z1_bfdM_bf}: $\seqZ{1} \monoArrow \seqStabIsotId{\func,\XSman} \epiArrow \seqStabIsotId{\func}$ is isomorphic with a split sequence~\eqref{equ:split_seq}:
$\seqZ{1} \monoArrow \seqZ{1} \times \seqStabIsotId{\func}  \epiArrow \seqStabIsotId{\func}$.
In particular,
\begin{multline*}
    \seqStabIsotId{\func,\XSman}\cong \seqZ{1} \times \seqStabIsotId{\func} \cong \seqStabIsotId{\func}  \times \seqZ{1} \cong \seqWrm{\seqStabIsotId{\func}}{1}: \\
    \phantom{AAA} \bigl(\mprod\limits_{i=1}^{a}\FolStabilizerIsotId{\func|_{\Xman_i},\partial\Xman_i}\bigr) \times \bZ
     \monoArrow
    \bigl( \mprod\limits_{i=1}^{a}\StabilizerIsotId{\func|_{\Xman_i},\partial\Xman_i}\bigr)  \times \bZ
     \epiArrow
    \mprod\limits_{i=1}^{a}\GrpKRIsotId{\func|_{\Xman_i},\partial\Xman_i},
\end{multline*}

\item\label{enum:SS:B}
Assume that $\bZmanY=\{\Yman_i\}_{i=1}^{b}\not=\varnothing$.
\begin{enumerate}[wide, label={\rm(\alph*)}, topsep=0.8ex, itemsep=0.8ex]
\item\label{enum::SS:B:SZ_Zm}
Then $\Stabilizer{\func,\XSman}/\Stabilizer{\bZman} \cong \bZ_m$ for some $m\geq2$;
\item\label{enum::SS:B:semifree}
The \myemph{effective} action of $\Stabilizer{\func,\XSman}/\Stabilizer{\bZman}$ on $\bZman$ is \myemph{semifree}, i.e.\! free on the set $\bZmanY$ of non-fixed elements, and
has exactly either \myemph{one} or \myemph{two} fixed elements.
\end{enumerate}

In particular, $m$ divides $b$ and that action has exactly $c:=b/m$ orbits.
Fix any ($2$-disks) $\Yman_1,\ldots,\Yman_c\in\bZmanY$ belonging to mutually distinct non-fixed $\bZ_m$-orbits and define the following short exact sequence:
\begin{align}\label{equ:DOS__jSdS_case_B}
\aSeq &:= \seqWrm{ \bigl( \myprod\limits_{j=1}^{c} \seqStabIsotId{\func|_{\Yman_i}, \partial\Yman_i} \bigr) }{m}: \\
 &
 \MResize{0.95\textwidth}{\Bigl(\mprod\limits_{i=1}^{a}\FolStabilizerIsotId{\func|_{\Xman_i},\partial\Xman_i}  \Bigr)^m \times \bZ^m
    \monoArrow
   \Bigl(\mprod\limits_{i=1}^{a}\StabilizerIsotId{\func|_{\Xman_i},\partial\Xman_i}  \Bigr)  \wrm{m} \bZ
    \epiArrow
   \mprod\limits_{i=1}^{a}\GrpKRIsotId{\func|_{\Xman_i},\partial\Xman_i}}. \nonumber
\end{align}
Then the following statements hold.
\begin{enumerate}[wide, label={\rm(\alph*)}, topsep=0.8ex, itemsep=0.8ex, resume]
\item\label{enum:SS:B:B0}
If $\Stabilizer{\func,\XSman}/\Stabilizer{\bZman}$ has a unique fixed element, i.e.~$\bZmanX=\{\XFixA\}$, then
\[ \seqStabIsotId{\func,\partial\Mman} \cong \aSeq. \]

\item\label{enum:SS:B:B0_X1}
Otherwise, due to~\ref{enum::SS:B:semifree}, $\Stabilizer{\func,\XSman}/\Stabilizer{\bZman}$ has exactly two fixed elements, so $\bZmanX=\{\XFixA, \XFixB\}$.
Let $\Cyli{1}=\overline{\Mman\setminus\XFixB}$.
Then
\[
    \bZmanX_{\Cyli{1}} = \bZmanX \setminus\{\XFixB\} = \{\XFixA\},
\]
whence the restriction $\func|_{\Cyli{1}}$ satisfies condition~\ref{enum:SS:B}\ref{enum:SS:B:B0}.
Therefore
\begin{align*}
    \seqStabIsotId{\func|_{\Cyli{1}},\partial\Cyli{1}} &\cong \aSeq, &
    \seqStabIsotId{\func,\partial\Mman} &\stackrel{\eqref{equ:bseq_two_cyl}}{\cong}\seqStabIsotId{\func|_{\Xman_1},\partial\Xman_1} \times \aSeq.
\end{align*}
\end{enumerate}
\end{enumerate}
\end{subtheorem}

\begin{subcorollary}[Decomposition into annuli and possibly one disk]\label{cor:stab:disk_ann:ann_decomp}
Let $\Mman$ be either a $2$-disk or a cylinder and $\func\in\FF(\Mman,\Pman)$.
Then there are $\func$-adapted subsurfaces $\Cyli{1},\ldots,\Cyli{n} \subset \Mman$ having the following properties.
\begin{enumerate}[label={\rm(\alph*)}, wide]
\item\label{enum:cor:split_into_cyls:1}
For $i=1,\ldots,n-1$ the surface $\Cyli{i}$ is a \myemph{cylinder}, while $\Cyli{n}$ is either a $2$-disk or a cylinder.
Moreover, the intersection $\Cyli{i}\cap \Cyli{j}$ for $i<j$ is non-empty only for $j=i+1$ and in this case it is a common boundary component of these subsurfaces.
Also $\Cyli{1}$ contains some boundary component of $\Mman$.

\item\label{enum:cor:split_into_cyls:2}
For each $i=1,\ldots,n-1$ there exist $m_i,c_i\geq1$ and certain $\func$-adapted mutually disjoint \myemph{$2$-disks} $\Yman_{i,1},\ldots,\Yman_{i,c_i} \subset \Cyli{i}$, such that we have an isomorphism
\begin{equation}\label{equ:seq:SDG:Cyl_i}
\seqStabIsotId{\func|_{\Cyli{i}}, \partial\Cyli{i}} \cong
\seqWrm{\biggl(\myprod\limits_{j=1}^{c_i} \seqStabIsotId{\func|_{\Yman_{i,j}}, \partial\Yman_{i,j}}\biggr)}{m_i},
\end{equation}
while the last sequence $\seqStabIsotId{\func|_{\Cyli{n}}, \partial\Cyli{n}}$ is isomorphic either with a sequence of type~\eqref{equ:seq:SDG:Cyl_i} or with $\seqTriv$ or with $\seqZ{m_n}$ for some $m_n\geq1$.

\item\label{enum:cor:split_into_cyls:3}
$\seqStabIsotId{\func, \partial\Mman}\cong \prod\limits_{i=1}^{n}\seqStabIsotId{\func|_{\Cyli{i}}, \partial\Cyli{i}}$.

\item\label{enum:cor:split_into_cyls:4}
The sequence~\eqref{equ:z1_bfdM_bf}: $\seqZ{1} \monoArrow \seqStabIsotId{\func,\partial\Mman} \epiArrow \seqStabIsotId{\func}$ is isomorphic with diagonal Garside sequence~\eqref{equ:diag_Garside_seq} for the sequences $\{\seqStabIsotId{\func|_{\Cyli{i}}, \partial\Cyli{i}}\}_{i=1}^{n}$.
\end{enumerate}
\end{subcorollary}

In order to formulate the general result about the structure of $\seqStabIsotId{\func,\partial\Mman}$ it will be convenient to introduce a series of classes of groups and short exact sequences, see Theorem~\ref{th:solvable_groups}.

\subsection{Classes of groups and short exact sequences}
A group $B$ is called \myemph{crystallographic} if it admits a short exact sequence: $A\monoArrow B\epiArrow C$ in which $C$ is finite and $A$ is a maximal abelian subgroup of $B$ being also free abelian.
A crystallographic group $B$ is \myemph{Bieberbach} if it is torsion free (i.e.~has no elements of finite order).

A short exact sequence $\uSeq: A\monoArrow B\epiArrow C$ will be called \myemph{nearly crystallographic} if $A\in\ccZ$, i.e.~it is a free abelian and $C$ is finite.
In this case we also say that $\uSeq$ is \myemph{crystallographic}, if $A$ is a maximal subgroup of $B$.
Moreover, a (nearly) crystallographic sequence $\uSeq: A\monoArrow B\epiArrow C$ will be called \myemph{(nearly) Bieberbach} if $B$ is torsion free.

For example, for $m\geq2$ the sequence $\seqZ{m}: m\bZ \monoArrow \bZ \epiArrow \bZ_m$ is \myemph{nearly Bieberbach} but \myemph{not Bieberbach}, since $m\bZ$ is not a maximal subgroup of $\bZ$, though the group $\bZ$ is Bieberbach.

A group $G$ is \myemph{solvable} if it has a finite increasing sequence of subgroups $1 = G_0 \vartriangleleft G_1 \vartriangleleft\cdots \vartriangleleft G_n = G$ such that $G_i$ is normal in $G_{i+1}$ and each quotient $G_{i+1}/G_i$ is abelian.
It is well known and is easy to check that in a short exact sequence of groups $A \monoArrow B\epiArrow C$ if any two groups are solvable, then so is the third one.

\subsubsection{Classes of groups}\label{sect:classes_M}
Consider the following classes of groups.
\begin{enumerate}[leftmargin=*, topsep=0.8ex, itemsep=1ex, label={$\bullet$}]
\item
Let $\ccZ = \{ \bZ^n \mid n\geq0\}$ be the set of all finitely generated free abelian groups.

\item
Let $\ccB$ be the \myemph{minimal} set of isomorphism classes of groups such that $\{1\} \in \ccB$ and if $A,B\in\ccB$ and $m\geq 1$, then $A\times B$ and $A \wrm{m} \bZ \in \ccB$ as well.

\item
Let $\clsBt$ be the \myemph{minimal} set of isomorphism classes of groups such that $\{1\} \in \clsBt$ and if $A,B\in\clsBt$ then $A\times B$ and $A \wrm{2} \bZ \in \ccB$.

\item
Let also $\ccP$ be the \myemph{minimal} set of isomorphism classes of groups such that $\{1\} \in \ccP$, and if $A,B\in\ccP$ and $m\geq 1$ then $A\times B$, $A \wr\bZ_m \in \ccP$ as well.

\item
Let $\clsGt$ be the \myemph{minimal} set of isomorphism classes of groups such that $\{1\} \in \clsGt$, and if $A,B\in\ccP$ then $A\times B$, $A \wr\bZ_2 \in \clsGt$.
\end{enumerate}

Evidently, $\clsBt \subset \ccB$ and $\clsGt \subset \ccP$.

It is easy to check that a group $G\in \ccB$ (resp.\ $\clsBt$, $\ccP$, $\clsGt$) if $G$ is obtained from the unit group by finitely many operations of direct products and wreath products of the form $\cdot\wrm{m}\bZ$ for some $m\geq1$ (resp.\ $\cdot\wrm{2}\bZ$ only, $\cdot\wr\bZ_{m}$ for some $m\geq1$, $\cdot\wr\bZ_{2}$ only).
For instance,
\begin{align*}
    \ccB  \ &\ni  \{1\}, \quad \bZ \cong \{1\}\wrm{m}\bZ, \quad \bZ^m, \quad ((\bZ^3\wrm{2}\bZ) \wrm{m}\bZ)\times (\bZ\wrm{18}\bZ), \\
    \clsBt \ &\ni  \{1\}, \quad \bZ^m, \quad ((\bZ^3\wrm{2}\bZ) \wrm{2}\bZ)\times (\bZ\wrm{2}\bZ), \\
    \ccP   \ &\ni  \{1\}, \quad \bZ_{m}, \quad ((\bZ_2^6\wr\bZ_8) \wr\bZ_{11})\times (\bZ_{3}\wr\bZ_{15}), \\
    \clsGt  \ &\ni  \{1\}, \quad \bZ_{2}, \quad ((\bZ_2^6\wr\bZ_2) \wr\bZ_{2})\times (\bZ_{2}\wr\bZ_{2}).
\end{align*}

\subsubsection{Classes of short exact sequences}
\begin{enumerate}[leftmargin=*, topsep=0.8ex, itemsep=1ex, label={$\bullet$}]
\item
Let $\gssZBP$ be the set of isomorphism classes of all \myemph{nearly crystallographic sequences $A\monoArrow B\epiArrow C$}, so $A\in\ccZ$, $B\in\ccB$, and $C\in\ccP$.

\item
Let $\ZZI = \{ (\seqZ{1})^m : \bZ^m = \bZ^m \epiArrow 1 \}_{m\geq0}$.

\item
Let also $\ssZBP$ be the \myemph{minimal} set of isomorphism classes of \myemph{short exact sequences} such that
$\seqTriv \in \ssZBP$, and if $\aSeq,\aSeq'\in\ssZBP$ and $m\in\bN$ then $\aSeq\times\aSeq'$ and $\seqWrm{\aSeq}{m} \in \ssZBP$ as well.

\item
Let also $\ssZBtPt$ be the \myemph{minimal} set of isomorphism classes of \myemph{short exact sequences} such that
$\seqTriv \in \ssZBtPt$, and if $\aSeq,\aSeq'\in\ssZBtPt$ then $\aSeq\times\aSeq'$ and $\seqWrm{\aSeq}{2} \in \ssZBtPt$ as well.
\end{enumerate}
Similarly, a short exact sequence $\aSeq$ belongs to $\ssZBP$ (resp.\ $\ssZBtPt$) iff it can be obtained from $\seqTriv$ by finitely many operations of products of sequences and wreath products of the form $\seqWrm{\cdot}{m}$, $m\in\bN$, (resp $m=2$ only).

The following lemma describes properties of the above classes.
\begin{sublemma}\label{lm:classes_SG}{\rm(\cite[Lemma~2.6]{Maksymenko:TA:2020})}
\begin{enumerate}[leftmargin=*, label={\rm(\arabic*)}, itemsep=0.8ex, topsep=0.8ex]
\item\label{enum:lm:classes_SG:solv_bib}
Every $B\in\ccB$ is solvable and nearly Bieberbach, and every $C\in\ccP$ is solvable and finite.

\item\label{enum:lm:classes_SG:seq_op:1}
$\seqZ{m} \in \gssZBP$ for all $m\geq0$.
\item\label{enum:lm:classes_SG:seq_op:2}
If $\aSeq_i:  \bZ^{k_i} \monoArrow B_i \epiArrow C_i$, $i=1,2$, are (nearly) crystallographic (resp.\ (nearly) Bieberbach) and $m\geq1$, then so are
\begin{align*}
\aSeq_1\times\aSeq_2 &: \bZ^{k_1+k_2} \monoArrow B_1\!\times\!B_2 \epiArrow C_1\!\times\!C_2, \\
\seqWrm{\aSeq_1}{m} &: \bZ^{k_1 m}\!\times\! m\bZ \monoArrow B_1\wrm{m}\bZ \epiArrow C_1\wr\bZ_m.
\end{align*}

\item\label{enum:lm:classes_SG:ZBP}
$\ssZBtPt \subset \ssZBP \subset \gssZBP$.
\end{enumerate}
\end{sublemma}

\begin{subtheorem}\label{th:solvable_groups}{\rm(\cite[Theorem~5.10]{Maksymenko:TA:2020})}
Let $\Mman$ be a connected orientable compact surface distinct from $S^2$ and $T^2$, $\func\in\FF(\Mman,\Pman)$, and $\XSman \subset\Mman$ an $\func$-adapted submanifold such that
\begin{enumerate}[leftmargin=*, label={\rm(\alph*)}]
\item
each connected component of $\XSman$ is of dimension $\geq1$, i.e.~it is not a critical point of $\func$;

\item
$\XSman\not=\varnothing$ whenever $\Mman=\Disk$ or $\Circle\times\UInt$, so $\chi(\Mman)<\ptnum{\XSman}$, whence  $\DiffId(\Mman,\XSman)$ is always contractible, and we have an isomorphism
\[
    \pi_1\OrbitPathComp{\func,\XSman}{\func} \xrightarrow[\cong]{~\partial_1~} \pi_0\StabilizerIsotId{\func,\XSman}.
\]
\end{enumerate}
Then the sequence
\begin{equation}\label{equ:final_SDG}
    \sDSG{\func,\XSman}:\pi_0\FolStabilizerIsotId{\func,\XSman}\monoArrow\pi_0\StabilizerIsotId{\func,\XSman}\epiArrow\GrpKRIsotId{\func,\XSman}
\end{equation}
belongs to the class $\ssZBP$.

In particular, $\pi_0\FolStabilizerIsotId{\func,\XSman}$ is a finitely generated free abelian group (which is already stated in Lemma~\ref{lm:pi0_Delta}\ref{enum:lm:pi0_Delta:4}), $\pi_1\OrbitPathComp{\func,\XSman}{\func}\cong \pi_0\StabilizerIsotId{\func,\XSman}$ are solvable nearly Bieberbach group, and $\GrpKRIsotId{\func,\XSman}$ is solvable and finite.

If $\func\in\MrsSmp{\Mman}{\Pman}$ is a simple Morse map, then $\sDSG{\func,\XSman} \in \ssZBtPt$.
\end{subtheorem}
\begin{smallproof}{Notes to the proof}
Assume that $\func\in\FSP{\Mman}{\Pman}$ has exactly $n$ critical points.

1) Suppose $\Mman$ is a $2$-disk or a cylinder.
If $\func$ has no saddle critical points, then by Theorems~\ref{th:stab:cylinder} and \ref{th:stab:disk:one_crpt} the sequence is either $\seqZ{0}$ or $\seqWrm{\seqZ{0}}{1}$.
On the other hand, if $\func$ has critical point, then Theorem~\ref{th:stab:disk_ann:gen_case} shows that $\seqStabIsotId{\func,\partial\Mman}$ expresses via analogous sequences $\seqStabIsotId{\func|_{\Xman_i},\partial\Xman_i}$ and $\seqStabIsotId{\func|_{\Yman_i},\partial\Yman_i}$ by operations of direct products and wreath product $\seqWrm{\cdot}{m}$ for some $m\geq1$.
Since each $\Xman_i$ and $\Yman_i$ is also a $2$-disk and cylinder and contain less number of critical points than $\func$, one can apply the induction on $n$, and obtain that $\seqStabIsotId{\func,\partial\Mman}$ is obtained from a sequence $\seqZ{0}$ by finitely many operations of direct products of short exact sequences and wreath products $\seqWrm{\cdot}{m}$ for some $m$.

2) If $\chi(\Mman)<0$, then by Theorem~\ref{th:stab:chi_neg} $\seqStabIsotId{\func,\partial\Mman}$ is again a product of sequences of the form $\seqStabIsotId{\func|_{\Bman},\partial\Bman}$, where $\Bman$ is either a $2$-disk or a cylinder, and therefore $\seqStabIsotId{\func,\partial\Mman}$ is also obtained from the sequence $\seqZ{0}$ by finitely many operations of direct products of short exact sequences and wreath products $\seqWrm{\cdot}{m}$ for some $m$.

If $\func$ is a generic Morse map, then an $\func$-regular neighborhood of each critical leaf of $\func$ is a disk with two holes, which leads to the observation that in Theorem~\ref{th:stab:disk_ann:gen_case} we always have that $m=2$.
This implies the last statement of Theorem~\ref{th:solvable_groups}.
\end{smallproof}

\newcommand\bdf[1]{\varepsilon_{#1}}
\subsection{Realization theorems for $\sDSG{\func,\XSman}$}
We discuss here the question when an abstract group $G$ can be realized as one of the groups~\eqref{equ:final_SDG} for some $(\Mman,\func,\Xman)$.
The presented results are obtained in the papers S.~Maksymenko and A.~Kravchenko~\cite{KravchenkoMaksymenko:EJM:2020}, B.~Feshchenko and A.~Kravchenko~\cite{KravchenkoFeshchenko:MFAT:2020}, I.~Kuznietsova and Yu.~Soroka~\cite{KuznietsovaSoroka:UMJ:2021}.

\subsubsection{Signs on the boundary}\label{sect:sing_on_bd}
Let $\Mman$ be a compact surface.
Then by definition each $\func\in\FSP{\Mman}{\Pman}$ takes constant values at connected components of $\partial\Mman$.
Fix an orientation on $\Pman$ (which is $\bR$ or $\Circle$).
Then for each boundary component $\Wman$ of $\Mman$ one can say whether $\func$ takes a local maximum or local minimum with respect to the orientation of $\Pman$.
Hence one can associate to $\func$ a function $\bdf{\func}: \pi_0\partial\Mman \to \{ \pm 1\}$ such that $\bdf{\func}(\Wman) = -1$ (resp $+1$) if $\func$ takes on $\Wman$ a local minimum (resp.\ maximum) with respect to the orientation of $\Pman$.

For every function $\eps: \pi_0\partial\Mman \to \{ \pm 1\}$ define the following four spaces:
\begin{align*}
\FSPE{\eps}{\Mman}{\Pman}     &= \{ \func\in \FSP{\Mman}{\Pman}  \mid \bdf{\func} = \eps\}, \\
\MrsE{\eps}{\Mman}{\Pman}     &=  \Mrs{\Mman}{\Pman}    \cap \FSPE{\eps}{\Mman}{\Pman}, \\
\MrsSmpE{\eps}{\Mman}{\Pman}  &=  \MrsSmp{\Mman}{\Pman} \cap \FSPE{\eps}{\Mman}{\Pman}, \\
\MrsGenE{\eps}{\Mman}{\Pman}  &=  \MrsGen{\Mman}{\Pman} \cap \FSPE{\eps}{\Mman}{\Pman}.
\end{align*}
Also for every subspace $\mathcal{X} \subset \FSP{\Mman}{\Pman}$ consider the following sets of isomorphism classes of groups
\begin{align*}
    \CStab{\mathcal{X}}   &= \{ \pi_0\StabilizerIsotId{\func,\partial\Mman} \mid \func\in\mathcal{X}   \}, &
    \CGrp{\mathcal{X}}    &= \{ \pi_0\GrpKRIsotId{\func,\partial\Mman} \mid \func\in\mathcal{X}   \},
\end{align*}
and the following set of isomorphism classes of short exact sequences:
\[
    \BSeq{\mathcal{X}} := \{ \sDSG{\func,\partial\Mman} \mid \func\in\mathcal{X}  \}.
\]

Then we have the following relations between those classes:
\begin{align*}
    &  \CStabMrsGen{\eps}{\Mman}{\Pman}    \subset \ccZ,
    && \CGrpKRMrsGen{\eps}{\Mman}{\Pman}   = \{1\}, \\
    &  \CStabMrsSmpE{\eps}{\Mman}{\Pman}  \subset \clsBt,
    && \CGrpKRMrsSmpE{\eps}{\Mman}{\Pman} \subset \clsGt, \\
    &  \CStabMrsE{\eps}{\Mman}{\Pman}   \subset \CStabFE{\eps}{\Mman}{\Pman} \subset \ccB,
    && \CGrpKRMrsE{\eps}{\Mman}{\Pman} \subset \CGrpKRFE{\eps}{\Mman}{\Pman} \subset \ccP
\end{align*}
\begin{align*}
  & \BSeqMrsGen{\eps}{\Mman}{\Pman}  \subset \ZZI, \\
  & \BSeqMrsSmpE{\eps}{\Mman}{\Pman} \subset \ssZBtPt, \\
  & \BSeqMrsE{\eps}{\Mman}{\Pman}    \subset \BSeqFE{\eps}{\Mman}{\Pman} \subset \ssZBP.
\end{align*}
The inclusions for generic Morse maps follow from Lemma~\ref{lm:GfX}, and all others from Theorem~\ref{th:solvable_groups}.

\begin{subtheorem}\label{th:realization:disk_cyl}
Let $\Mman$ be an orientable compact surface distinct from $\Torus$ and $\Sphere$, and $\eps:\pi_0\partial\Mman\to\{\pm1\}$ be any function.
\begin{enumerate}[leftmargin=*, label={\rm(\arabic*)}, topsep=0.8ex, itemsep=0.8ex]
\item\label{enum:th:realization_classes:1}
If $\Mman=\Circle\times[0,1]$, and $\eps$ takes the same value on both connected components of $\partial\Mman$, then
\begin{align*}
&   \ \ \ \CStabMrsSmpE{\eps}{\Mman}{\Pman}  = \clsBt \setminus \{1\},
&&  \ \ \ \CStabMrsE{\eps}{\Mman}{\Pman}     = \CStabFE{\eps}{\Mman}{\Pman}  =  \ccB \setminus \{1\}, \\
&   \ \ \ \CGrpKRMrsSmpE{\eps}{\Mman}{\Pman} = \clsGt,
&&  \ \ \ \CGrpKRMrsE{\eps}{\Mman}{\Pman}    = \CGrpKRFE{\eps}{\Mman}{\Pman} =  \ccP, \\
&   \ \ \ \BSeqMrsSmpE{\eps}{\Mman}{\Pman}   = \ssZBtPt \setminus\{\seqZ{0}\},
&&  \ \ \ \BSeqMrsE{\eps}{\Mman}{\Pman}      = \BSeqFE{\eps}{\Mman}{\Pman} = \\ &&&    \    \quad\qquad \qquad \qquad \qquad = \ssZBP \setminus \{\seqZ{0}\}.
\end{align*}
In particular, for every $\func\in\FSPE{\eps}{\Mman}{\Pman}$, the group $\pi_0\StabilizerIsotId{\func,\partial\Mman}$ is always non-trivial.

\item\label{enum:th:realization_classes:2}
In all other cases of $\Mman$ and $\eps$ we have that
\begin{align*}
&   \ \ \CStabMrsSmpE{\eps}{\Mman}{\Pman}  = \clsBt,
&&  \ \ \CStabMrsE{\eps}{\Mman}{\Pman}     = \CStabFE{\eps}{\Mman}{\Pman}  =  \ccB, \\
&   \ \ \CGrpKRMrsSmpE{\eps}{\Mman}{\Pman} = \clsGt,
&&  \ \ \CGrpKRMrsE{\eps}{\Mman}{\Pman}    = \CGrpKRFE{\eps}{\Mman}{\Pman} =  \ccP, \\
&   \ \ \BSeqMrsSmpE{\eps}{\Mman}{\Pman}   = \ssZBtPt,
&&  \ \ \BSeqMrsE{\eps}{\Mman}{\Pman}      = \BSeqFE{\eps}{\Mman}{\Pman} = \ssZBP.
\end{align*}
\end{enumerate}
\end{subtheorem}
\begin{smallproof}{Notes to the proof}
For groups $\CGrp{\mathcal{\Xman}}$ the proof is given in S.~Maksymenko and A.~Krav\-chenko~\cite{KravchenkoMaksymenko:EJM:2020}, and for groups $\CStab{\mathcal{\Xman}}$ by I.~Kuznietsova and Yu.~Soroka~\cite{KuznietsovaSoroka:UMJ:2021}.
Notice that in the case~\ref{enum:th:realization_classes:1} any map $\func\in\FSPE{\eps}{\Mman}{\Pman}$ has local minimum (or local maximum) on both boundary components of $\Mman$.
Therefore $\func$ must have saddle critical points inside $\Mman$, and therefore by Theorem~\ref{th:stab:disk_ann:gen_case} the group $\pi_0\StabilizerIsotId{\func,\partial\Mman}$ is always non-trivial.
This explains why one should remove the unit group $\{1\}$ from the classes in the first line of~\ref{enum:th:realization_classes:1}.

A thorough analysis of arguments in~\cite{KuznietsovaSoroka:UMJ:2021} shows that they actually contain realization theorems for sequences $\sDSG{\func,\partial\Mman}$.
This gives the corresponding relations for $\BSeq{\mathcal{X}}$.

The case $\Mman=\Torus$ is described in Theorem~\ref{th:classes_SG} below.
\end{smallproof}

\section{Maps on $2$-torus}\label{sect:maps_torus}

In this section we will describe the results of S.~Maksymenko, B.~Feshchenko, and A.~Kravchenko
\cite{MaksymenkoFeshchenko:UMZ:ENG:2014, Feshchenko:Zb:2015, MaksymenkoFeshchenko:MS:2015, MaksymenkoFeshchenko:MFAT:2015, Feshchenko:MFAT:2016, Feshchenko:PIGC:2019, KravchenkoFeshchenko:MFAT:2020, Feshchenko:PIGC:2021}
about algebraic structure of $\pi_1\OrbitPathComp{\func}{\func}$ for the maps $\func\in\FSP{\Torus}{\Pman}$.

Let $\func\in\FSP{\Torus}{\Pman}$.
If $\func$ has no critical points, then $\func$ is a locally trivial fibration $\Torus\to\Circle$ and the homotopy types of stabilizers and orbits for the pair $(\func,\varnothing)$ is described by Theorem~\ref{th:sdo:except_cases}\ref{enum:specfunc:torus}.

Assume that $\func$ has at least one critical point.
Then we have the following diagram~\eqref{equ:diagram_3x3_sdo_b:reduced} relating all the groups which we are interested in:
\begin{equation}\label{equ:diagram_3x3_sdo_b:reduced:torus}
    \aligned
    \xymatrix@R=3ex{
        \ \pi_1\Diff(\Torus)  \oplus \pi_0\FolStabilizerIsotId{\func} \ \ar@{->>}[rr]^-{\prj_2} \ar@{^(->}[rd]^-{\pj_1 \oplus \psi_1}  \ar@{->>}[d]_{\prj_1} & &
        \ \pi_0\FolStabilizerIsotId{\func} \ \ar@{^(->}[d]^{} \\
        \ \pi_1\Diff(\Torus) \ar@{^(->}[r]^-{\pj_1} \  &
        \ \pi_1\OrbitPathComp{\func}{\func} \ \ar@{->>}[r]^-{\partial_1} \ar@{^(->}[rd]_-{\rho\circ\partial_1} &
        \ \pi_0\StabilizerIsotId{\func} \ \ar@{->>}[d]^{\rho} \\
        &
        &
        \ \GrpKRIsotId{\func} \
    }
    \endaligned
\end{equation}
Recall that by Lemma~\ref{lm:p_H1M_H1G_surj} the graph $\KRGraphf$ of $\func$ is either a tree or contains a unique cycle.
In each of those cases we will describe this diagram up to isomorphism.

\begin{lemma}{\rm(\cite[Lemma~5.4]{KravchenkoFeshchenko:MFAT:2020})}\label{lm:cond_torus_KG_tree}
Let $\func\in\FSP{\Torus}{\Pman}$.
If $\func$ is either simple, or $\Pman=\Circle$ and $\func$ is not null-homotopic, then $\KRGraphf$ contains a cycle.
\end{lemma}

\newcommand\cLeaf[1]{\Cman_{#1}}
\newcommand\cLeaves{\mathcal{C}}
\newcommand\qCyl[1]{\Qman_{#1}}

\subsection{Maps $\func\in\FSP{\Torus}{\Pman}$ whose graph $\KRGraphf$ contains a cycle}
Suppose $\func:\Mman\to\Pman$ be a map from $\FSP{\Torus}{\Pman}$ such that its graph $\KRGraphf$ contains a simple cycle $\gamma$.
Let $\px$ be a point belonging to some open edge of $\gamma$.
Then $\px$ corresponds to some regular leaf $\cLeaf{0}$ of $\func$.
Let
\[
\cLeaves = \{ \dif(\cLeaf{0}) \mid \dif\in\StabilizerIsotId{\func} \}
\]
be the set of images of $\cLeaf{0}$ under all maps of $\StabilizerIsotId{\func}$.
Then $\cLeaves$ consists of finitely many, say $n$, leaves of $\func$ which can be cyclically ordered along $\Torus$ and enumerated as follows:
\[
\cLeaf{0}, \ \cLeaf{1}, \ \cdots, \ \cLeaf{n-1}.
\]

If $n\geq2$, then for each $i$ the leaves $\cLeaf{i}$ and $\cLeaf{i+1}$ bound a cylinder $\qCyl{i}$ containing no other leaves $\cLeaf{j}$ and $\Int{\qCyl{j}} \cap \Int{\qCyl{i}} = \varnothing$ for $i\not=j$.
Every $\dif\in\Stabilizer{\func}$ cyclically shift those cylinders.
Denote by $\Qman$ one of them, e.g. put $\Qman = \qCyl{0}$.

Suppose $n=1$, so $\cLeaf{0}$ is invariant with respect to $\StabilizerIsotId{\func}$, then $\Torus\setminus\cLeaf{0}$ is an open cylinder ``bounded by $\cLeaf{0}$ from both sides''.
In this case we put $\Qman = \overline{\Torus\setminus\regN{\cLeaf{0}}}$, where $\regN{\cLeaf{0}}$ is any $\func$-regular neighborhood of $\cLeaf{0}$

Now we will define several constructions and homomorphisms.
\begin{enumerate}[wide, label={\rm\alph*)}]
\item
Recall that $\Torus = \Circle\times\Circle$ is a abelian group, and therefore for each $(a,b) \in\Torus$, the shift $\eta_{a,b}:(\px,\py) \mapsto (\px+a, \py+b)$ is a diffeomorphism of $\Torus$.
Then the correspondence $(a,b) \mapsto \eta_{a,b}$ is an embedding $\eta: \Torus \monoArrow \DiffId(\Torus)$.
It is well known that $\eta$ is a homotopy equivalence, see Table~\ref{tbl:hom_type_DidMX}.

\item
Let $\cLeaf{0}'$ be a simple closed curve which transversely intersects $\cLeaf{0}$ at a unique point $\px$.
It will be convenient to imagine $\cLeaf{0}'$ and $\cLeaf{0}$ as a parallel and a meridian of $\Torus$.
Let $\lambda = [\cLeaf{0}'], \mu=[\cLeaf{0}] \in \pi_1(\Torus,\px)$ be the corresponding elements of the fundamental group of $\Torus$.
Then one can assume that
\[
    \lambda = (1,0), \mu = (0,1) \in \bZ^2 \cong \pi_1\Torus \xrightarrow[\cong]{\eta} \pi_1\DiffId(\Torus).
\]

\item
Let $\XSman_0$ and $\XSman_1$ be the boundary components of $\Qman$.
By Corollary~\ref{cor:stab:disk_ann:ann_decomp} one can represent $\Qman$ as a ``chain'' of $n$ cylinders $\Cyli{1},\ldots,\Cyli{n}$ having the following properties.
\begin{itemize}
\item
$\XSman_0 \subset \partial\Cyli{1}$, $\XSman_1 \subset \partial\Cyli{n}$, and only consecutive pairs $\Cyli{i}$ and $\Cyli{i+1}$ intersects and $\Cyli{i} \cap\Cyli{i+1}$ is their common boundary circle;
\item
$\seqStabIsotId{\func|_{\Qman}, \partial\Qman}\cong \prod\limits_{i=1}^{n}\seqStabIsotId{\func|_{\Cyli{i}}, \partial\Cyli{i}}$;
\end{itemize}

\item
Let $\eps:\partial\Qman\to\{\pm1\}$ be the function defining the sings on the boundary of the restriction $\func|_{\Qman}$, see~Section~\ref{sect:sing_on_bd}.
Thus $\eps$ takes constant values on each connected component $\Vman$ of $\partial\Qman$ and equals $+1$ (resp $-1$) if $\func$ takes on $\Vman$ a local maximum (resp.\ local minimum).
Then it is easy to see if $\Pman=\bR$, then $\eps$ takes the same value on both boundary components of $\partial\Qman$, whence by Theorem~\ref{th:realization:disk_cyl}\ref{enum:th:realization_classes:1}, $\pi_0\Stabilizer{\func|_{\Qman},\partial\Qman}$ is a non-trivial group.
On the other hand, for the case $\Pman=\Circle$, the function $\eps$ takes distinct values on boundary components of $\partial\Qman$, whence in this case $\pi_0\Stabilizer{\func|_{\Qman},\partial\Qman}$ can be trivial, which holds, e.g. when $\func$ has no critical points, see Theorem~\ref{th:sdo:except_cases}\ref{enum:specfunc:torus}.

\item
To simplify notations define the following groups $\mathcal{A}_{i} := \pi_0\StabilizerIsotId{\func|_{\Cyli{i}}, \partial\Cyli{i}}$,
\begin{align*}
    \Delta_{\partial} &:= \pi_0\FolStabilizerIsotId{\func|_{\Qman},\partial\Qman},  &
    \Delta            &:= \pi_0\FolStabilizerIsotId{\func|_{\Qman}}, \\
    \Stab_{\partial}  &:= \pi_0\StabilizerIsotId{\func|_{\Qman},\partial\Qman},  &
    \Stab             &:= \pi_0\StabilizerIsotId{\func|_{\Qman}}.
\end{align*}
Then the latter isomorphism for the sequence $\seqStabIsotId{\func|_{\Qman}, \partial\Qman}$ contains an isomorphism
\[
    \Stab_{\partial} \cong \mprod_{i=1}^{n}\mathcal{A}_{i}.
\]

\item
Let $\widehat{\lambda} = ([\id_{\Qman}], \ldots, [\id_{\Qman}], m) \in \Stab_{\partial}\wrm{m}\bZ$ be the Garside element of the group $\Stab_{\partial}\wrm{m}\bZ$, see \S\ref{sect:garside_elem}.
Then $\widehat{\lambda}$ contained in the center of $\Stab_{\partial}\wrm{m}\bZ$.

\item
Let $\widehat{\mu}_i$ be the Garside element of $\mathcal{A}_i$, and $\widehat{\mu}=(\widehat{\mu}_1\ldots,\widehat{\mu}_n)$ be the \myemph{diagonal Garside element} of $\Stab_{\partial}$, see~\eqref{equ:diag_Garside_seq}.
Notice $\widehat{\mu}$ is the generator of the kernel $\jInclZ: \Stab_{\partial} \to \Stab$, and corresponds to a pair of Dehn twists along boundary components of $\Qman$ produces in different directions.

\item
Recall that by Lemma~\ref{lm:incl_SprfX_Sprf}\ref{enum:xx:3x3_diagram} there exists a section
\[  \xi: \Delta \to \Delta_{\partial} \subset \Stab_{\partial}. \]

\item
Then one can define the following homomorphism:
\begin{gather*}
    \alpha: \pi_1\DiffId(\Torus) \times  \Delta^m \to \Stab_{\partial}\wrm{m}\bZ, \\
    \alpha(\lambda^a \mu^b, q_1,\ldots,q_n) =
    \bigl( \xi(q_1) \widehat{\mu}^b , \ldots, \xi(q_n)  \widehat{\mu}^b,  am \bigr).
\end{gather*}
Since $\alpha(\lambda) = \widehat{\lambda}$ and $\alpha(\mu) = (\widehat{\mu},\ldots,\widehat{\mu},0)$ belongs to the center of $\Stab_{\partial}\wrm{m}\bZ$, one easily checks that $\alpha$ is in fact a well-defined homomorphism.
Moreover, it is also easy to see that $\alpha$ is injective.
\end{enumerate}

\begin{subtheorem}\label{th:bib_seq_torus:tree}
{\rm(\cite{MaksymenkoFeshchenko:MFAT:2015, MaksymenkoFeshchenko:MS:2015, Feshchenko:PIGC:2019})}
Diagram~\eqref{equ:diagram_3x3_sdo_b:reduced} is isomorphic to the following one:
\begin{equation*}
\MResize{\textwidth}{
\xymatrix@R=4ex{
    \ \pi_1\Diff(\Torus) \oplus \bigl(\pi_0\FolStabilizerIsotId{\func|_{\Qman}}\bigr)^m \
            \ar@{->>}[rr]^-{\prj_2} \ar@{^(->}[rd]^-{\alpha} \ar@{->>}[d]_{\prj_1} & &
    \ \bigl(\pi_0\FolStabilizerIsotId{\func|_{\Qman_0}}\bigr)^m \ \ar@{^(->}[d]^{}  \\
    \ \pi_1\Diff(\Torus) \ar@{^(->}[r]^-{\alpha} \  &
    \ \pi_0\StabilizerIsotId{\func|_{\Qman},\partial\Qman} \wrm{m} \bZ \ \ar@{->>}[r]^-{\delta} \ar@{^(->}[rd] &
    \ \pi_0\StabilizerIsotId{\func|_{\Qman}} \wr \bZ_m \ \ar@{->>}[d] \\
    &
    &
    \ \GrpKRIsotId{\func|_{\Qman}} \wr \bZ_m \
}}
\end{equation*}
where $\delta(s_1,\ldots,s_m,k) = (\jInclZ(s_1),\ldots,\jInclZ(s_m), k \bmod m)$, and other arrows are obvious homomorphisms.

In particular, for the right columns of the corresponding diagrams we have the isomorphism, see~\eqref{equ:seqWrZm_finite}:
\[ \seqStabIsotId{\func} \cong \seqWrm{\seqStabIsotId{\func|_{\Qman}}}{m}.\]
\end{subtheorem}

\subsection{Maps $\func\in\FSP{\Torus}{\Pman}$ whose graph $\KRGraphf$ is a tree}
\begin{sublemma}{\rm(\cite[Proposition~1]{MaksymenkoFeshchenko:UMZ:ENG:2014})}\label{lm:exist_spec_level}
Suppose $\KRGraphf$ is a tree.
Then there exists a unique critical leaf $\crLev$ of $\func$ such that $\Torus\setminus\crLev$ is a union of open $2$-disks.
\end{sublemma}

It follows from uniqueness of such critical leaf, that $\dif(\crLev) = \crLev$ for all $\dif\in\Stabilizer{\func}$.

Let $c = \func(\crLev) \in \Pman$.
Take $\eps>0$ and consider the connected component $\regN{\crLev}$ of $\func^{-1}\bigl([c-\eps,c+\eps]\bigr)$ containing $\crLev$.
Decreasing $\eps$ one can assume that $\regN{\crLev} \cap \fSing = \crLev\cap\fSing$.
Then $\regN{\crLev}$ is an $\func$-regular neighborhood of $\crLev$, and it is also invariant with respect to $\Stabilizer{\func}$.

Similarly to~\S\ref{sect:disk_cyl:gen_case} let $\bZman$ be the collection of all connected components of $\overline{\Torus\setminus\regN{\crLev}}$.
Then by Lemma~\ref{lm:exist_spec_level} each element of $\bZman$ is a $2$-disk, and $\Stabilizer{\func}$ interchanges those disks.
Let also
\begin{equation}\label{equ:StabZ:torus}
 \Stabilizer{\bZman} = \{ \dif\in\StabilizerIsotId{\func} \mid \dif(\Zman)=\Zman \ \text{for each} \ \Zman\in\bZman\}
\end{equation}
be the kernel of non-effectiveness of the action of $\StabilizerIsotId{\func}$ on $\bZman$.
Then the quotient $\StabilizerIsotId{\func}/\Stabilizer{\bZman}$ \myemph{effectively} acts on $\bZman$.

\begin{subtheorem}\label{th:bib_seq_torus}
{\rm(\cite{MaksymenkoFeshchenko:UMZ:ENG:2014, Feshchenko:Zb:2015, Feshchenko:MFAT:2016, Feshchenko:PIGC:2019, KravchenkoFeshchenko:MFAT:2020})}
\begin{enumerate}[wide, label={\rm(\Alph*)}, topsep=0.8ex, itemsep=0.8ex]
\item
Suppose all elements of $\bZman$ are invariant with respect to $\StabilizerIsotId{\func}$.
To simplify notation put
\[
\MResize{\textwidth}{
\aligned
    \ \Delta       &:= \mprod\limits_{\Zman\in\bZman} \pi_0\FolStabilizerIsotId{\func|_{\Zman}, \partial\Zman}, &
    \ \Stab        &:= \mprod\limits_{\Zman\in\bZman} \pi_0\StabilizerIsotId{\func|_{\Zman}, \partial\Zman},    &
    \ \mathcal{G}  &:= \mprod\limits_{\Zman\in\bZman} \GrpKRIsotId{\func|_{\Zman}, \partial\Zman}.
\endaligned
}
\]
Then diagram~\eqref{equ:diagram_3x3_sdo_b:reduced} is isomorphic to the following one:
\begin{equation}\label{equ:diagram_3x3_sdo_b:reduced:torus:triv_act}
\aligned
        \xymatrix@R=4ex{
            \ \pi_1\Diff(\Torus) \oplus \Delta \
                    \ar@{->>}[rr]^-{\prj_2} \ar@{^(->}[rd] \ar@{->>}[d]_{\prj_1} & &
            \ \Delta \ \ar@{^(->}[d]^{}  \\
            \ \pi_1\Diff(\Torus) \ar@{^(->}[r] \  &
            \  \pi_1\Diff(\Torus) \oplus \Stab \ \ar@{->>}[r]^-{\delta} \ar@{^(->}[rd] &
            \ \Stab\ \ar@{->>}[d] \\
            &
            &
            \ \mathcal{G} \
        }
\endaligned
\end{equation}
In particular, for the right column we have an isomorphism
\[ \seqStabIsotId{\func} \cong \prod_{\Zman\in\bZman}\seqStabIsotId{\func|_{\Zman},\partial\Zman}.\]

\item
Otherwise, the following statements hold.
\begin{enumerate}[label={\rm\alph*)}, topsep=0.8ex, itemsep=0.8ex]
\item
$\StabilizerIsotId{\func} / \Stabilizer{\bZman} \cong \bZ_m \oplus \bZ_n$ for some $n,m\geq1$;
\item
The action of $\StabilizerIsotId{\func} / \Stabilizer{\bZman}$ on $\bZman$ is \myemph{free}.
For instance $\bZman$ contains $mnc$ disks, where $c\geq1$ is the number of orbits of that action.

\item
Choose any collection $\Zman_1,\ldots,\Zman_c$ of elements of $\bZman$ belonging to mutually distinct orbits and for simplicity denote
\[
 \quad \MResize{0.99\textwidth}{
\aligned
     \Delta       &:= \mprod\limits_{i=1}^{c}\pi_0\FolStabilizerIsotId{\func|_{\Zman_i}, \partial\Zman_i}, &
    \ \Stab        &:= \mprod\limits_{i=1}^{c}\pi_0\StabilizerIsotId{\func|_{\Zman_i}, \partial\Zman_i},    &
    \ \mathcal{G}  &:= \mprod\limits_{i=1}^{c}\GrpKRIsotId{\func|_{\Zman_i}, \partial\Zman_i}.
\endaligned}
\]
Then diagram~\eqref{equ:diagram_3x3_sdo_b:reduced} is isomorphic to the following one:
\begin{equation}\label{equ:diagram_3x3_sdo_b:reduced:torus:gen_case}
\aligned
        \xymatrix@R=3ex{
            \ \pi_1\Diff(\Torus) \oplus \Delta^{mn} \
                    \ar@{->>}[rr]^-{\prj_2} \ar@{^(->}[rd] \ar@{->>}[d]_{\prj_1} & &
            \ \Delta^{mn} \ \ar@{^(->}[d]^{}  \\
            \ \pi_1\Diff(\Torus) \ar@{^(->}[r]^{\alpha} \  &
            \ \Stab \wrm{m,n} \bZ^2 \ \ar@{->>}[r]^-{\delta} \ar@{^(->}[rd] &
            \ \Stab \wr (\bZ_{m}\times\bZ_{n})\ \ar@{->>}[d] \\
            &
            &
            \ \mathcal{G} \wr (\bZ_{m}\times\bZ_{n}) \
        }
\endaligned
\end{equation}
where $\alpha:\pi_1\Torus \to \Stab \wrm{m,n} \bZ$ is given by
\[ \alpha(\lambda^a \mu^b) = (e,\ldots,e, \,a \bmod m, \, b \bmod n) \] and $e$ is the unit of $\Stab$.
\end{enumerate}
In particular, for the right column we have an isomorphism, see~\eqref{equ:seqWrZmn_finite}:
\[ \seqStabIsotId{\func} \cong \seqWrmn{\Bigl( \prod_{\Zman\in\bZman}\seqStabIsotId{\func|_{\Zman},\partial\Zman} \Bigr)} {m}{n}.\]
\end{enumerate}
\end{subtheorem}

Evidently, \eqref{equ:diagram_3x3_sdo_b:reduced:torus:triv_act} is a particular case of~\eqref{equ:diagram_3x3_sdo_b:reduced:torus:gen_case} for $m=n=1$.
The last identity can also be proved using Lemma~\ref{lm:charact_seq_wrmn}.

We will now formulate the results about the structure of $\seqStabIsotId{\func}$ sequences for maps on $\Torus$ similar to
Theorems~\ref{th:solvable_groups} and~\ref{th:realization:disk_cyl}.

Define the following spaces of maps:
\begin{align*}
    \FSPTorusT{\Pman} &= \{ \func\in \FSP{\Torus}{\Pman}  \mid \text{$\KRGraphf$ is a tree} \}, \\[1mm]
    \FSPTorusO{\Pman} &= \{ \func\in \FSP{\Torus}{\Pman}  \mid \text{$\KRGraphf$ contains a cycle} \}, \\[1mm]
    \MrsTorusT{\Pman} &= \Mrs{\Torus}{\Pman} \cap \FSPTorusT{\Pman}, \\[1mm]
    \MrsTorusO{\Pman}  &= \Mrs{\Torus}{\Pman} \cap \FSPTorusO{\Pman}.
\end{align*}
Then Lemma~\ref{lm:cond_torus_KG_tree} implies that we have the following inclusions:
\begin{gather*}
    \MrsGen{\Torus}{\Pman} \ \subset \ \MrsSmp{\Torus}{\Pman} \ \subset \ \MrsTorusO{\Pman} \ \subset \ \FSPTorusO{\Pman}, \\
    \MrsTorusT{\Pman} \ \subset \ \FSPTorusT{\Pman}.
\end{gather*}
As a consequence of Theorems~\ref{th:realization:disk_cyl}, \ref{th:bib_seq_torus:tree}, \ref{th:bib_seq_torus} we get the following:
\begin{subtheorem}\label{th:classes_SG}{\rm(\cite{KravchenkoFeshchenko:MFAT:2020, KuznietsovaSoroka:UMJ:2021})}
The following identities hold:
\begin{align*}
    \CStab{\MrsGen{\Torus}{\Pman}}   &= \ccZ = \{ \bZ^n \mid n = 0,1,\ldots, \},  \\
    \CStab{\MrsSmp{\Torus}{\bR}}     &= \{ (A\times B) \wrm{m} \bZ \mid A,B\in\clsBt\setminus\{1\}, m\geq1 \}, \\
    \CStab{\MrsSmp{\Torus}{\Circle}} &= \{ A \wrm{m} \bZ \mid A\in\clsBt, m\geq1 \}, \\
    \CStab{\MrsTorusO{\bR}}        &= \CStab{\FSPTorusO{\bR}} =  \{ (A\times B) \wrm{m} \bZ \mid A,B\in\ccB\setminus\{1\}, m\geq1 \}, \\
    \CStab{\MrsTorusO{\Circle}}    &= \CStab{\FSPTorusO{\Circle}} = \ccB, \\ 
    \CStab{\MrsTorusT{\Pman}}      &= \CStab{\FSPTorusT{\Pman}} = \{ A \wrm{m,n} \bZ^2 \mid A\in\ccB, m,n\geq1 \}, \\
    \CGrp{\MrsGen{\Torus}{\Pman}}   &= \{1\},\\
    \CGrp{\MrsSmp{\Torus}{\Pman}}   &= \{ G\wr\bZ_m \mid \text{for some $G\in\clsGt$ and $m\geq1$} \}, \\
    \CGrp{\MrsTorusO{\Pman}}        & = \CGrp{\FSPTorusO{\Pman}}  = \ccP, \\
    \CGrp{\MrsTorusT{\Pman}}        & = \CGrp{\FSPTorusT{\Pman}} = \{ G\wr(\bZ_m\times\bZ_n) \mid G\in\ccP, m,n\geq1 \}, \\
\BSeq{\MrsGen{\Torus}{\Pman}}    &= \ZZI, \\
\BSeq{\MrsSmp{\Torus}{\bR}}      &= \{ \seqWrm{(\uSeq\times\vSeq)}{m} \mid \uSeq,\vSeq\in \ssZBtPt\setminus\{\seqTriv\}, m\geq1 \}, \\
\BSeq{\MrsSmp{\Torus}{\Circle}}  &= \{ \seqWrm{\uSeq}{m} \mid \uSeq\in\ssZBP, m\geq1 \}, \\
\BSeq{\MrsTorusO{\Pman}}         &= \BSeq{\FSPTorusO{\Pman}} = \\
        & \qquad = \{ \seqWrm{(\uSeq\times\vSeq)}{m} \mid \uSeq,\vSeq\in \ssZBP\setminus\{\seqTriv\}, m\geq1 \}, \\
\BSeq{\MrsTorusT{\Pman}}         & = \BSeq{\FSPTorusT{\Pman}} = \\
        & \qquad = \{ \seqWrmn{\uSeq}{m}{n} \mid \uSeq\in\ssZBP, m,n\geq1 \}.
\end{align*}
\end{subtheorem}

The proof for classes $\CGrp{\mathcal{X}}$ is given in~\cite{KravchenkoFeshchenko:MFAT:2020} and for $\CStab{\mathcal{X}}$ in~\cite{KuznietsovaSoroka:UMJ:2021}.
Those proofs imply realization theorems for $\BSeq{\mathcal{X}}$.

Partial computation of sequences $\seqStabIsotId{\func}$ were obtained in~\cite{KravchenkoMaksymenko:PIGC:2018, KravchenkoMaksymenko:JMFAG:2020} and in~\cite{MaksymenkoKuznietsova:PIGC:2019} for the cases when $\Mman$ is $2$-sphere and M\"obius band respectively.
They are not complete, and therefore we do not present them here.

\section{First homology groups of orbits}
Let $\Mman$ be a connected compact orientable surface distinct from $\Sphere$, and $\func\in\FSP{\Mman}{\Pman}$.
Let also
\begin{align*}
    \ccB' = \CStab{\FSPTorusT{\Pman}} = \{ A \wrm{m,n} \bZ^2 \mid A\in\ccB, m,n\geq1 \}.
\end{align*}
Then by Theorems~\ref{th:classes_SG} and~\ref{th:realization:disk_cyl} the fundamental group $G = \pi_1\OrbitPathComp{\func}{\func}$ belongs either to $\ccB$ or to $\ccB'$.
Moreover, as mentioned after definitions class $\ccB$ in Sections~\ref{sect:classes_M}, $G$ is obtained from the unit group by finitely many operations of direct products and wreath products of the form $\cdot\wrm{m}\bZ$, and (for the case when $\Mman=\Torus$ and $\KRGraphf$ is a tree) possibly \myemph{a unique and the last} operation of wreath product of the form $\cdot\wrm{a,b}\bZ^2$ for some $a,b\geq1$.
One can formalize this observation as follows.

\newcommand\AlphM{\mathcal{A}}
\newcommand\AlphT{\AlphM'}
\newcommand\AdmAlphM{\mathcal{W}}
\newcommand\AdmAlphT{\AdmAlphM'}

Consider the following two alphabets:
\begin{align*}
\AlphM &= \bigl\{ \ 1, \, \bZ, \, (, \, ), \, \times \ \bigr\} \ \cup \
          \bigl\{ \ \wrm{a}\bZ \ \bigr\}_{a\geq1} \\
\AlphT &= \AlphM \ \cup  \ \bigl\{ \ \wrm{a,b}\bZ^2 \ \bigr\}_{a,b\geq1},
\end{align*}
so $\AlphM$ it consists of the unit group $1$, group on integers $\bZ$, brackets ``('' and ``)'', a product sign ``$\times$'', and \myemph{wreath product with $\bZ$} symbols  ``$\wrm{a}\bZ$'' for all $a\geq1$, while $\AlphT$ additionally contains \myemph{wreath products with $\bZ^2$} symbols ``$\wrm{a,b}\bZ^2$'' for all $a,b\geq1$.

Then every group $G \in \ccB$ (resp.\ $\ccB'$) is written (though not in a unique way) as a word $w$ in the alphabet $\AlphM$ (resp.\ $\AlphT$).
For example, $\bZ^2$ can be written by the following words:
\begin{align*}
    &\bZ \times \bZ, &
    &\bZ \times (1\wrm{a} \bZ), &
    &1 \times (1\wrm{a} \bZ) \times (1\wrm{b} \bZ),
    & 1\wrm{a,b} (\bZ\times \bZ).
\end{align*}
for any $a,b\geq1$.
Of course, there are words which do not define a group, e.g. $)\bZ\wrm{4}$.


Given a group $G \in \ccB$ (resp.\ $\ccB'$) every word $w$ which correctly defines $G$ will be called a \myemph{realization of $G$ in the alphabet $\AlphM$ (resp.\ $\AlphT$)}.
Denote by $\beta_1(w)$ the number of symbols $\bZ$ in the word $w$.

The following result is obtained by I.~Kuznietsova and Yu.~Soroka as a consequence of Lemma~\ref{lm:center_comm}.
\begin{theorem}[\cite{KuznietsovaSoroka:UMJ:2021}]
Let $G \in \ccB$ (resp.\ $\ccB'$), and $w$ be any realization of $G$ in the alphabet $\AlphM$ (resp.\ $\AlphT$).
Then the center and the abelianization of $G$ are free abelian groups of the same rank $\beta_1(w)$:
\[
Z(G) \cong G / [G,G] \cong \bZ^{\beta_1(w)}.
\]
In particular, $\beta_1(w)$ does not depend on a concrete realization $w$ of $G$ in the alphabet $\AlphM$ (resp.\ $\AlphT$).
\end{theorem}
Recall that by Hurewicz theorem, e.g.~\cite[Theorem~2A.1]{Hatcher:AT:2002}, for any path connected topological space $X$ with the fundamental group $G = \pi_1 X$ we have an isomorphism $H_1(X,\bZ) \cong G /[G,G]$.
Hence we get the following
\begin{corollary}[\cite{KuznietsovaSoroka:UMJ:2021}]
Let $\Mman$ be a connected compact orientable surface distinct from $\Sphere$, $\func\in\FSP{\Mman}{\Pman}$, $G = \pi_1\OrbitPathComp{\func}{\func}$, and $\beta_1$ be the number of symbols $\bZ$ in any realization of $G$ in the alphabet $\AlphM$ (or $\AlphT$).
Then the first homology group $H_1(\OrbitPathComp{\func}{\func}, \bZ) \cong G /[G,G]$ is a free abelian group of rank $\beta_1$, that is $\beta_1$ is the first Betti number of the orbit $\OrbitPathComp{\func}{\func}$.
\end{corollary}


\def\cprime{$'$}

\end{document}

%% file: macro.tex
\newcommand{\theoremname}{Theorem}%
\newcommand{\lemmaname}{Lemma}%
\newcommand{\definitionname}{Definition}%
\newcommand{\examplename}{Example}%
\newcommand{\examplesname}{Examples}%
\newcommand{\problemname}{Problem}%
\newcommand{\propertyname}{Property}%
\newcommand{\conjecturename}{Conjecture}%
\newcommand{\assumptionname}{Assumption}%
\newcommand{\corollaryname}{Corollary}%
\newcommand{\propositionname}{Proposition}%
\newcommand{\claimname}{Claim}%
\newcommand{\remarkname}{Remark}%
\newcommand{\algorithname}{Algorithm}%
\newcommand{\questionname}{Question}%
\newcommand{\notationname}{Notation}%

\newtheorem{theorem}[subsection]{\protect\theoremname}
\newtheorem{lemma}[subsection]{\protect\lemmaname}

\newtheorem{corollary}[subsection]{\protect\corollaryname}

\newtheorem{definition}[subsection]{\protect\definitionname}

\newtheorem{subtheorem}[subsubsection]{\protect\theoremname}
\newtheorem{sublemma}[subsubsection]{\protect\lemmaname}

\newtheorem{subcorollary}[subsubsection]{\protect\corollaryname}

\newtheorem{subexample}[subsubsection]{\protect\examplename}

\newtheorem{subremark}[subsubsection]{\protect\remarkname}

\newtheorem{subdefinition}[subsubsection]{\protect\definitionname}

\newenvironment{smallproof}[1]
    {\footnotesize\begin{proof}[#1]}
    {\end{proof}}

\makeatletter
\newcommand\testshape{family=\f@family; series=\f@series; shape=\f@shape.}
\def\myemphInternal#1{\if n\f@shape%
	\begingroup\itshape #1\endgroup\/%
	\else\begingroup\sffamily #1\endgroup%
	\fi}
\def\myemph{\futurelet\testchar\MaybeOptArgmyemph}
\def\MaybeOptArgmyemph{\ifx[\testchar \let\next\OptArgmyemph
	\else \let\next\NoOptArgmyemph \fi \next}
\def\OptArgmyemph[#1]#2{\index{#1}\myemphInternal{#2}}
\def\NoOptArgmyemph#1{\myemphInternal{#1}}
\makeatother

\newcommand\bC{{\mathbb C}}

\newcommand\bN{{\mathbb N}}

\newcommand\bR{{\mathbb R}}

\newcommand\bZ{{\mathbb Z}}

\newcommand\FF{{\mathcal F}}

\newcommand\eps{\varepsilon}

\newcommand\id{\mathrm{id}}          
\newcommand\Int{\mathrm{Int}}        
\newcommand\ab{\mathrm{ab}}          
\newcommand\wrm[1]{\mathop{\wr}\limits_{#1}} 
\newcommand\nwr[3]{#1\wr_{#2}{#3}}   
\newcommand\ewr[2]{#1\wr{#2}}        

\newcommand\ptnum[1]{|#1|}
\newcommand\MOD[2]{#1\bmod#2}

\newcommand\AxCrPt{{\text{\rm(L)}}}
\newcommand\AxBd{{\text{\rm(B)}}}

\newcommand\Aman{A}
\newcommand\Bman{B}
\newcommand\Cman{C}
\newcommand\Kman{K}

\newcommand\Mman{M}
\newcommand\Nman{N}
\newcommand\Pman{P}
\newcommand\Qman{Q}

\newcommand\Uman{U}
\newcommand\Vman{V}
\newcommand\Wman{W}
\newcommand\Xman{X}
\newcommand\Yman{Y}
\newcommand\Zman{Z}

\newcommand\Circle{S^1}            
\newcommand\UInt{[0,1]}

\newcommand\Disk{D^2}              
\newcommand\Cylinder{A} 

\newcommand\Sphere{S^2}            
\newcommand\Torus{T^2}             

\newcommand\PrjPlane{\bR{P}^2}    
\newcommand\KleinBottle{K}       

\newcommand\Cyli[1]{\Cylinder_{#1}}

\newcommand\GL{\mathrm{GL}}

\newcommand\Orb{\mathcal{O}}        
\newcommand\Stab{\mathcal{S}}       
\newcommand\Diff{\mathcal{D}}       
\newcommand\Homeo{\mathcal{H}}      
\newcommand\Aut{\mathrm{Aut}}       

\newcommand\DiffId{\Diff_{\id}}     
\newcommand\StabId{\Stab_{\id}}     

\newcommand\Cinfty{\mathcal{C}^{\infty}}
\newcommand\Cr[1]{\mathcal{C}^{#1}}
\newcommand\Ci[2]{\mathcal{C}^{\infty}(#1,#2)}               
\newcommand\Cid[2]{\mathcal{C}_{\partial}^{\infty}(#1,#2)}   

\newcommand\func{f}
\newcommand\gfunc{g}
\newcommand\dif{h}

\newcommand\DiffIdM{\DiffId(\Mman)}

\newcommand\DiffMX{\Diff(\Mman, \Xman)}
\newcommand\DiffIdMX{\DiffId(\Mman, \Xman)}

\newcommand\Stabilizer[1]{\Stab(#1)}             
\newcommand\StabilizerPlus[1]{\Stab^{+}(#1)}     
\newcommand\StabilizerId[1]{\StabId(#1)}         
\newcommand\StabilizerIsotId[1]{\Stab'(#1)}      

\newcommand\Orbit[1]{\Orb(#1)}                   
\newcommand\OrbitPathComp[2]{\Orb_{#2}(#1)}      

\newcommand\SingularSet[1]{\Sigma_{#1}}             

\newcommand\AutKRGraphStab[1]{\mathbf{G}}             

\newcommand\fSing{\SingularSet{\func}}                

\newcommand\regN[1]{R_{#1}}
\newcommand\canN[1]{N_{#1}}

\newcommand\crLev{K}

\newcommand\regNK{R_{K}}

\newcommand\LStab{\mathcal{L}}
\newcommand\LStabPl{\LStab^{+}}

\newcommand\FolStab{\Delta} 
\newcommand\FolStabilizer[1]{\FolStab(#1)}

\newcommand\FolStabilizerIsotId[1]{\FolStab'(#1)}

\newcommand\GKR{\mathcal{G}}
\newcommand\GrpKR[1]{\GKR(#1)}
\newcommand\GrpKRIsotId[1]{\GKR'(#1)}

\newcommand\GH{\Gamma}

\newcommand\GHomIsotId[1]{\GH'(#1)}

\newcommand\hXman{{\Xman}^{\partial}}

\newcommand\monoArrow{\lhook\joinrel\rightarrow}

\newcommand\xmonoArrow[1]{\lhook\joinrel\xrightarrow{~#1~}}

\newcommand\epiArrow{\rightarrow\!\!\!\!\!\to}

\newcommand\xepiArrow[1]{\xrightarrow{#1}\!\!\!\!\!\to}

\newcommand\mprod{\mathop{\Pi}} 
\newcommand\myprod{\mathop{\prod}}

\def\cnt{k}

\newcommand\jIncl{j}
\newcommand\jInclZ{\jIncl_0}

\newcommand\KRGraphf{\Gamma_{\func}}
\newcommand\EKRGraphf{\widehat{\Gamma}_{\func}}
\newcommand\PF[1]{\widehat{#1}}

\newcommand\funcSeq{\mathbf{b}}
\newcommand\seqStab[1]{\funcSeq(#1)}
\newcommand\seqStabIsotId[1]{\funcSeq'(#1)}

\newcommand\DSG{{\funcSeq'}\!} 
\newcommand\sDSG[1]{\DSG\left(#1\right)}

\newcommand\seqZ[1]{\mathbf{z}_{#1}}

\newcommand\seqWrm[2]{#1\wr\seqZ{#2}}
\newcommand\seqWrmn[3]{#1\wr\seqZ{#2,#3}^2}

\newcommand\seqTriv{\seqZ{0}} 
\newcommand\aSeq{\mathbf{q}}
\newcommand\uSeq{\mathbf{u}}
\newcommand\vSeq{\mathbf{v}}

\newcommand\kSeq{\mathbf{k}}
\newcommand\lSeq{\mathbf{l}}

\newcommand\dclstwo[1]{#1_{2}}

\newcommand\ccZ{\mathcal{Z}}  %
\newcommand\ccB{\mathcal{B}}  %
\newcommand\ccP{\mathcal{P}}  %
\newcommand\ssZBP{\mathcal{ZBP}}  %

\newcommand\clsBt{\dclstwo{\ccB}}

\newcommand\clsGt{\dclstwo{\ccP}}

\newcommand\gssZBP{\widetilde{\ccZ\ccB\ccP}}

\newcommand\ssZBtPt  {\dclstwo{\ssZBP}}

\newcommand\FSPTorusT[1]{\mathcal{F}^{\Psi}(\Torus,#1)}
\newcommand\FSPTorusO[1]{\mathcal{F}^{O}(\Torus,#1)}
\newcommand\MrsTorusT[1]{\mathcal{M}^{\Psi}(\Torus,#1)}
\newcommand\MrsTorusO[1]{\mathcal{M}^{O}(\Torus,#1)}

\newcommand\lfrm{\Phi}
\newcommand\zfrm[1]{\lfrm(#1)}

\newcommand\bZman{\mathbf{Z}}
\newcommand\bZmanX{\bZman^{fix}}
\newcommand\bZmanY{\bZman^{reg}}

\newcommand\XFix{\Xman}
\newcommand\XFixA{\XFix_0}
\newcommand\XFixB{\XFix_1}

\newcommand\hb{\eta}

\newcommand\fc{\rho}

\newcommand\xA{A} 
\newcommand\xB{B} 
\newcommand\xC{C} 

\newcommand\kA{K} 
\newcommand\kB{L} 
\newcommand\kC{M} 

\newcommand\zA[1]{K_{#1}}
\newcommand\zB[1]{L_{#1}}
\newcommand\zC[1]{M_{#1}}

\newcommand\xv[1]{b_{#1}}

\newcommand\OrbfX{\Orbit{\func,\Xman}}        

\newcommand\OrbffX{\OrbitPathComp{\func,\Xman}{\func}}   

\newcommand\FSP[2]{\mathcal{F}(#1,#2)}
\newcommand\Mrs[2]{\mathcal{M}(#1,#2)}
\newcommand\MrsSmp[2]{\mathcal{M}^{smp}(#1,#2)}
\newcommand\MrsGen[2]{\mathcal{M}^{gen}(#1,#2)}

\newcommand\FSPE[3]{\mathcal{F}_{#1}(#2,#3)}
\newcommand\MrsE[3]{\mathcal{M}_{#1}(#2,#3)}
\newcommand\MrsSmpE[3]{\mathcal{M}^{smp}_{#1}(#2,#3)}
\newcommand\MrsGenE[3]{\mathcal{M}^{gen}_{#1}(#2,#3)}

\newcommand\CStab[1]{\mathbf{S}_{#1}}
\newcommand\CGrp[1]{\mathbf{G}_{#1}}
\newcommand\BSeq[1]{\mathbf{\bm{\Delta}SG}_{#1}}

\newcommand\CStabFE[3]       {\CStab{\FSPE{#1}{#2}{#3}}}
\newcommand\CStabMrsE[3]     {\CStab{\MrsE{#1}{#2}{#3}}}    
\newcommand\CStabMrsSmpE[3]  {\CStab{\MrsSmpE{#1}{#2}{#3}}} 
\newcommand\CStabMrsGen[3]   {\CStab{\MrsGenE{#1}{#2}{#3}}} 

\newcommand\CGrpKRFE[3]       {\CGrp{\FSPE{#1}{#2}{#3}}}
\newcommand\CGrpKRMrsE[3]     {\CGrp{\MrsE{#1}{#2}{#3}}}    
\newcommand\CGrpKRMrsSmpE[3]  {\CGrp{\MrsSmpE{#1}{#2}{#3}}} 
\newcommand\CGrpKRMrsGen[3]   {\CGrp{\MrsGenE{#1}{#2}{#3}}} 

\newcommand\BSeqFE[3]       {\BSeq{\FSPE{#1}{#2}{#3}}}
\newcommand\BSeqMrsE[3]     {\BSeq{\MrsE{#1}{#2}{#3}}}    
\newcommand\BSeqMrsSmpE[3]  {\BSeq{\MrsSmpE{#1}{#2}{#3}}} 
\newcommand\BSeqMrsGen[3]   {\BSeq{\MrsGenE{#1}{#2}{#3}}} 

\newcommand\ZZI{\mathcal{ZZI}}

\newcommand\torus[1]{\mathbb{T}^{#1}}
\newcommand\vn{n}

\newcommand\pz{z}
\newcommand\px{x}
\newcommand\pt{t}
\newcommand\ps{s}
\newcommand\py{y}
\newcommand\prj{p}

\newcommand\skl[2]{#1^{(#2)}}

\newcommand\homeq{\simeq}
\newcommand\isom{\cong}

\newcommand\sdo{\mathbf{O}}
\newcommand\sdoseq[1]{\sdo\left(#1\right)}
\newcommand\pnt{\bullet}
\newcommand\whomeq{\stackrel{\text{\sf\tiny{w}}}{\simeq}}

\newcommand\XSman{\Vman} 

\newcommand\dtw[1]{\tau_{#1}}
\newcommand\idtw[2]{[\dtw{#1}]_{#2}}

\newcommand\pj{\mathsf{p}}

\newcommand*{\Resize}[2]{\resizebox{#1}{!}{$#2$}}
\newcommand*{\MResize}[2]{\text{\Resize{#1}{#2}}}